\documentclass[11pt,letterpaper]{amsart}

\usepackage{graphicx}
\usepackage{epstopdf}
\usepackage{subfigure}

\usepackage{amsmath,amsfonts,amsthm,mathrsfs,amssymb}  

\usepackage{array} 

\usepackage{cite}
\usepackage{hyperref}
\hypersetup{hypertex=true,colorlinks=true,linkcolor=blue,anchorcolor=blue,citecolor=blue}

\usepackage[usenames]{color}
\usepackage{xcolor}

\newtheorem{thm}{Theorem}[section]

\newtheorem{lem}{Lemma}[section]
\newtheorem{prop}{Proposition}[section]
\theoremstyle{definition}
\newtheorem{defn}{Definition}[section]
\newtheorem{asm}{Assumption}
\theoremstyle{remark}
\newtheorem{rem}{Remark}[section]

\numberwithin{equation}{section}

\theoremstyle{remark}

\usepackage{tikz}
\usetikzlibrary{arrows.meta,positioning}

\allowdisplaybreaks
\title[Stably determining an impedance obstacle]{Stably Determining a generalised Impedance Obstacle from a Single Far-Field Pattern}

\author{Huaian Diao}
\address{School of Mathematics, Jilin University, Changchun 130012, China}
\email{diao@jlu.edu.cn, hadiao@gmail.com}

\author{Hongyu Liu}
\address{Department of Mathematics, City University of Hong Kong, Kowloon, Hong Kong SAR, China}
\email{hongyu.liuip@gmail.com, hongyliu@cityu.edu.hk}

\author{Longyue Tao}
\address{Department of Mathematics, City University of Hong Kong, Kowloon, Hong Kong SAR, China}
\email{sdyctly@163.com, longyue.tao@my.cityu.edu.hk}

\begin{document}

\begin{abstract}

Inverse scattering focuses on recovering unknown scatterers from wave measurements. A fundamental challenge is determining whether an inverse obstacle problem can be resolved from a single far-field measurement, a task particularly demanding for non-convex polytope obstacles under generalized impedance boundary conditions and closely linked to the long-standing Schiffer problem. 

In this paper, we develop a novel \emph{Artificial Test Domain} (ATD) framework for single-measurement inverse scattering of impenetrable polytope obstacles. Based on microlocal analysis near exterior-visible flat boundary patches, this approach transcends traditional methods reliant on observable corners. The ATD framework establishes two primary conceptual advancements: a unified \emph{generalized impedance hyperplane (GIH) exclusion mechanism}, which clarifies the structural role of uniqueness mechanisms, and a unified \emph{qualitative--quantitative principle} for the generalized impedance setting. 

Quantitatively, the method yields a \emph{far-field--geometry relation} where geometric discrepancy is controlled by far-field error, scaled by a leading ATD coefficient. Qualitatively, the non-vanishing of this coefficient reduces to the exclusion of exterior generalized impedance hyperplanes. Once uniqueness is established, this relation produces sharp stability estimates. Within this framework, the classical stability estimates for the sound-soft and sound-hard cases are recovered as special instances of a much more general stability theory. At the same time, we obtain several new sharp stability results that are of significant importance. These results unify currently available single-measurement uniqueness regimes for polytope geometry and provide new insights into the Schiffer problem across multiple generalized impedance settings.



\end{abstract}

\maketitle
\section{Introduction}

In this paper, we study the stable determination of impenetrable polytope obstacles from a single far-field pattern in a generalized impedance setting.
The problem is particularly challenging under such limited data, especially for non-convex geometries and for boundary regimes beyond the classical sound-soft and sound-hard cases, and is closely connected with the long-standing Schiffer problem in inverse scattering (cf. \cite{lax1990scattering,colton2019inverse,colton2018looking,diao2023spectral}).
Our aim is to quantify this single-measurement determination problem and thereby establish, in this general setting, a unified relation between unique identifiability and sharp stability.

\subsection{The direct scattering problem}\label{subsec Direct}


We first introduce the direct scattering problem for an impenetrable obstacle $D \subset \mathbb{R}^n,\ n=2,3$ in the generalized impedance setting, which will be defined in Definition~\ref{defn:gen-imp}. We assume that $D$ consists of finitely many pairwise disjoint solid components, namely,
$$
D = \bigcup_{\ell=1}^N D_\ell, \quad N \in \mathbb{N}, \quad D_{\ell_1} \cap D_{\ell_2} = \emptyset \quad \text{for all } \ell_1 \ne \ell_2 \in \{1,\ldots,N\},
$$
where each $D_\ell$ is a simply connected bounded Lipschitz domain in $\mathbb{R}^n$, for $\ell = 1,\ldots,N$. We assume that the exterior domain of $D$,
\begin{equation}\label{eq exterior}
\mathbf{G} := \mathbb{R}^n \setminus \overline{D}, \quad n=2,3
\end{equation}
is connected.


We take the incident field to be a plane wave with wavenumber $k=\frac{\omega}{c_0}\in\mathbb{R}_+$ and incident direction $\mathbf{p}\in\mathbb{S}^{n-1}$, namely
\begin{equation}\label{eq incident}
u^i(\mathbf{x}) = e^{\mathrm{i}k\mathbf{x}\cdot\mathbf{p}}.
\end{equation}
Other types of incident fields can also be considered in this paper, for example, a point source of the form
\begin{align}
\label{eq:point source}
u^i(\mathbf{x}; \mathbf{z}_0)
=
\frac{\mathrm{i}}{4}
\left(\frac{k}{2\pi |\mathbf{x}-\mathbf{z}_0|}\right)^{\frac{n-2}{2}}
H^{(1)}_{\frac{n-2}{2}}(k|\mathbf{x}-\mathbf{z}_0|),
\end{align}
associated with the wavenumber $k\in\mathbb{R}_+$. Here $H_\mu^{(1)}$ denotes the Hankel function of the first kind of order $\mu$, and $\mathbf{z}_0\in\mathbb{R}^n\setminus\overline{D}$ denotes the location of the point source.




Let the incident field $u^i$ impinge on $D$, thereby generating the scattered field $u^s$ in $\mathbf{G}$. The total field $u := u^i + u^s$ belongs to $H^1_{\mathrm{loc}}(\mathbf{G})$ and satisfies the following direct scattering problem:
\begin{equation}\label{eq main system}
\begin{cases}
\Delta u+k^2u=0 & \mbox{in }\mathbf{G},\\
\partial_\nu u+\eta u=0 & \mbox{on }\partial D,\\
\displaystyle{\lim_{r\to\infty}r^{\frac{n-1}{2}}\left(\frac{\partial u^s}{\partial r}-\mathrm{i}ku^s\right)=0} & \mbox{with }r=|\mathbf{x}|.
\end{cases}
\end{equation}
Here, $\nu$ denotes the exterior unit normal to $D$ on $\partial D$. To describe the boundary conditions considered in this paper in a unified manner, we introduce the following terminology.
\begin{defn}\label{defn:gen-imp}
The boundary condition
$$
\partial_\nu u + \eta u = 0
\quad \text{on } \partial D
$$
is called a \emph{generalized impedance boundary condition} if it is understood in one of the following senses:
\begin{enumerate}
\item
$\eta \in L^\infty(\partial D)$, possibly complex-valued, in which case it is an impedance boundary condition;

\item
$\eta \equiv 0$ on $\partial D$, in which case it reduces to the sound-hard boundary condition;

\item
$\eta \equiv +\infty$ on $\partial D$, understood formally, in which case it reduces to the sound-soft boundary condition $u=0$ on $\partial D$;

\item
with respect to a Lipschitz partition of $\partial D$, $\eta$ may be of type \textnormal{(1)}, \textnormal{(2)}, or \textnormal{(3)} on each component of the partition, in which case the boundary condition is said to be of mixed type.
\end{enumerate}
\end{defn}

Thus, the present definition includes the classical  sound-soft, sound-hard, impedance, and mixed obstacles considered in the literature (cf.\cite{colton2019inverse,diao2023spectral}). For the purposes of this paper, we treat all these classes within the unified framework of generalized impedance obstacles.






The Sommerfeld radiation condition yields the following far-field expansion for the scattered field \(u^s\) (cf. \cite{colton2019inverse,mclean2000strongly}):
\begin{equation}\label{eq SRC}
u^s(\mathbf{x})=\frac{e^{\mathrm{i}kr}}{r^{\frac{n-1}{2}}}\left\{u_\infty(\hat{\mathbf{x}})+\mathcal O\left(\frac{1}{r}\right)\right\},
\quad r=\vert\mathbf{x}\vert\to\infty,
\end{equation}
uniformly for $\hat{\mathbf{x}}:=\mathbf{x}/\vert\mathbf{x}\vert\in\mathbb{S}^{n-1}$.
The far-field pattern $u_\infty$ contains information about both the geometry of $D$ and the impedance parameter $\eta$.

\subsection{The inverse scattering problem}

The inverse scattering problem concerns the recovery of the shape and location of an unknown obstacle \(D\), together with the impedance parameter \(\eta\), from far-field measurements.
This problem has been extensively studied in the literature; see, e.g., \cite{colton2019inverse,colton2013integral,isakov2017inverse,diao2023spectral}.
The corresponding inverse problem can be formulated as
\begin{equation}\label{eq IP}
\mathfrak{F}(D,\eta)=u_\infty(\hat{\mathbf x};D,\eta,u^i),
\qquad \hat{\mathbf x}\in\mathbb S^{n-1},
\end{equation}
where \(u^i\) denotes the incident wave introduced above.
It is well known that the inverse problem \eqref{eq IP} is nonlinear and severely ill-posed.

When the incident wave $u^i$ is fixed, the corresponding far-field pattern $u_\infty(\hat{\mathbf{x}};D,\eta,u^i)$ is called a \emph{single far-field measurement}. More precisely, this means that, in the case of a plane incident wave given by \eqref{eq incident}, both the incident direction $\mathbf{p}\in\mathbb{S}^{n-1}$ and the wavenumber $k$ are fixed, while, in the case of a point source given by \eqref{eq:point source}, both the source point $\mathbf{z}_0$ and the wavenumber $k$ are fixed. If any of the parameters $k$, $\mathbf{p}$, or $\mathbf{z}_0$ is allowed to vary, then one obtains multiple far-field measurements. In the single-measurement setting, two basic issues arise naturally: \emph{uniqueness} and \emph{stability}. These correspond, respectively, to the \emph{qualitative} and \emph{quantitative} aspects of the inverse problem \eqref{eq IP}.

\smallskip\noindent
\emph{Uniqueness.}
The uniqueness result establishes a one-to-one correspondence between an obstacle, together with its surface impedance parameter, and its associated far-field pattern.
More precisely, for a fixed incident field $u^i$, the mapping $(D,\eta)\mapsto u_\infty(\hat{\mathbf{x}};D,\eta,u^i)$ is injective.
This means that if two admissible pairs $(D_1,\eta_1)$ and $(D_2,\eta_2)$ generate the same far-field pattern for all observation directions $\hat{\mathbf{x}}\in\mathbb{S}^{n-1}$, then the two pairs must coincide:
\begin{equation}\label{eq uniqueness form}
\begin{aligned} u_\infty(\hat{\mathbf{x}};D_1,\eta_1,u^i)=u_\infty(\hat{\mathbf{x}};D_2,\eta_2,u^i),
\ \forall\,\hat{\mathbf{x}}\in\mathbb{S}^{n-1} 
\ \Longleftrightarrow \quad (D_1,\eta_1)=(D_2,\eta_2).
\end{aligned}
\end{equation}
Thus, the far-field pattern uniquely determines both the shape $D$ of the obstacle and the impedance function $\eta$ on its boundary.

This qualitative aspect is closely connected with the long-standing Schiffer problem in inverse scattering; see \cite{colton2018looking,colton2019inverse} for further discussion.
In fact, from a cardinality point of view, a single far-field pattern is defined on $\mathbb{S}^{n-1}$ and therefore carries $(n-1)$-cardinality observation data, which matches the cardinality scale of the geometric information needed to characterize the shape and position of the obstacle.
In this sense, the geometric inverse problem \eqref{eq uniqueness form}, namely, the determination of the shape and position of the obstacle from a single far-field measurement, is formally determined.
This makes the problem mathematically optimal and at the same time highly challenging in inverse scattering.


\smallskip\noindent
\emph{Stability.}
The corresponding stability issue addresses the quantitative aspect of the uniqueness result.
It asks whether, for the same incident wave $u^i$, one can estimate a geometric distance between two obstacles in terms of the discrepancy between their far-field patterns.
In other words, we seek an estimate of the form
\begin{equation}\label{eq stability form}
{\sf GM}(D_1,D_2)\leq \psi\!\left(\|u_\infty(\hat{\mathbf x};D_1,\eta_1,u^i)-u_\infty(\hat{\mathbf x};D_2,\eta_2,u^i)\|_{L^2(\mathbb S^{n-1})}\right),
\end{equation}
where ${\sf GM}(D_1,D_2)$ denotes a suitable geometric distance between the two obstacles $D_1$ and $D_2$ (for instance, the Hausdorff distance or the Lebesgue measure of their symmetric difference).
The function $\psi$ is a modulus of continuity, assumed to be nondecreasing and to satisfy
\begin{equation}\label{eq modulus condition}
\lim_{\mu\to0^+}\psi(\mu)=0.
\end{equation}
Condition \eqref{eq modulus condition} ensures that, as the far-field patterns become arbitrarily close in the $L^2$-norm, the geometric distance between the corresponding obstacles also tends to zero.
Thus, a stability estimate of this type provides continuous dependence of the obstacle, together with its impedance, on the far-field data, and hence complements the uniqueness statement with quantitative control.

\smallskip\noindent
\emph{Qualitative--quantitative principle.}
In this sense, \eqref{eq stability form} together with \eqref{eq modulus condition} indicates an intrinsic relation between the quantitative and qualitative aspects of the single-measurement inverse problem: uniqueness in the sense of \eqref{eq uniqueness form} provides the qualitative basis for stability, while stability may be understood as a quantitative expression of uniqueness.
Accordingly, the study of this relation can also be viewed as a quantitative response to the Schiffer problem.
The present work is devoted to clarifying this relation in the generalized impedance setting, which we refer to as the \emph{qualitative--quantitative principle}.
Further discussion of this viewpoint is given in the next subsection.

\subsection{Background, challenges, and a new framework for the Schiffer problem}

A central question behind single-measurement inverse scattering is whether highly limited wave data can still encode sufficient geometric information for the recovery of an unknown obstacle.
At a broad conceptual level, this is reminiscent of the \emph{holographic principle}, in the sense that lower-dimensional data may encode higher-dimensional structure, although the mathematical setting here is entirely different from that of holography in quantum gravity~\cite{Susskind1995,Maldacena1998,Bousso2002}.
In the present context, a single far-field pattern is precisely such a minimal data set, while the obstacle geometry is the unknown structure to be recovered.
This viewpoint leads naturally to the Schiffer problem in time-harmonic inverse scattering, namely, whether the shape of an unknown obstacle can be determined from a single far-field measurement~\cite{colton2018looking,colton2019inverse}.

When full far-field data are available at a fixed frequency, uniqueness results have been established for sound-soft, sound-hard, and impedance obstacles; see \cite{lax1990scattering,colton1983uniqueness,colton2019inverse} and the references therein.
Under suitable a priori assumptions, limited far-field measurements may already suffice for unique recovery.
For sound-soft or sound-hard polytope obstacles, uniqueness under limited measurements has been proved by combining geometric arguments with the reflection principle for the Helmholtz equation~\cite{alessandrini2005determining,cheng2003uniqueness,elschner2006uniqueness,liu2006uniqueness,hu2014unique}.
More recently, uniqueness results for polygonal and polyhedral obstacles have also been established without prior knowledge of the boundary condition, based on geometric properties of Laplacian eigenfunctions~\cite{cao2020nodal,cao2021novel}.
Furthermore, a convex polytope impedance obstacle can be uniquely determined by at most two far-field measurements~\cite{cao2022two}.
For mixed or partially coated polyhedral scatterers, single-measurement uniqueness results were obtained in \cite{liu2007unique}.
In two dimensions, single-measurement uniqueness for polygonal obstacles with constant impedance was also obtained in \cite{hu2020uniqueness}.

In inverse scattering, stability is the quantitative expression of uniqueness, aiming to estimate the Hausdorff-type distance between two obstacles in terms of the discrepancy between their far-field data.
Such estimates are typically of logarithmic type, reflecting the severe ill-posedness of the inverse problem, and this behavior is in general essential optimal even in the presence of multiple measurements~(see \cite{dicristo2003examples} for more details).
For polyhedral sound-soft obstacles, sharp and essential optimal logarithmic stability estimates from a single far-field measurement were established in \cite{rondi2008stable}.
For sound-hard polyhedral scatterers, related stability results based on a minimal number of scattering measurements were later obtained in \cite{liu2017stable}.
For purely constant impedance polygonal obstacles, sharp single-measurement stability estimates were obtained in \cite{diao2024stable} under an a priori convexity assumption, including the simultaneous recovery of constant impedance parameters.
By contrast, the reflection principle and path argument underlying the sound-soft and sound-hard stability results do not apply to purely impedance scatterers; instead, the analysis in \cite{diao2024stable} relies essentially on the corner singularity in the impedance case.
Consequently, none of the above results directly covers the single-measurement stability problem for generalized impedance obstacles of polytope geometry in the non-convex and variable-impedance setting considered here.
The obstruction comes from both the non-convex geometry and the presence of variable impedance.

In the non-convex setting, the relevant geometric discrepancy may occur along an exterior-visible flat portion without producing any useful corner contribution, so arguments based on observable corners are no longer sufficient.
Variable impedance creates an additional difficulty, since the direct scattering problem must remain uniformly well posed under admissible perturbations of both the obstacle and the impedance parameter.
These two difficulties require different but compatible ingredients in the analysis.

To overcome these difficulties, we develop a unified \emph{Artificial Test Domain} (ATD) approach in the generalized impedance setting for inverse obstacle problems of polytope geometry, carried out here for polygonal and polyhedral obstacles.
To capture the relevant geometric discrepancy, we work with a modified distance adapted to exterior visibility, which is quantitatively equivalent to the Hausdorff distance on the admissible class; see Subsection~\ref{subsec DRO}.
The ATD construction is then localized near a single exterior-visible edge in two dimensions or a single exterior-visible face in three dimensions; see Sections~\ref{sec 2D} and \ref{sec proof}.
We then use a quantitative propagation-of-smallness argument to transfer the far-field error from infinity to the corresponding local region; see Section~\ref{sec 4}.
For variable-impedance boundary conditions, we further establish a uniform direct scattering framework based on sharp a priori estimates and Mosco convergence; see Subsection~\ref{subsec math}, Appendix~\ref{app:mosco}, and also \cite{liu2019mosco,menegatti2013stability}.

These ingredients lead to the main quantitative result of the present approach, namely Theorem~\ref{th relation}, a unified \emph{far-field--geometry relation} in the generalized impedance setting.
More precisely, the ATD analysis yields a quantitative estimate in which the geometric discrepancy between two obstacles is controlled by the far-field error up to a \emph{leading ATD coefficient}, namely, the first nontrivial coefficient extracted from the local ATD identity near a chosen exterior-visible flat portion.
This coefficient is not merely a pointwise local quantity, but is determined by the local boundary behavior of the total field, the underlying boundary condition, and the actual geometric mismatch between the two scatterers along the chosen exterior-visible flat boundary portion.

The next issue is therefore structural rather than quantitative: one needs to understand what the vanishing of the leading ATD coefficient actually implies.
This is carried out in Section~\ref{sec:nondeg}, where we show that if the leading ATD coefficient vanished, then the total field would necessarily satisfy a Dirichlet, Neumann, or Robin relation on an exterior flat portion, or equivalently, the exterior domain would contain the corresponding exterior hyperplane.
Accordingly, the non-vanishing needed for stability is reduced to the exclusion of the corresponding exterior hyperplane.
This leads to a further conceptual point of the present work, namely, that the relevant single-measurement uniqueness mechanisms can be organized through a unified \emph{generalized impedance hyperplane (GIH) exclusion mechanism}.
Thus, Theorem~\ref{thm:qualitative-hyperplane-exclusion} is not merely a reformulation of existing uniqueness results, but identifies the structural mechanism through which the far-field--geometry relation closes into the sharp stability estimate. From this viewpoint, the existing uniqueness results for polytope obstacles \cite{liu2006uniqueness,liu2007unique,hu2014unique,hu2020uniqueness}, although proved by variants of the reflection principle and path argument, can be interpreted in a unified way as GIH exclusion mechanisms. In particular, the appearance of the corresponding Dirichlet, Neumann, or Robin hyperplane is the first step in carrying out this reflection-based mechanism for solving the Schiffer problem in inverse scattering by polytope obstacles. In this sense, the present paper clarifies the common underlying structure of these arguments by introducing the GIH exclusion mechanism.

The hyperplane-exclusion viewpoint also clarifies the role of singularity in the corresponding uniqueness results; see \cite{cao2020nodal,cao2021novel,cao2022two,honda2013analytic}.
Beyond the reflection-based uniqueness results in \cite{liu2006uniqueness,liu2007unique,hu2014unique,hu2020uniqueness}, other approaches to the Schiffer problem may likewise be understood, in some sense, through singularity analysis \cite{cao2020nodal,cao2021novel,cao2022two,honda2013analytic}.
In the existing uniqueness results for polytope obstacles, the exclusion of the relevant exterior hyperplanes is typically supported by a geometric singularity, for instance a corner singularity arising from the structure of Laplacian eigenfunctions \cite{cao2020nodal,cao2021novel,cao2022two}.
By contrast, in the nowhere-analytic impedance regime considered here, the required singularity is provided by the impedance coefficient itself, which leads to the exclusion of the corresponding exterior Robin hyperplane; see Theorem~\ref{thm:Xi2-excludes-degeneracy}.
This should be compared with the geometric non-analyticity mechanism in \cite{honda2013analytic}.
In particular, Theorem~\ref{thm:Xi2-excludes-degeneracy} establishes single-measurement uniqueness for purely impedance polygonal and polyhedral obstacles whose impedance parameters are nowhere analytic, and thus solves the corresponding Schiffer problem in this regime.
From this viewpoint, the essential issue behind the GIH mechanism is the presence of a suitable geometric or parameter singularity.
We expect that this perspective can also be useful for extending the GIH exclusion mechanism, and hence the corresponding stability theory, to more general geometric settings.

This hyperplane-exclusion viewpoint is precisely what connects uniqueness with stability.
Once the corresponding uniqueness mechanism is valid, the leading ATD coefficient is forced to be nonzero, and Theorem~\ref{th relation} immediately yields  the sharp stability estimate of Theorem~\ref{cor stability}.
In this way, the qualitative and quantitative aspects of the single-measurement inverse problem are brought into a unified framework.
On the qualitative side, the relevant uniqueness results are organized through the generalized impedance hyperplane exclusion mechanism.
On the quantitative side, Theorem~\ref{th relation} provides the unified far-field--geometry relation.
Accordingly, once the corresponding uniqueness mechanism is established, the far-field--geometry relation immediately incudes a   sharp stability.
This is the \emph{qualitative--quantitative principle} described above: in the generalized impedance setting, it provides a unified mechanism through which single-measurement uniqueness gives rise to sharp stability.

The main results are presented in two parts.
The first part establishes a unified quantitative relation between the far-field discrepancy and the geometric distance in the generalized impedance setting.
The second part establishes sharp stability in all regimes where the corresponding uniqueness mechanism is available.
In particular, within the admissible geometric class proposed in this paper, this provides a unified derivation of the sharp stability results for the sound-soft and sound-hard cases \cite{rondi2008stable,liu2017stable}, and further yields new sharp stability results for the nowhere-analytic impedance regime, the mixed sound-soft/finite-impedance regime, and the two-dimensional constant-impedance regime.


\subsection*{Organization of the paper}

Section~\ref{sec main results} presents the main quantitative results, including Theorem~\ref{th relation} on the unified far-field--geometry relation and Theorem~\ref{cor stability} on the corresponding sharp stability estimates.

Section~\ref{sec pre} introduces the notation and the admissible classes.
Subsection~\ref{subsec DRO} presents the geometric setup and the relevant metrics, and proves the quantitative equivalence between the exterior-visible distance and the Hausdorff distance on the admissible class.

Section~\ref{sec 4} constructs exterior visibility paths and propagates far-field smallness to neighborhoods of exterior-visible boundary patches.
Section~\ref{sec 2D} develops the two-dimensional ATD extraction and the corresponding relation analysis.
Section~\ref{sec proof} develops the three-dimensional ATD geometry near an exterior-visible face and proves the corresponding three-dimensional relation.

Section~\ref{sec:nondeg} establishes the regime-dependent non-degeneracy conditions and completes the proof of Theorem~\ref{cor stability}.
Appendix~\ref{app:classical-relation} supplements the relation analysis for the classical sound-soft and sound-hard regimes.
Appendix~\ref{app:mosco} collects the Mosco convergence framework and the uniform estimates used in the argument.

\section{Main results}\label{sec main results}

We state the main results that clarify, in precise mathematical terms, the relation between uniqueness and stability in the single-measurement inverse obstacle problem under generalized impedance boundary conditions.
Two main viewpoints emerge from this relation.
The first is that the relevant single-measurement uniqueness mechanisms can be organized through a unified generalized impedance hyperplane exclusion mechanism.
The second is the qualitative--quantitative principle, namely, that once the corresponding uniqueness mechanism is available, the far-field--geometry relation closes into sharp stability.
\subsection{The far-field--geometry relation}\label{subsec part1}

We first state the quantitative part of the qualitative--quantitative principle in the single-measurement setting.
Throughout this paper, the incident plane wave $u^i$ in \eqref{eq incident} is fixed, with fixed incident direction $\mathbf{p}\in\mathbb{S}^{n-1}$ and fixed wavenumber $k\in\mathbb{R}_+$.
The following theorem establishes the unified quantitative relation between the far-field error and the geometric discrepancy.

\begin{thm}\label{th relation}
Let $(D,\eta)$ and $(D',\eta')$ be two scattering configurations in the generalized impedance setting of Definition~\ref{defn:gen-imp}.
Assume that $D$ and $D'$ are admissible polygonal or polyhedral obstacles in the sense of Definition~\ref{defn admissible class}.
Let $u_\infty$ and $u'_\infty$ be the far-field patterns generated by $(D,\eta)$ and $(D',\eta')$, respectively, under the same incident plane wave $u^i$.
Assume that $n\in\{2,3\}$ and that
\[
\|u_\infty-u'_\infty\|_{L^2(\mathbb S^{n-1})}\leq \varepsilon
\]
for some $0<\varepsilon\leq \varepsilon_m<e^{-1}$, where $\varepsilon_m$ is a sufficiently small constant depending only on the a priori parameters.
Then there exist constants $p\geq 1$, $\kappa_0>0$, and $C>0$, depending only on the a priori parameters, such that for any corresponding ATD construction with leading ATD coefficient $C_A$, one has
\begin{equation}\label{eq relation main}
|C_A|\,(\mathrm d_H(D,D'))^p \leq C\,(\log|\log(1/\varepsilon)|)^{-\kappa_0}.
\end{equation}
\end{thm}

\begin{rem}\label{rem:CA}
We prove Theorem~\ref{th relation} by the ATD method, using the modified distance introduced in Subsection~\ref{subsec DRO}, which is quantitatively equivalent to the Hausdorff distance on the admissible class.
We select an exterior-visible boundary point $\mathbf{x}_0\in \partial D\setminus D'$ at which the modified distance is attained, and attach a local test domain near $\mathbf{x}_0$ inside $D\setminus D'$.
Accordingly, the local analysis is carried out for $u'$ rather than $u$.
Since this test region lies in the exterior of $D'$, the field $u'$ is analytic there, and its local expansion, combined with CGO solutions, yields \eqref{eq relation main}.
The coefficient $C_A$ in \eqref{eq relation main} is the leading ATD coefficient associated with the chosen exterior-visible flat portion near $\mathbf{x}_0$, namely, the first nontrivial coefficient extracted from the corresponding local ATD identity after inserting the local expansion of $u'$; see \eqref{eq 2Dexpansion}, \eqref{eq u expansion}, and Propositions~\ref{prop lower2D} and \ref{prop low2}.
Thus $C_A$ is determined by the local boundary behavior of the total field, the underlying boundary condition, and the actual geometric discrepancy near the chosen exterior-visible flat portion.
\end{rem}

\begin{rem}\label{rem:relation-framework}
Theorem~\ref{th relation} provides the quantitative part of the qualitative--quantitative principle.
It is established by a unified ATD argument in the generalized impedance setting and does not rely on any regime-dependent uniqueness input.
For clarity of presentation, the detailed derivation in Sections~\ref{sec 2D} and \ref{sec proof} is carried out in the bounded-impedance setting, while the remaining boundary regimes are treated in Appendix~\ref{app:classical-relation} through brief supplementary proofs.
The role of uniqueness enters only later, through the verification that the leading ATD coefficient does not vanish.
Accordingly, once the relevant uniqueness mechanism is available to exclude the corresponding exterior hyperplane, Theorem~\ref{th relation} closes into the sharp stability estimate.
\end{rem}

\medskip
\subsection{GIH exclusion mechanism and sharp stability}\label{subsec part2}

We now turn to the qualitative side of the argument.
Theorem~\ref{th relation} yields a quantitative estimate of the form
\[
|C_A|\,(\mathrm d_H(D,D'))^p \leq C\,(\log|\log(1/\varepsilon)|)^{-\kappa_0}.
\]
This already establishes the far-field--geometry relation.
To derive sharp stability, it remains to show that the leading ATD coefficient $C_A$ is nonzero.

This is precisely the issue addressed in Section~\ref{sec:nondeg}.
There we show that if $C_A=0$, then the total field must satisfy a Dirichlet, Neumann, or Robin relation on an exterior flat portion, or equivalently, the exterior domain must contain the corresponding exterior hyperplane.
Accordingly, the non-vanishing of $C_A$ is reduced to the exclusion of the corresponding exterior hyperplane.
This is the point where the corresponding uniqueness mechanism enters.
We now introduce the relevant exterior hyperplanes and the admissible regimes in which the corresponding GIH exclusion mechanism is available.

\begin{defn}\label{defn:ext-hyper}
Let $u$ denote the total field associated with the scattering problem \eqref{eq main system} in the exterior domain $\mathbf G$ defined by \eqref{eq exterior}.
\begin{enumerate}
\item
A \emph{Dirichlet hyperplane in $\mathbf G$} is an affine hyperplane $\Pi\subset \mathbf G$ such that
\[
u=0 \qquad \text{on }\Pi.
\]

\item
A \emph{Neumann hyperplane in $\mathbf G$} is an affine hyperplane $\Pi\subset \mathbf G$ such that
\[
\partial_\nu u=0 \qquad \text{on }\Pi,
\]
where $\nu$ denotes a unit normal to $\Pi$.

\item
A \emph{Robin hyperplane in $\mathbf G$} is an affine hyperplane $\Pi\subset \mathbf G$ such that
\[
\partial_\nu u+\eta u=0 \qquad \text{on }\Pi,
\]
where $\nu$ denotes a unit normal to $\Pi$.
\end{enumerate}
\end{defn}

The following admissible classes of generalized impedance parameters are used to formulate the qualitative hyperplane-exclusion principle and the corresponding sharp stability theorem.

\begin{defn}\label{defn eta}
Within the generalized impedance setting of Definition~\ref{defn:gen-imp}, we call the generalized impedance parameter $\eta$ \emph{admissible} if it belongs to one of the following classes.
\begin{itemize}
\item
The class $\Xi_1$ is the bounded-impedance class.
One has $\eta\in L^\infty(\partial D)$ and $\|\eta\|_{L^\infty(\partial D)}\leq M_0$ for some constant $M_0>0$, $\Im \eta\geq 0$ almost everywhere on $\partial D$, and $\eta\not\equiv 0$.
This class includes constant impedances.

\item
The class $\Xi_2$ is the nowhere-analytic impedance class.
One has $\eta\in\Xi_1$, and $\eta$ is nowhere real-analytic on $\partial D$ in the sense that for every $\mathbf{x}_0\in\partial D$ there is no open neighborhood of $\mathbf{x}_0$ on which $\eta$ coincides with the trace of a real-analytic function in $\mathbb R^n$.

\item
The class $\Xi_3$ is the mixed sound-soft/finite-impedance class.
One assumes that the boundary admits a Lipschitz dissection
\[
\partial D=\Gamma_D\cup \Gamma_I\cup \Gamma_0,
\]
where $\Gamma_D$ and $\Gamma_I$ are disjoint relatively open subsets of $\partial D$, and $\Gamma_0$ denotes their common boundary on $\partial D$.
The generalized impedance parameter satisfies
\[
\eta=+\infty \quad \text{on }\Gamma_D,
\qquad
\eta\in C(\Gamma_I), \quad \eta_0\leq \eta\leq M_0 \quad \text{on }\Gamma_I,
\]
where $M_0$ and $\eta_0$ are positive constants.

\item
The class $\Xi_4$ is the constant positive impedance class.
One has
\[
\eta\equiv \lambda>0 \qquad \text{on }\partial D
\]
for some positive constant $\lambda$.
\end{itemize}
\end{defn}

\begin{rem}\label{rem:eta-classes}
The generalized impedance setting covers the classical  sound-soft, sound-hard, and impedance cases.
Among the admissible classes above, $\Xi_1$ is the general bounded-impedance class and is mainly used to formulate the direct scattering framework for variable impedance, based on sharp a priori estimates and Mosco convergence; see Appendix~\ref{app:mosco} for further details.

By contrast, the classes $\Xi_2$, $\Xi_3$, and $\Xi_4$ are introduced because in these regimes the corresponding uniqueness mechanism needed to exclude the relevant exterior hyperplanes becomes available.
More precisely, they are the regimes in which the non-vanishing of the leading ATD coefficient can be verified, either by the analysis in the present paper or by the existing uniqueness literature.
For the general bounded-impedance class $\Xi_1$, the single-measurement uniqueness problem remains open.
Thus the role of the additional classes is to identify those regimes in which the corresponding GIH exclusion mechanism is already available, so that the far-field--geometry relation of Theorem~\ref{th relation} closes into the sharp stability estimate within the same framework.
For the class $\Xi_4$, the qualitative exclusion and the resulting stability conclusion used in this paper are available only in two dimensions.
\end{rem}

\medskip
The corresponding qualitative mechanism established in Theorem \ref{thm:qualitative-hyperplane-exclusion} will be referred to as the \emph{generalized impedance hyperplane-exclusion mechanism}.
It consists in excluding the exterior Dirichlet, Neumann, or Robin hyperplanes relevant to the boundary regime under consideration.
As will be shown later in Section~\ref{sec:nondeg}, this exclusion is precisely what guarantees the non-vanishing of the leading ATD coefficient and hence allows the far-field--geometry relation to close into sharp stability.

\medskip
We summarize the relevant single-measurement uniqueness results in the following qualitative form: the total field cannot admit the corresponding exterior Dirichlet, Neumann, or Robin hyperplane.
This qualitative hyperplane-exclusion principle is exactly the qualitative mechanism needed to close the far-field--geometry relation of Theorem~\ref{th relation} into the sharp stability estimate.

\begin{thm}\label{thm:qualitative-hyperplane-exclusion}
Let $n=2,3$, and let $D$ be a polygonal or polyhedral obstacle with connected exterior $\mathbf G$.
Assume that $(D,\eta)$ belongs to one of the following regimes.

\begin{enumerate}
\item[(i)]
If $D$ is sound-soft, then $\mathbf G$ contains no Dirichlet hyperplanes.

\item[(ii)]
If $D$ is sound-hard, then $\mathbf G$ contains no Neumann hyperplanes.

\item[(iii)]
If $\eta\in\Xi_2$, namely in the nowhere-analytic impedance regime, then $\mathbf G$ contains no Robin hyperplanes.

\item[(iv)]
If $\eta\in\Xi_3$, namely in the mixed sound-soft/finite-impedance regime, then $\mathbf G$ contains neither Dirichlet hyperplanes nor Robin hyperplanes.

\item[(v)]
If $n=2$ and $\eta\in\Xi_4$, namely in the two-dimensional constant-impedance regime, then $\mathbf G$ contains no Robin hyperplanes.
\end{enumerate}
\end{thm}

\begin{rem}\label{rem:qualitative-sources}
The uniqueness results in \cite{liu2006uniqueness} yield the qualitative exclusions in cases (i) and (ii).
The qualitative exclusion and the corresponding uniqueness result in case (iii) are established later in Theorem~\ref{thm:Xi2-excludes-degeneracy}.
The uniqueness results in \cite{liu2007unique} and \cite{hu2020uniqueness} yield the qualitative exclusions in cases (iv) and (v), respectively.
For the sound-soft case, we also note that \cite{hu2014unique} proved single-measurement uniqueness for polyhedral scatterers with a single incident point source wave.

Thus, Theorem~\ref{thm:qualitative-hyperplane-exclusion} does not merely collect the relevant uniqueness results, but for the first time organizes them into a unified GIH exclusion mechanism.
The point is not simply that the corresponding uniqueness results are available, but that the present work identifies and clarifies the common structural mechanism underlying them.
More importantly, the present paper is the first to apply this unified GIH exclusion mechanism to the stability analysis in the generalized impedance setting.
In particular, it shows that this mechanism is precisely the qualitative factor needed to force the non-vanishing of the leading ATD coefficient and hence to close the far-field--geometry relation into the sharp stability estimate.
\end{rem}

\medskip
We can now state the sharp stability theorem.
It shows that, in the present framework, once the corresponding GIH exclusion mechanism is available, the far-field--geometry relation of Theorem~\ref{th relation} yields the sharp stability estimate.

\begin{thm}\label{cor stability}
Assume the setting of Theorem~\ref{th relation}.
Assume in addition that $(D,\eta)$ belongs to one of the following regimes:
\begin{enumerate}
\item[(i)]
the sound-soft regime;

\item[(ii)]
the sound-hard regime;

\item[(iii)]
the nowhere-analytic impedance regime;

\item[(iv)]
the mixed sound-soft/finite-impedance regime;

\item[(v)]
the two-dimensional constant-impedance regime.
\end{enumerate}
Then, in each of these regimes, Theorem~\ref{thm:qualitative-hyperplane-exclusion}, combined with the non-vanishing analysis in Section~\ref{sec:nondeg}, implies that the leading ATD coefficient is nonzero.
Consequently, the relation estimate \eqref{eq relation main} yields the sharp stability estimate
\begin{equation}\label{eq stability}
\mathrm d_H(D,D')\leq \mathbf{C}\,(\log|\log(1/\varepsilon)|)^{-\kappa},
\end{equation}
where the constants $\mathbf{C}>0$ and $\kappa>0$ depend only on the a priori parameters of the regime under consideration.
\end{thm}

\begin{rem}\label{rem:open-general-impedance}
To the best of our knowledge, the single-measurement uniqueness problem remains open for general bounded-impedance obstacles in the class $\Xi_1$ of Definition~\ref{defn eta}, allowing variable coefficients and non-convex polygonal or polyhedral geometries.
This shows that, within the present framework, the only missing ingredient for the sharp stability estimate in the full bounded-impedance regime is the corresponding hyperplane exclusion mechanism.

Indeed, Theorem~\ref{th relation} already provides the far-field--geometry relation in the full generalized impedance setting.
Accordingly, once the relevant uniqueness mechanism is available so as to exclude the corresponding exterior hyperplane and guarantee the non-vanishing of the leading ATD coefficient, the same relation yields the sharp stability estimate.
In this sense, the present work clarifies the relation between uniqueness and stability in the generalized impedance setting through the qualitative--quantitative principle.

For the sound-soft and sound-hard regimes, related stability results were previously established in \cite{rondi2008stable,liu2017stable} by different methods based on the reflection principle and path arguments.
Here these regimes are recovered within the same hyperplane-exclusion and ATD framework, while the stability results for cases (iii)--(v) are new.
\end{rem}

Theorem~\ref{th relation} is proved later by combining the two-dimensional ATD analysis in Section~\ref{sec 2D}, the three-dimensional ATD analysis in Section~\ref{sec proof}, and the supplementary treatment of the sound-soft and sound-hard cases in Appendix~\ref{app:classical-relation}.
Theorem~\ref{cor stability} is then obtained by combining Theorem~\ref{th relation}, Theorem~\ref{thm:qualitative-hyperplane-exclusion}, and the non-vanishing mechanism established in Section~\ref{sec:nondeg}.
\section{The direct scattering problem}\label{sec pre}
We begin this section by introducing  notation and conventions used throughout the paper. 

\subsection{Notations}
The integer $n = 2, 3$ denotes the spatial dimension.
Throughout the paper, points in $\mathbb{R}^n$ are denoted by $\mathbf{x}=(x_1,\ldots,x_n)$, and similarly for other variables such as $\mathbf{y}$ and $\mathbf{z}$.
When $n=2$, this means $\mathbf{x}=(x_1,x_2)$, and when $n=3$, $\mathbf{x}=(x_1,x_2,x_3)$.

For $r>0$ and $\mathbf{z}\in\mathbb{R}^n$ we write $B_r(\mathbf{z})$ for the \textit{open ball} in $\mathbb{R}^n$ of radius $r$ centered at $\mathbf{z}$.
When $n=2$, the set $B_r(\mathbf{z})$ is the \textit{open disk}.
We also set $B_r := B_r(0)$.

Let $\Pi$ be a plane in $\mathbb{R}^3$ and, after a rigid change of coordinates, let $\Pi =\{ \mathbf{x} \in \mathbb{R}^3 : x_3 = 0 \}$.
For $\mathbf{x} \in \Pi$ and $ r > 0 $ we denote by
\[
B_r^+(\mathbf{x}) = B_r(\mathbf{x}) \cap \{ x_3 > 0 \},
\qquad
B_r^-(\mathbf{x}) = B_r(\mathbf{x}) \cap \{ x_3 < 0 \}
\]
the upper and lower half-balls, respectively.

Let $\Sigma \subset \mathbb{R}^3$ be a bounded \textit{polyhedral obstacle}.
Its boundary decomposes as $\partial\Sigma=\bigcup_{m=1}^p\Pi_m$ with $p \in \mathbb{N}$ and $p \geq 4$, where each $\Pi_m$ is a planar polygonal face.
The set of edges is $\mathcal{E}(\Sigma)=\{\ell_1, \ldots, \ell_q \}$, where each $ \ell_j $ is a line segment of the form $ \Pi_m \cap \Pi_{m'} $ for some $m \neq m'$.
Given two adjacent faces $\Pi_1$ and $\Pi_2$ with common edge $\ell$, we denote by $ \mathcal{D}(\Pi_1, \Pi_2; \ell) $ the dihedral wedge and by $ \alpha_{12} \in (0, 2\pi) $ its opening angle.
Let $\mathbf{x}_0 $ be a vertex of $\Sigma$ with incident faces $\{\Pi_j\}_{j=1}^r$ and $r \geq 3$.
We fix a cyclic ordering so that $\ell_j = \Pi_j \cap \Pi_{j+1}$ with the convention $\Pi_{r+1} = \Pi_1$.
We then set $\mathcal{D}_j:= \mathcal{D}(\Pi_j, \Pi_{j+1}; \ell_j)$ for $ j = 1, \ldots, r$ and denote this vertex together with its incident dihedral wedges by $\mathcal{V}(\{ \Pi_j \}_{j=1}^r, \mathbf{x}_0)$.

Let $K\subset\mathbb{R}^2$ be a \textit{polygonal obstacle} consisting of a single polygon.
Its boundary is $\partial K=\bigcup_{j=1}^q\ell_j$ with $q\ge3$, where each $\ell_j$ is a line segment.

Fix $\delta>0$ and $\theta\in(0,\pi/2]$.
For $\mathbf{x}\in\mathbb{R}^n$ and $\boldsymbol\omega\in\mathbb{S}^{n-1}$ with $\|\boldsymbol\omega\|=1$, define the \textit{open cone}
\[
\mathcal{C}(\mathbf{x},\boldsymbol\omega,\delta,\theta)
=\Bigl\{\mathbf{y}\in\mathbb{R}^n:0<\|\mathbf{y}-\mathbf{x}\|<\delta,\ \frac{\mathbf{y}-\mathbf{x}}{\|\mathbf{y}-\mathbf{x}\|}\cdot\boldsymbol\omega>\cos\theta\Bigr\}.
\]
It consists of points whose direction from $\mathbf{x}$ makes angle $<\theta$ with $\boldsymbol\omega$ and whose distance to $\mathbf{x}$ is $<\delta$.
\subsection{The classes of admissible obstacles}\label{subsec admissible}
We now introduce the admissible classes of obstacles used in this paper.
These consist of the three-dimensional class \(\mathcal A\) of admissible polyhedral obstacles and the two-dimensional class \(\mathcal B\) of admissible polygonal obstacles.
The corresponding a~priori assumptions are grouped below according to geometry, boundary regularity, and exterior accessibility.

\begin{asm}\label{asm 1}
Fix $r_0>0$, $R_0>0$, and $\theta_0\in(0,\pi/2)$.
Let $D\subset\mathbb{R}^n$ with $n=2,3$, polygonal when $n=2$ and polyhedral when $n=3$.
\begin{enumerate}
  \item Edge size.
  Every edge $\ell_j$ of $\partial D$ satisfies $|\ell_j|\ge r_0$.
  \item Two-dimensional vertex angles.
  If $n=2$, every interior angle $\beta$ of $\partial D$ satisfies $\beta\in(\theta_0,\pi-\theta_0)\cup(\pi+\theta_0,2\pi-\theta_0)$.
  \item Three-dimensional face angles.
  If $n=3$ and $\Pi_i$ is a face of $\partial D$, every planar interior angle $\beta$ of $\Pi_i$ at its vertices satisfies $\beta\in(\theta_0,\pi-\theta_0)\cup(\pi+\theta_0,2\pi-\theta_0)$.
  \item Three-dimensional dihedral angles.
  If $n=3$, every dihedral angle $\alpha_{\ell_j}$ along each edge $\ell_j$ satisfies $\alpha_{\ell_j}\in(\theta_0,\pi-\theta_0)\cup(\pi+\theta_0,2\pi-\theta_0)$.
  \item Three-dimensional face thickness.
  If $n=3$ and $\Pi_i$ is a face of $\partial D$, there exists $\mathbf{x}_i\in\Pi_i$ with $\operatorname{dist}(\mathbf{x}_i,\partial\Pi_i)\ge r_0$.
  \item Boundedness.
  $D\subset B_{R_0}$.
\end{enumerate}
\end{asm}

\begin{asm}\label{asm 2}
Fix $L_0>0$ and $r_0>0$.
We say that $\partial D$ is \textit{Lipschitz} with constants $(L_0,r_0)$ if for every $\mathbf{x}\in\partial D$ there exists a rigid motion $\mathcal R$ of $\mathbb R^n$ with $\mathcal R(\mathbf{x})=0$ and a function $\varphi$ defined on $\{y\in\mathbb R^{n-1}: |y|<r_0\}$ such that
\[
|\varphi(y_1)-\varphi(y_2)|\leq L_0 |y_1-y_2|
\qquad\text{for all } y_1,y_2 \text{ with } |y_1|,|y_2|<r_0,
\]
and, writing $\mathcal R(\mathbf{z})=(y,t)$ with $y\in\mathbb R^{n-1}$ and $t\in\mathbb R$,
\[
\mathcal R(\partial D)\cap B_{r_0}=\{(y,\varphi(y)):\ |y|<r_0\},
\qquad
\mathcal R(D)\cap B_{r_0}=\{(y,t)\in B_{r_0}:\ t<\varphi(y)\}.
\]
\end{asm}

\begin{asm}\label{asm 3}
Fix $r_0>0$ and $\theta_0\in(0,\pi/2)$.
If $n=3$, for every $\mathbf{x}\in\partial D$ there exists $\boldsymbol\omega_{\mathbf{x}}\in\mathbb S^2$ such that for all $\mathbf{y}\in\partial D\cap B_{r_0}(\mathbf{x})$ one has
\[
\mathcal{C}(\mathbf{y},\boldsymbol\omega_{\mathbf{x}},r_0,\theta_0)\subset\mathbb R^3\setminus D.
\]
If $n=2$, for every $\mathbf{x}\in\partial D$ there exists $\boldsymbol\omega_{\mathbf{x}}\in\mathbb S^1$ such that for all $\mathbf{y}\in\partial D\cap B_{r_0}(\mathbf{x})$ one has
\[
\mathcal{C}(\mathbf{y},\boldsymbol\omega_{\mathbf{x}},r_0,\theta_0)\subset\mathbb R^2\setminus D.
\]
\end{asm}

With these assumptions in hand, we now define the admissible classes \(\mathcal A\) and \(\mathcal B\).

\begin{defn}\label{defn admissible class}
The admissible classes \(\mathcal A\) and \(\mathcal B\) are defined as follows.
\begin{itemize}
\item A polyhedral obstacle \(\Sigma\subset\mathbb R^3\), consisting of finitely many pairwise disjoint solid polyhedral components, belongs to the class \(\mathcal A\) with a~priori parameters \((R_0,r_0,L_0,\theta_0)\) if and only if Assumptions~\ref{asm 1}, \ref{asm 2}, and \ref{asm 3} hold.
\item A polygonal obstacle \(K\subset\mathbb R^2\), consisting of finitely many pairwise disjoint solid polygonal components, belongs to the class \(\mathcal B\) with a~priori parameters \((R_0,r_0,L_0,\theta_0)\) if and only if Assumptions~\ref{asm 2} and \ref{asm 3} hold, and Assumption~\ref{asm 1} holds in its two-dimensional form, namely items~(1), (2), and (6).
\end{itemize}
An obstacle is called \textit{admissible} if it belongs to \(\mathcal A\) or \(\mathcal B\).
\end{defn}

\subsection{Mathematical formulation}\label{subsec math}

This subsection collects several auxiliary ingredients used throughout the paper.
The first part concerns the well-posedness of the exterior impedance problem and a uniform $L^2$ bound.
The second part gives a local decomposition for analytic solutions near a point.
For completeness, we state the relevant results here, while the corresponding definitions and proofs are deferred to Appendix~\ref{app:mosco}.

Throughout the paper, we set
\[
\Omega:=B_{R_0}(0),
\]
where $R_0$ is the a~priori bound in Assumption~\ref{asm 1}.
The following well-posedness result is classical for exterior impedance problems with $\Im\eta\ge 0$; see \cite{colton2019inverse, kirsch2008factorization}.

\begin{lem}\label{lem wellposs}
Let $D$ be an admissible obstacle in the sense of Definition~\ref{defn admissible class}.
Let $\eta$ be an impedance parameter as in Definition~\ref{defn eta} with $\Im\eta\ge 0$.
Then there exists a unique radiating solution $u\in H^1_{\mathrm{loc}}(\mathbb{R}^n\setminus D)$ to \eqref{eq main system} satisfying \eqref{eq SRC}.
In particular,
\[
u\in H^1(\Omega\setminus D).
\]
\end{lem}

\begin{lem}\cite{ramm1996existence}\label{lem RCP}
Let $D$ be an admissible obstacle.
Then the following compactness properties hold.
\begin{itemize}
\item
The embedding $H^1(\Omega\setminus D)\hookrightarrow L^2(\Omega\setminus D)$ is compact.
\item
The trace operator $\sigma:H^1(\Omega\setminus D)\to H^{1/2}(\partial D)$ is bounded, and the embedding $H^{1/2}(\partial D)\hookrightarrow L^2(\partial D)$ is compact.
\end{itemize}
\end{lem}

We next recall a compactness statement for the admissible classes and then state a convergence result for solutions under Mosco and Hausdorff convergence of the domains.

\begin{lem}\cite{liu2017stable}\label{lem compact}
The admissible classes $\mathcal{A}$ and $\mathcal{B}$ are compact with respect to convergence in the Hausdorff distance.
\end{lem}

\begin{prop}\label{prop Mosco}
Let $\{D_n\}$ and $D$ be admissible obstacles in the sense of Definition~\ref{defn admissible class}.
Let $\{\eta_n\}\subset\Xi_1$ and $\eta\in\Xi_1$ with $\eta_n\to\eta$ in the sense of Definition~\ref{def eta moving}.
Then the following assertions hold.
\begin{enumerate}
\item
If $H^1(\Omega\setminus D_n)$ converges to $H^1(\Omega\setminus D)$ in the sense of Mosco, then the corresponding solutions $u_n$ to \eqref{eq main system} converge to $u$ in $L^2(\Omega)$ and $\nabla u_n\to\nabla u$ in $L^2(\Omega;\mathbb R^n)$, with the usual extensions by zero.
\item
If $D_n\to D$ in the sense of Hausdorff distance in $\overline{\Omega}$, then $H^1(\Omega\setminus D_n)$ converges to $H^1(\Omega\setminus D)$ in the sense of Mosco.
\end{enumerate}
\end{prop}

As a consequence of the above compactness properties and standard energy estimates, we also have the following uniform bound.

\begin{lem}\label{lem uniform bound}
Let $D\in\mathcal A\cup\mathcal B$ be an admissible obstacle, and let $u\in H^1(\Omega\setminus D)$ be the solution to \eqref{eq main system}.
Then there exists a constant $\mathcal E>0$, depending only on the a~priori parameters and on $k$, such that
\begin{equation}\label{eq uniform bounded}
\|u\|_{L^2(\Omega\setminus D)}\leq \mathcal E.
\end{equation}
\end{lem}

The proofs of Proposition~\ref{prop Mosco} and Lemma~\ref{lem uniform bound}, together with the definitions of Mosco convergence used here, are given in Appendix~\ref{app:mosco}.

\medskip
In the second part, we begin with the classical unique continuation principle for second-order elliptic equations; see \cite{jerison1985unique}.
This principle is used to define the vanishing order and to derive a local decomposition in both two and three dimensions.

\begin{lem}\label{lem vanish}
If $u\in H^1(\Omega\setminus D)$ solves \eqref{eq main system} and, for some $\mathbf{x}_0\in\Omega\setminus D$, satisfies
\[
\lim_{r\to 0} r^{-N}\,\frac{1}{|B_r(\mathbf{x}_0)|}\int_{B_r(\mathbf{x}_0)} |u|\,\mathrm d \mathbf{x}=0
\]
for every $N\in\{0\}\cup\mathbb N$, then $u\equiv 0$ in $\Omega\setminus D$.
\end{lem}

For a nontrivial solution $u$, the vanishing order at $\mathbf{x}_0\in\Omega\setminus D$ is defined by
\begin{equation}\label{eq vanishing order}
\mathrm{Vani}(u;\mathbf{x}_0):=\max\left\{N\in\mathbb N\cup\{0\}:\ \lim_{r\to 0} r^{-N}\,\frac{1}{|B_r(\mathbf{x}_0)|}\int_{B_r(\mathbf{x}_0)} |u|\,\mathrm d \mathbf{x}=0\right\}.
\end{equation}
Equivalently, for every integer $m\leq \mathrm{Vani}(u;\mathbf{x}_0)$ one has
\[
\lim_{r\to 0} r^{-m}\,\frac{1}{|B_r(\mathbf{x}_0)|}\int_{B_r(\mathbf{x}_0)} |u|\,\mathrm d \mathbf{x}=0,
\]
whereas the same limit fails for $m=\mathrm{Vani}(u;\mathbf{x}_0)+1$.
By Lemma~\ref{lem vanish}, any nontrivial radiating solution to \eqref{eq main system} has finite vanishing order at each $\mathbf{x}\in\Omega\setminus D$.

Let $\mathbf{x}_0\in\Omega\setminus \overline D$.
Since $u$ solves $\Delta u+k^2 u=0$ in a neighborhood of $\mathbf{x}_0$, it is real-analytic there.
After a rigid motion, we may assume $\mathbf{x}_0=0$, set $N:=\mathrm{Vani}(u;0)$, and fix $\rho_0\in(0,R_0)$ such that $B_{\rho_0}\subset\Omega\setminus \overline D$.

\begin{prop}\label{prop decom}
If $u$ is analytic in $B_{\rho_0}$ and $\mathrm{Vani}(u;0)=N$, then there exist functions $u_N$ and $\delta u_{N+1}$ such that
\begin{equation}\label{eq decomposition}
u=u_N+\delta u_{N+1}\qquad\text{in }B_{\rho_0},
\end{equation}
and
\begin{equation}\label{eq decomposition estimation}
|u_N(\mathbf{x})|\leq C_N\,|\mathbf{x}|^{N},
\qquad
|\delta u_{N+1}(\mathbf{x})|\leq C_{N+1}\,|\mathbf{x}|^{N+1},
\end{equation}
where $C_N,C_{N+1}>0$ depend only on $k$, $N$, $\rho_0$, and $\mathcal E$.
\end{prop}

\begin{proof}
\noindent\textit{Case I: Three dimensions.}

Since $u$ is analytic in $B_{\rho_0}\subset\mathbb R^3$, it admits the convergent spherical expansion in $B_{\rho_0}$; see \cite{nedelec2001acoustic}:
\begin{equation}\label{eq u expansion}
u(r,\theta,\phi)=4\pi\sum_{\ell=0}^{\infty}\sum_{m=-\ell}^{\ell}\mathrm{i}^{\,\ell} a_\ell^{\,m}\, j_\ell(kr)\, Y_\ell^{\,m}(\theta,\phi),
\end{equation}
where $j_\ell(t)=\frac{t^\ell}{(2\ell+1)!!}+\mathcal{O}(t^{\ell+2})$ as $t\to 0$ is the spherical Bessel function of the first kind.
The condition $\mathrm{Vani}(u;0)=N$ implies $a_\ell^{\,m}=0$ for all $\ell<N$.
Define the leading homogeneous term
\[
u_N(r,\theta,\phi):=\frac{4\pi k^{N} r^{N}}{(2N+1)!!}\sum_{m=-N}^{N}\mathrm{i}^{\,N} a_{N}^{\,m}\, Y_{N}^{\,m}(\theta,\phi),
\]
and the remainder
\[
\begin{aligned}
\delta u_{N+1}(r,\theta,\phi)
&:=4\pi\sum_{m=-N}^{N}\mathrm{i}^{\,N} a_{N}^{\,m}\Big(j_{N}(kr)-\frac{(kr)^{N}}{(2N+1)!!}\Big)Y_{N}^{\,m}(\theta,\phi)\\
&\quad +\,4\pi\sum_{\ell=N+1}^{\infty}\sum_{m=-\ell}^{\ell}\mathrm{i}^{\,\ell} a_\ell^{\,m}\, j_\ell(kr)\, Y_\ell^{\,m}(\theta,\phi).
\end{aligned}
\]
Then $u=u_N+\delta u_{N+1}$.
Using
\[
j_{N}(kr)-\frac{(kr)^{N}}{(2N+1)!!}=\mathcal{O}(r^{N+2})
\]
and
\[
j_\ell(kr)=\mathcal{O}(r^\ell)\qquad\text{for }\ell\ge N+1,
\]
we obtain
\[
|u_N(\mathbf{x})|\leq C_N\,r^{N},
\qquad
|\delta u_{N+1}(\mathbf{x})|\leq C_{N+1}\,r^{N+1},
\]
with constants depending only on $k$, $N$, the radius of analyticity, and the uniform bound $\mathcal E$ of $u$ given in Lemma~\ref{lem uniform bound}.

\medskip
\noindent\textit{Case II: Two dimensions.}

In two dimensions, write $u$ in polar coordinates $(r,\theta)$ as
\begin{equation}\label{eq 2Dexpansion}
u(r,\theta)
=
\sum_{n=0}^{\infty}\big(a_n\,\mathrm{i}^{\,n}e^{\mathrm{i}n\theta}
+
b_n\,\mathrm{i}^{\,n}e^{-\mathrm{i}n\theta}\big)\,J_n(kr),
\end{equation}
where $J_n(t)=\sum_{p=0}^{\infty}\frac{(-1)^p}{p!\,(n+p)!}\,(t/2)^{n+2p}$ is the Bessel function of the first kind.

Assume $\mathrm{Vani}(u;0)=N$.
Then $a_n=b_n=0$ for all $n<N$ in \eqref{eq 2Dexpansion}.
Define the leading term
\begin{equation}\label{eq 2D_uN}
u_N(r,\theta)
:=
\frac{(kr/2)^N}{N!}\,
\big(a_N\,\mathrm{i}^{\,N}e^{\mathrm{i}N\theta}
+
b_N\,\mathrm{i}^{\,N}e^{-\mathrm{i}N\theta}\big),
\end{equation}
and the remainder
\begin{equation}\label{eq 2D_delta}
\begin{aligned}
\delta u_{N+1}(r,\theta)
&:=
\big(a_N\,\mathrm{i}^{\,N}e^{\mathrm{i}N\theta}
+
b_N\,\mathrm{i}^{\,N}e^{-\mathrm{i}N\theta}\big)\,
\Big(J_N(kr)-\frac{(kr/2)^N}{N!}\Big)\\
&\quad
+
\sum_{n=N+1}^{\infty}\big(a_n\,\mathrm{i}^{\,n}e^{\mathrm{i}n\theta}
+
b_n\,\mathrm{i}^{\,n}e^{-\mathrm{i}n\theta}\big)\,J_n(kr).
\end{aligned}
\end{equation}
Then $u=u_N+\delta u_{N+1}$ in a neighborhood of $0$.
Since
\[
J_N(kr)-\frac{(kr/2)^N}{N!}=\mathcal{O}(r^{N+2})
\qquad\text{as }r\to0,
\]
and
\[
J_n(kr)=\mathcal{O}(r^{n})\qquad\text{for }n\ge N+1,
\]
there exist constants $C_N,C_{N+1}>0$ such that
\begin{equation}\label{eq 2D_bounds}
|u_N(r,\theta)|\leq C_N\,r^{N},
\qquad
|\delta u_{N+1}(r,\theta)|\leq C_{N+1}\,r^{N+1}.
\end{equation}
The constants $C_N$ and $C_{N+1}$ depend only on $k$, $N$, a radius $\rho_0>0$ such that $B_{\rho_0}(0)$ is contained in the domain of analyticity, and the upper bound $\mathcal E$ from Lemma~\ref{lem uniform bound}, and are independent of $r$ and $\theta$.

Although the numerical values of $C_N$ and $C_{N+1}$ in two and three dimensions need not coincide, we use the same symbols to denote positive constants with the dependence specified above.
This completes the proof.
\end{proof}

\subsection{Geometric metrics and conditions}\label{subsec DRO}

We recall the definition of the classical Hausdorff distance and explain why we introduce an exterior-visible modification adapted to the admissible classes in Subsection~\ref{subsec admissible}.

Let $D$ and $D'$ be admissible obstacles belonging to either $\mathcal A$ or $\mathcal B$.
The classical Hausdorff distance is defined by
\begin{equation}\label{eq Hausdorff}
\mathrm d_H(D,D')
=
\max\left\{
\sup_{\mathbf{x}\in D}\mathrm{dist}(\mathbf{x},D'),
\ \sup_{\mathbf{x}\in D'}\mathrm{dist}(\mathbf{x},D)
\right\}.
\end{equation}

Recall that
\[
\mathbf G=\mathbb{R}^n\setminus\overline{D}.
\]
We decompose the complement of the union as
\[
\mathbb R^n\setminus(D\cup D')=\mathbf G_1\cup\bigcup_j\mathcal C_j,
\]
where $\mathbf G_1$ is the unbounded connected component and $\{\mathcal C_j\}$ are the bounded connected components.
Define
\[
\mathbf G_2:=\mathbb{R}^n\setminus\mathbf G_1.
\]
Then $\partial\mathbf G_1=\partial\mathbf G_2$ is the common boundary separating the unbounded exterior from the interior region.

We measure the geometric discrepancy at boundary points visible from infinity.
Define the modified Hausdorff distance by
\begin{equation}\label{eq d_M}
\tilde{\mathrm d}(D,D')
:=
\max\Big\{
\max_{\mathbf{x}\in\partial D\cap\partial\mathbf G_2}\mathrm{dist}(\mathbf{x},D'),
\ \max_{\mathbf{x}\in\partial D'\cap\partial\mathbf G_2}\mathrm{dist}(\mathbf{x},D)
\Big\}.
\end{equation}
We next show that $\tilde{\mathrm d}$ is equivalent to the classical Hausdorff distance; see \cite{aspri2022lipschitz, alessandrini2014stable}.

\begin{lem}\label{lem dm}
There exists $\hat C_1>0$, depending only on the a~priori parameters of the admissible class, such that
\begin{equation}\label{eq equiv1}
\tilde{\mathrm d}(D,D')
\leq \mathrm d_H(\partial D,\partial D')
\leq \hat C_1\,\tilde{\mathrm d}(D,D').
\end{equation}
\end{lem}

We also record the boundary-set equivalence for the Hausdorff distance within the admissible class; see \cite{liu2017stable, rondi2008stable}.

\begin{lem}\label{lem bdry-vs-set}
Let $D$ and $D'$ be admissible obstacles in $\mathcal A$ or $\mathcal B$.
There exist constants $\hat C_2,\hat C_3>0$, depending only on the a~priori parameters, such that
\begin{equation}\label{eq equiv2}
\hat C_2\,\mathrm d_H(\partial D,\partial D')
\leq \mathrm d_H(D,D')
\leq \hat C_3\,\mathrm d_H(\partial D,\partial D').
\end{equation}
\end{lem}

We adopt
\[
\mathfrak d:=\tilde{\mathrm d}(D,D')
\]
as the working geometric distance.
If an inner maximum is taken over the empty set, it is understood to be $0$.
For $D\neq D'$, one has $\mathfrak d>0$.
A maximizer $\mathbf{x}_0$ of $\mathfrak d$ exists by the compactness of $\partial D\cap\partial\mathbf G_2$ and $\partial D'\cap\partial\mathbf G_2$.
Every such maximizer lies on $\partial\mathbf G_1=\partial\mathbf G_2$ and is reachable from infinity through $\mathbf G_1$.

\begin{rem}
For convex polyhedra, the classical Hausdorff distance is attained at boundary points, often at vertices.
In the nonconvex case, however, a maximizer of $\mathrm d_H$ may lie inside $\mathbf G_2$ and hence be hidden from exterior propagation.
By using the localization induced by $\tilde{\mathrm d}$ in \eqref{eq d_M}, the comparison is restricted to exterior-visible boundary points on $\partial\mathbf G_1$, where three-sphere inequalities and local analytic decompositions are applicable.
Lemma~\ref{lem dm} transfers the resulting estimates to the boundary Hausdorff distance on admissible classes.
Combining Lemma~\ref{lem dm} with Lemma~\ref{lem bdry-vs-set}, it follows from \eqref{eq equiv1} and \eqref{eq equiv2} that the same order of stability is obtained for $\mathrm d_H(D,D')$.
\end{rem}

We next introduce another key tool used in the analysis, namely the \textit{Complex Geometric Optics (CGO) solution}.

\begin{defn}\label{defn CGO}
Let $\tau\in\mathbb R_+$ and let $k\in\mathbb R_+$ denote the wavenumber.
Choose $\mathbf d,\mathbf d^\perp\in\mathbb S^{n-1}$ with $\mathbf d\perp\mathbf d^\perp$, where $n=2,3$.
Set
\begin{equation}\label{eq rho}
\rho=\rho(\tau,k):=\tau\,\mathbf d+\mathrm i\sqrt{k^2+\tau^2}\,\mathbf d^\perp.
\end{equation}
Define
\begin{equation}\label{eq CGO}
u_0:=\exp(\rho\cdot\mathbf x),\qquad \mathbf x\in\mathbb R^n.
\end{equation}
Then $\rho\cdot\rho=-k^2$, and hence $\Delta u_0+k^2 u_0=0$.
\end{defn}

The following elementary bounds will be used repeatedly; see also \cite{cakoni2020corner}.

\begin{lem}\label{lem Gamma}
Fix $h>0$ and $b>0$.
There exist $\mu_0=\mu_0(h,b)>1$ and $0<\alpha<1$ such that for every complex $\mu$ with $\Re\mu\ge\mu_0$ one has
\begin{equation}\label{eq gamma}
\left|\int_0^h r^{b-1}e^{-\mu r}\,\mathrm dr\right|
\leq \left|\frac{\Gamma(b)}{\mu^b}\right|+\frac{2\,e^{-\Re\mu\,\alpha h}}{\Re\mu},
\end{equation}
and
\begin{equation}\label{eq gamma2}
\left|\int_h^\infty r^{b-1}e^{-\mu r}\,\mathrm dr\right|
\leq \int_h^\infty r^{b-1}e^{-\Re\mu r}\,\mathrm dr
\leq \int_h^\infty e^{-\Re\mu\alpha r}\,\mathrm dr
=\frac{2\,e^{-\Re\mu \alpha h}}{\Re\mu}.
\end{equation}
\end{lem}

For the CGO solution $u_0$, let $P_h\subset\mathbb R^3$ and $Q_h\subset\mathbb R^2$ be bounded test domains in three and two dimensions, respectively.
Assume there exist sets of directions $\mathcal K_{\alpha'}\subset\mathbb S^2$ and $\mathcal K_{\alpha^*}\subset\mathbb S^1$, and constants $\alpha'>0$ and $\alpha^*>0$, such that
\begin{equation}\label{eq ec condition}
-\mathbf d\cdot\hat{\mathbf x}>\alpha',
\quad \forall\,\mathbf d\in\mathcal K_{\alpha'},\ \mathbf x\in P_h\setminus\{0\};
\qquad
-\mathbf d\cdot\hat{\mathbf x}>\alpha^*,
\quad \forall\,\mathbf d\in\mathcal K_{\alpha^*},\ \mathbf x\in Q_h\setminus\{0\},
\end{equation}
where $\hat{\mathbf x}=\mathbf x/|\mathbf x|$.
We refer to $\mathcal K_{\alpha'}$ and $\mathcal K_{\alpha^*}$ as the corresponding dual direction sets.

Under \eqref{eq ec condition}, the CGO solution satisfies the uniform decay
\begin{equation}\label{eq alpha}
|u_0(\mathbf x)|=e^{\Re(\rho\cdot\mathbf x)}\leq e^{-\alpha'\tau|\mathbf x|}\leq 1,
\qquad \mathbf x\in P_h\setminus\{0\},
\end{equation}
in three dimensions, and
\begin{equation*}
|u_0(\mathbf x)|=e^{\Re(\rho\cdot\mathbf x)}\leq e^{-\alpha^*\tau|\mathbf x|}\leq 1,
\qquad \mathbf x\in Q_h\setminus\{0\},
\end{equation*}
in two dimensions.
It is therefore essential to choose $\mathbf d$ in \eqref{eq rho} from $\mathcal K_{\alpha'}$ in three dimensions and from $\mathcal K_{\alpha^*}$ in two dimensions.

\section{Propagation of Smallness from Far-Field to Boundary}\label{sec 4}

In this section, we establish a quantitative link between the far-field error and the boundary error on the obstacle.
The stability estimate is a quantitative manifestation of unique continuation.
The argument uses regularity and unique continuation tools that differ in three and two dimensions.
We first present the argument in three dimensions and then indicate the corresponding modifications in two dimensions.
We begin with the three-sphere inequality for the Helmholtz equation.

\begin{lem}\label{lem three spheres}
There exist constants $\tilde R>0$, $C>0$, and $c_1\in(0,1)$, depending only on $k$, such that the following holds in $\mathbb R^3$.
Let $0<r_1<r_2<r_3\leq \tilde R$ and let $B_{r_3}(\mathbf{x}_0)\subset U$.
If $u$ satisfies $\Delta u+k^2u=0$ in $U$, then for any $s$ with $r_2<s<r_3$,
\[
\|u\|_{L^\infty(B_{r_2}(\mathbf{x}_0))}
\leq
C\left(1-\frac{r_2}{s}\right)^{-3/2}
\|u\|_{L^\infty(B_{r_3}(\mathbf{x}_0))}^{1-\beta}
\|u\|_{L^\infty(B_{r_1}(\mathbf{x}_0))}^{\beta},
\]
where $\beta$ satisfies the same bounds as in \cite{rondi2008stable}.
\end{lem}

A crucial geometric ingredient is the existence of an exterior path linking the near-field region to a point near the boundary, along which smallness can be propagated.
The following proposition provides the corresponding construction; we omit the proof and refer to \cite[Proposition~3.8]{aspri2022lipschitz}.

\begin{prop}\label{prop G1}
Let $\Sigma,\Sigma'\in\mathcal A$ be two admissible obstacles.
There exist constants $\hat C>0$ and $r>0$, depending only on the a~priori parameters, with the following property.
For any near-field point $\mathbf x_0\in B_{2R_0}\setminus\Omega$ and any point $\mathbf x\in\partial\Sigma\cap\partial\mathbf G_1$, one has
\[
\hat C\, \tilde{\mathrm d}(\Sigma,\Sigma')^{3} \leq \operatorname{dist}(\mathbf x,\Sigma').
\]
Moreover, there exists a curve $\gamma$ joining $\mathbf x_0$ to $\mathbf x+a\,\omega_{\mathbf x}$, where $\omega_{\mathbf x}$ is the direction given by Assumption~\ref{asm 3} at $\mathbf x$ and $a\in(0,1)$ depends only on the a~priori parameters, such that
\[
V_r(\gamma)\subset \mathbf G_1,
\qquad
V_r(\gamma):=\bigcup_{\mathbf y\in\gamma}B_r(\mathbf y).
\]
\end{prop}

With the curve $\gamma$ from Proposition~\ref{prop G1} in place, we construct a chain of overlapping balls to propagate smallness along $\gamma$ by means of Lemma~\ref{lem three spheres}.
Fix $r>0$.
We consider a chain of balls of radius $r_1=r/4$ whose centers lie on $\gamma$ and whose successive centers are separated by at most $r_2-r_1=r/4$.
Accordingly, Lemma~\ref{lem three spheres} will be applied with
\[
r_1=\frac r4,\qquad r_2=\frac r2,\qquad r_3=r.
\]

\begin{lem}\label{lem three}
Let $\mathbf G_1\subset\mathbb R^3$ be connected, and let $\gamma$ be a rectifiable curve joining two distinct points $\mathbf x,\mathbf y\in\mathbf G_1$ such that $V_r(\gamma)\subset\mathbf G_1$.
Assume that $u\in H^1_{\mathrm{loc}}(\mathbf G_1)$ satisfies $\Delta u+k^2u=0$ in $\mathbf G_1$.
Then
\begin{equation}\label{eq iteration of three sphere}
\|u\|_{L^\infty(B_{r_1}(\mathbf y))}
\leq C_t\,\mathcal E\,\|u\|_{L^\infty(B_{r_1}(\mathbf x))}^{\beta^{\,N}},
\end{equation}
where $\mathcal E$ is from Lemma~\ref{lem uniform bound}, $\beta\in(0,1)$ is from Lemma~\ref{lem three spheres}, $N=\big\lceil d_\gamma/(r_2-r_1)\big\rceil$ with $d_\gamma$ the length of $\gamma$, and $C_t>0$ depends only on the a~priori parameters.
\end{lem}

\begin{proof}
Choose points $\mathbf z_1=\mathbf x,\mathbf z_2,\ldots,\mathbf z_{N+1}=\mathbf y$ on $\gamma$ such that
\[
|\mathbf z_{l+1}-\mathbf z_l|\leq r_2-r_1.
\]
Then
\[
B_{r_1}(\mathbf z_{l+1})\subset B_{r_2}(\mathbf z_l),
\qquad
B_{r_3}(\mathbf z_l)\subset V_r(\gamma)\subset\mathbf G_1.
\]
Applying Lemma~\ref{lem three spheres} at each $\mathbf z_l$, we obtain a constant $C>0$, depending only on the a~priori parameters and on $k$, such that
\[
\|u\|_{L^\infty(B_{r_1}(\mathbf z_{l+1}))}
\leq
C\,\|u\|_{L^\infty(B_{r_3}(\mathbf z_l))}^{1-\beta}
\|u\|_{L^\infty(B_{r_1}(\mathbf z_l))}^{\beta}.
\]

Since $B_{r_3}(\mathbf z_l)\subset B_{2R_0}$ for all $l$, standard interior elliptic estimates together with \eqref{eq uniform bounded} yield a constant $C_u>0$ such that
\[
\|u\|_{L^\infty(B_{r_3}(\mathbf z_l))}\leq C_u\,\mathcal E.
\]
Hence
\[
\|u\|_{L^\infty(B_{r_1}(\mathbf z_{l+1}))}
\leq
(C\,C_u^{\,1-\beta})\,\mathcal E^{\,1-\beta}
\|u\|_{L^\infty(B_{r_1}(\mathbf z_l))}^{\beta}.
\]
Iterating for $l=1,\ldots,N$ yields
\[
\|u\|_{L^\infty(B_{r_1}(\mathbf y))}
\leq
(C\,C_u^{\,1-\beta})^{1+\beta+\cdots+\beta^{N-1}}
\mathcal E^{\,1-\beta^{N}}
\|u\|_{L^\infty(B_{r_1}(\mathbf x))}^{\beta^{\,N}}.
\]
Since
\[
1+\beta+\cdots+\beta^{N-1}\leq (1-\beta)^{-1},
\]
the fixed factor can be absorbed into a constant $C_t>0$, and \eqref{eq iteration of three sphere} follows.

The proof is complete.
\end{proof}

When the incident direction $\mathbf p$ and the wavenumber $k$ in \eqref{eq incident} are fixed, the single far-field error is measured by
\[
\|u_\infty(\hat{\mathbf{x}})-u'_\infty(\hat{\mathbf{x}})\|_{L^2(\mathbb S^2)}.
\]
We aim to estimate $u-u'$ and $\nabla(u-u')$ on boundary portions of the obstacles.
Define the error field by
\begin{equation}\label{eq w}
w:=u-u'
\qquad\text{in }B_{2R_0}\setminus(\Sigma\cup \Sigma').
\end{equation}
Then $w$ satisfies
\[
\Delta w+k^2 w=0
\qquad\text{in }B_{2R_0}\setminus(\Sigma\cup \Sigma'),
\]
and represents the error propagated from the far field.

\subsection{Stability estimates: from far-field to near-field}

The bounded annulus $B_{2R_0}\setminus \Omega$ associated with the admissible class $\mathcal A$ is called the \textit{near-field region}.
The function $w$ defined in \eqref{eq w} is referred to as the \textit{near-field error}.

\begin{lem}\cite{isakov1992stability}\label{lem Isakov}
Assume that the far-field error satisfies
\begin{equation}\label{eq far-field error}
\|u_\infty(\hat{\mathbf x})-u'_\infty(\hat{\mathbf x})\|_{L^2(\mathbb S^{n-1})}\leq \varepsilon
\end{equation}
for some $\varepsilon>0$.
Then there exists $\varepsilon_m\in(0,e^{-1})$, depending only on the a~priori parameters, such that if $0<\varepsilon<\varepsilon_m$, then
\[
\|w\|_{L^\infty(B_{2R_0}\setminus \Omega)}
\leq
C_f\,\exp\!\Big[-\big(-\log \varepsilon\big)^{1/2}\Big],
\]
where $w$ is defined in \eqref{eq w}, and $C_f>0$ depends only on the a~priori parameters.

Moreover, there exists $\zeta>0$, depending only on the a~priori parameters, such that for every $\mathbf x_0$ with
\[
B_\zeta(\mathbf x_0)\subset B_{2R_0}\setminus \Omega,
\]
one has
\begin{equation}\label{eq near field error}
\|w\|_{L^\infty(B_\zeta(\mathbf x_0))}
\leq
\varepsilon_1,
\qquad
\varepsilon_1:=C_f\,\exp\!\Big[-\big(-\log \varepsilon\big)^{1/2}\Big].
\end{equation}
\end{lem}


\subsection{Stability estimates: from near-field to boundary}

The next lemma gives the local regularity of the error field $w$ near a boundary patch, namely its local $C^{1,\alpha}$ regularity.

\begin{lem}\label{lem local Holder}
Let $\Sigma,\Sigma' \in \mathcal A$, and let $u,u'$ be the solutions to \eqref{eq main system} associated with $(\Sigma,\eta)$ and $(\Sigma',\eta')$, respectively.
Assume that there exist a point $\mathbf y_1\in\partial\Sigma$ and a constant $\tilde h>0$ such that
\[
B_{\tilde h}(\mathbf y_1)\cap \Sigma'=\emptyset.
\]
Then there exist $h\in(0,\tilde h)$ and $\alpha\in(0,1)$, depending only on the a~priori parameters, such that
\[
w\in C^{1,\alpha}\bigl(\overline{B_h(\mathbf y_1)\setminus(\Sigma\cup\Sigma')}\bigr),
\qquad
w:=u-u'.
\]
\end{lem}

\begin{proof}
After a rigid motion, we may assume that $\mathbf y_1=0$ and that the face of $\Sigma$ containing $\mathbf y_1$ is contained in the plane
\[
\Pi:=\{\mathbf{x}=(x_1,x_2,x_3)\in\mathbb R^3:\ x_3=0\}.
\]
Let $\Pi_1\subset\partial\Sigma$ denote the corresponding planar face near the origin.

Set $h_1:=\tilde h$.
Since $u\in H^1_{\mathrm{loc}}(\mathbb R^3\setminus\Sigma)$, one has
\[
u\in H^1(B_{h_1}(\mathbf y_1)\setminus\Sigma).
\]
By the trace theorem \cite[Theorem~3.37]{mclean2000strongly} and Lemma~\ref{lem RCP}, there exists a constant $C_1>0$ such that
\[
\|\sigma u\|_{H^{1/2}(\Pi_1\cap B_{h_1}(\mathbf y_1))}
\leq
C_1\|u\|_{H^1(B_{h_1}(\mathbf y_1)\setminus\Sigma)},
\]
where $\sigma$ denotes the trace operator.

Since $u$ satisfies the impedance boundary condition on $\Pi_1\cap B_{h_1}(\mathbf y_1)$, we have
\[
\partial_\nu u+\eta u=0
\qquad\text{on }\Pi_1\cap B_{h_1}(\mathbf y_1).
\]
Hence there exists a constant $C_2>0$ such that
\[
\left\|\partial_\nu u\right\|_{H^{1/2}(\Pi_1\cap B_{h_1}(\mathbf y_1))}
\leq
C_2\|\sigma u\|_{H^{1/2}(\Pi_1\cap B_{h_1}(\mathbf y_1))}.
\]
Therefore, the assumptions of \cite[Theorem~4.18]{mclean2000strongly} are satisfied.
It follows that for some $h_2\in(0,h_1)$ there exists $C_3>0$ such that
\[
\|u\|_{H^2(B_{h_2}(\mathbf y_1)\setminus\Sigma)}
\leq
C_3\|u\|_{H^1(B_{h_1}(\mathbf y_1)\setminus\Sigma)}
+
C_3\left\|\partial_\nu u\right\|_{H^{1/2}(\Pi_1\cap B_{h_1}(\mathbf y_1))}.
\]
Combining the previous estimates, we obtain
\[
\|u\|_{H^2(B_{h_2}(\mathbf y_1)\setminus\Sigma)}
\leq
C_4\|u\|_{H^1(B_{h_1}(\mathbf y_1)\setminus\Sigma)}
\]
for some constant $C_4>0$.

Since $\Pi_1$ is planar, we may repeat the same argument on a smaller ball.
Thus, for some $h_3\in(0,h_2)$ there exists $C_5>0$ such that
\[
\|u\|_{H^3(B_{h_3}(\mathbf y_1)\setminus\Sigma)}
\leq
C_5\|u\|_{H^2(B_{h_2}(\mathbf y_1)\setminus\Sigma)}.
\]
Hence
\[
u\in H^3(B_{h_3}(\mathbf y_1)\setminus\Sigma).
\]

By the Sobolev embedding theorem \cite[Theorem~5.4]{adams1975sobolev}, together with the local Lipschitz regularity of admissible obstacles, there exist $\alpha\in(0,1)$ and $C_6>0$, depending only on the a~priori parameters, such that
\[
\|u\|_{C^{1,\alpha}(\overline{B_{h_3}(\mathbf y_1)\setminus\Sigma})}
\leq
C_6\|u\|_{H^3(B_{h_3}(\mathbf y_1)\setminus\Sigma)}.
\]

On the other hand, since $u'$ is analytic in $B_{\tilde h}(\mathbf y_1)$ and $h_3<\tilde h$, one has
\[
u'\in C^{1,\alpha}(\overline{B_{h_3}(\mathbf y_1)\setminus\Sigma'}).
\]
Taking $h:=h_3$, it follows that
\[
w=u-u'\in C^{1,\alpha}\bigl(\overline{B_h(\mathbf y_1)\setminus(\Sigma\cup\Sigma')}\bigr).
\]

The proof is complete.
\end{proof}

\begin{prop}\label{prop w}
Let $\Sigma,\Sigma'\in\mathcal A$ be two admissible obstacles.
Assume that the far-field error satisfies $0<\varepsilon\leq \varepsilon_m$.
Then there exist a point $\mathbf x_1\in\partial\Sigma\cap\partial\mathbf G_2$ attaining the maximum in $\tilde{\mathrm d}(\Sigma,\Sigma')$ and a radius
\[
h\in(0,\tilde{\mathrm d}(\Sigma,\Sigma'))
\]
such that
\[
S_h:=\partial\Sigma\cap B_h(\mathbf x_1)
\qquad\text{satisfies}\qquad
S_h\cap\Sigma'=\emptyset.
\]
Moreover,
\begin{equation}\label{eq T_1}
\max_{\mathbf{x}\in S_h}\bigl\{|\nabla w(\mathbf{x})|,\ |w(\mathbf{x})|\bigr\}
\leq
C_w(\log|\log(1/\varepsilon)|)^{-\alpha}
=:T(\varepsilon),
\end{equation}
where $\alpha\in(0,1)$ and $C_w>0$ depend only on the a~priori parameters.
\end{prop}

\begin{proof}
We propagate smallness from the near-field region to a boundary patch of $\partial\Sigma$.
By the definition of $\tilde{\mathrm d}$ in \eqref{eq d_M}, choose
\[
\mathbf x_1\in\partial\Sigma\cap\partial\mathbf G_2
\]
such that
\[
\mathfrak d:=\tilde{\mathrm d}(\Sigma,\Sigma')>0
\]
is attained at $\mathbf x_1$.
Fix
\[
h=c_h\mathfrak d,
\qquad 0<c_h<1,
\]
so that
\[
S_h:=\partial\Sigma\cap B_h(\mathbf x_1)
\]
satisfies $S_h\cap\Sigma'=\emptyset$.
After a rigid motion, we may assume that $\mathbf x_1=0$ and that $0\in\Pi_1\subset\partial\Sigma$.

By Proposition~\ref{prop G1}, there exists a rectifiable curve $\gamma_0$ connecting a near-field point
\[
\mathbf x_0\in B_{2R_0}\setminus B_{R_0},
\]
for which \eqref{eq near field error} holds, to a point
\[
\mathbf z_1\in B_h(0)\setminus\overline\Sigma,
\]
together with a radius $r>0$ such that
\[
V_r(\gamma_0)\subset\mathbf G_1.
\]
Using the exterior cone condition from Assumption~\ref{asm 3}, we may extend $\gamma_0$ by a short segment inside the exterior cone at the origin and thereby obtain a curve $\gamma$ whose tubular neighbourhood still satisfies
\[
V_r(\gamma)\subset\mathbf G_1.
\]
Choosing $r>0$ sufficiently small, we may also ensure that
\[
B_r(\mathbf z_1)\subset B_h(0)\setminus\overline\Sigma.
\]

Fix $r\in(0,h/4)$ and define
\[
Q_1:=\{\mathbf y\in B_h(0): \mathrm{dist}(\mathbf y,\partial\Sigma)\geq r\},
\qquad
Q_2:=B_h(0)\setminus Q_1.
\]
Since $B_h(0)\cap\Sigma'=\emptyset$, every point $\mathbf y\in Q_2$ can be connected to a point $\mathbf z_3\in Q_1$ with
\[
|\mathbf y-\mathbf z_3|<r,
\qquad
B_r(\mathbf z_3)\cap(\Sigma\cup\Sigma')=\emptyset.
\]

Construct a chain of overlapping balls of radius $r$ with centers on $\gamma$, beginning at $B_r(\mathbf z_0)$, where $\mathbf z_0=\mathbf x_0$, and ending at $B_r(\mathbf z_3)$.
Let $d_\gamma$ be the length of $\gamma$.
If $N$ denotes the number of links in the chain, then
\[
N\leq \frac{d_\gamma}{r}+1.
\]

By Lemma~\ref{lem Isakov}, the near-field error satisfies
\[
\varepsilon_1:=\|w\|_{L^\infty(B_\zeta(\mathbf x_0))}
\leq
C_f\,\exp\!\left(-(-\log\varepsilon)^{1/2}\right).
\]
Applying Lemma~\ref{lem three} along the chain gives
\[
\|w\|_{L^\infty(B_r(\mathbf z_3))}
\leq
C_t\,\mathcal E\,\varepsilon_1^{\,\beta^{\,N}}
\leq
C_t\,\mathcal E\,\varepsilon_1^{\,\beta^{\,d_\gamma/r+1}},
\]
where $\beta\in(0,1)$ is the constant from Lemma~\ref{lem three spheres}.

We now choose
\[
r:=\frac{d_\gamma\,|\log\beta|}{(1-\alpha)\,\log|\log\varepsilon_1|}.
\]
For sufficiently small $\varepsilon$, this choice ensures
\[
r\leq \min\{h/4,\zeta\}.
\]
Moreover, there exist constants $C_1,C_2>0$, depending only on the a~priori parameters, such that
\[
\varepsilon_1^{\,\beta^{\,d_\gamma/r+1}}
=
\exp\!\Big(-|\log\varepsilon_1|\,\beta^{\,\log_\beta(|\log\varepsilon_1|^{\,\alpha-1})+1}\Big)
\leq
C_1(\log|\log(1/\varepsilon)|)^{-\alpha},
\]
and
\[
(2r)^\alpha
=
\Big(\frac{2d_\gamma\,|\log\beta|}{1-\alpha}\Big)^\alpha
(\log|\log\varepsilon_1|)^{-\alpha}
\leq
C_2(\log|\log(1/\varepsilon)|)^{-\alpha}.
\]

Let $\mathbf x\in S_h$.
Choose $\mathbf y\in Q_2$ with
\[
|\mathbf x-\mathbf y|=r,
\]
and let $\mathbf z_3\in Q_1$ be the corresponding point above, so that
\[
|\mathbf x-\mathbf z_3|\leq 2r
\qquad\text{and}\qquad
B_r(\mathbf z_3)\cap(\Sigma\cup\Sigma')=\emptyset.
\]
By Lemma~\ref{lem local Holder}, we have
\[
|w(\mathbf x)|
\leq
\mathcal T\,|\mathbf x-\mathbf z_3|^\alpha
+\|w\|_{L^\infty(B_r(\mathbf z_3))}
\leq
C_w(\log|\log(1/\varepsilon)|)^{-\alpha},
\]
where $\mathcal T$ denotes the corresponding H\"older seminorm and $C_w>0$ depends only on the a~priori parameters.

For the gradient, Lemma~\ref{lem local Holder} gives
\[
|\nabla w(\mathbf x)-\nabla w(\mathbf z_3)|
\leq
\mathcal T\,|\mathbf x-\mathbf z_3|^\alpha
\leq
(2r)^\alpha\mathcal T.
\]
On the other hand, interior estimates for $\Delta w+k^2w=0$ imply
\[
|\nabla w(\mathbf z_3)|
\leq
C\,r^{-1}\,\|w\|_{L^\infty(B_r(\mathbf z_3))},
\qquad
r^{-1}
=
\frac{(1-\alpha)\log|\log\varepsilon_1|}{d_\gamma\,|\log\beta|}.
\]
Combining the above bounds for $(2r)^\alpha$ and for $\|w\|_{L^\infty(B_r(\mathbf z_3))}$ yields
\[
|\nabla w(\mathbf x)|
\leq
C_w(\log|\log(1/\varepsilon)|)^{-\alpha},
\qquad
\mathbf x\in S_h.
\]
Together with the estimate for $|w(\mathbf x)|$, this proves \eqref{eq T_1}.

The proof is complete.
\end{proof}

\section{The Two-Dimensional Proof of Theorem~\ref{th relation}}\label{sec 2D}

We begin with the two-dimensional configuration near a flat boundary segment of a non-convex polygonal obstacle.
For such polygons, the geometric discrepancy may be localized in the symmetric difference $K\triangle K'$, which need not contain any corner.
To capture the geometric discrepancy in this corner-free regime, we introduce a two-dimensional ATD attached to a flat boundary segment.
\subsection{The ATD construction in the two-dimensional case}\label{subsec ATD2D}
We denote by 
\[
\mathfrak d := \tilde{\mathrm d}(K,K')
\]
the modified Hausdorff distance defined in \eqref{eq d_M}, and we restrict attention to the case
\begin{equation}\label{eq contr 2}
	\mathfrak d>0 .
\end{equation}

By the definition of $\tilde{\mathrm d}$ and the discussion in Subsection~\ref{subsec DRO}, there exists an exterior-visible \emph{ATD test point} $\mathbf{p}_0\in\partial K\cap\partial\mathbf G_2$ such that 
\[
B_{\mathfrak d}(\mathbf{p}_0)\cap K'=\emptyset .
\]
Since no strict convexity is assumed for $K$, the ball $B_{\mathfrak d}(\mathbf{p}_0)$ need not contain any corner of $K$. 
We therefore work in a neighborhood of a line segment contained in $\partial K$.
After a rigid motion, we may assume that $\mathbf{p}_0=0$ and that 
\[
I_1\subset B_{\mathfrak d}(\mathbf{p}_0)\cap\partial K
\]
lies on the $x_1$-axis. 

We fix a constant $c_2\in(0,1)$ and set $h=c_2\mathfrak d$. 
We also fix an opening angle $\theta_0\in(0,\pi)$. 
We then define the \emph{ATD test domain} $Q_h\subset\mathbb R^2$ by
\[
\partial Q_h = I_1\cup I_2\cup I_3,
\]
where the boundary components are parametrized by
\begin{equation}\label{eq I123}
\begin{aligned}
I_1 &= \{\mathbf{x}=(x_1,x_2)\in\mathbb{R}^2 \mid x_1\in(0,h),\ x_2=0\}, \\
I_2 &= \{\mathbf{x}=(x_1,x_2)\in\mathbb{R}^2 \mid (x_1,x_2)=(r\cos\theta_0,r\sin\theta_0),\ r\in(0,h)\}, \\
I_3 &= \{\mathbf{x}=(x_1,x_2)\in\mathbb{R}^2 \mid (x_1,x_2)=(h\cos\theta,h\sin\theta),\ \theta\in(0,\theta_0)\}.
\end{aligned}
\end{equation}

Two rays obtained by extending $I_1$ and $I_2$ from the origin are denoted by $\tilde{I_1}$ and $\tilde{I_2}$, respectively, and are parametrized by
\begin{equation}\label{eq II'}
\begin{aligned}
\tilde{I_1} &= \{\mathbf{x}= (x_1(t),x_2(t)) = (t,0) \in \mathbb{R}^2 \mid t \geq 0 \}, \\
\tilde{I_2} &= \{\mathbf{x}= (x_1(t),x_2(t)) = (t\cos\theta_0, t\sin\theta_0) \in \mathbb{R}^2 \mid t \geq 0 \}.
\end{aligned}
\end{equation}
We define the exterior parts of the rays beyond $Q_h$ by $I_1' := \tilde{I_1} \setminus I_1$ and $I_2' := \tilde{I_2} \setminus I_2$.
The parameter $\theta_0$ is kept free so that the opening angle of $Q_h$ can be adjusted and geometric degeneracy of the test region can be avoided.
For illustration, Figure~\ref{fig:quarter_circle} shows the special case $\theta_0 = \frac{\pi}{2}$.
In the picture, the region above the boundary arc represents the interior of the obstacle.
\begin{figure}[htbp]
    \centering
    \includegraphics[width=0.6\textwidth, keepaspectratio]{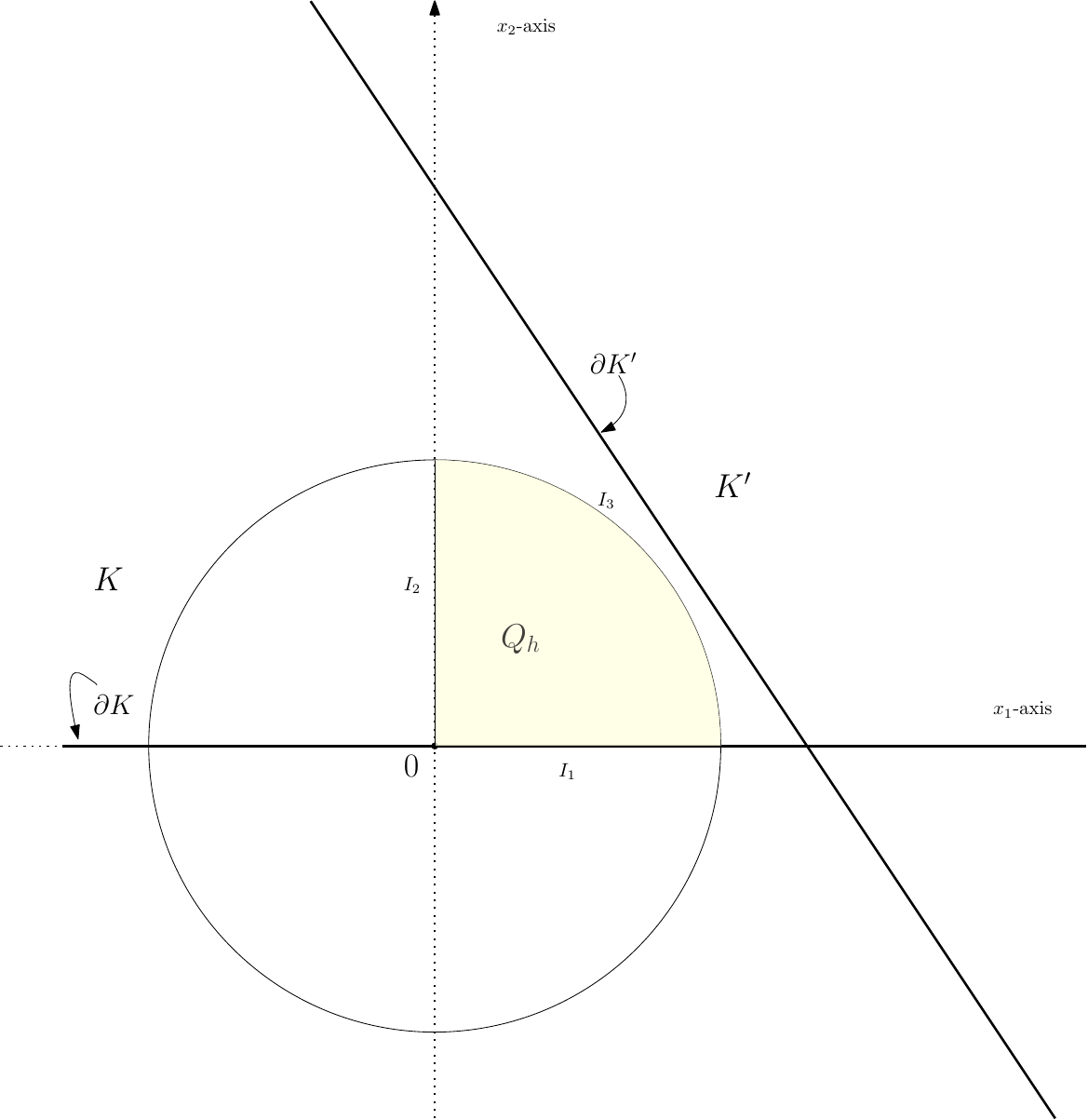}
    \caption{The two-dimensional test domain $Q_h$ attached to a boundary segment.}
    \label{fig:quarter_circle}
\end{figure}

The following remarks explain why the ATD is constructed along a flat boundary segment and why the later non-degeneracy analysis is formulated in terms of hyperplane-type configurations.

\begin{rem}\label{rem:2D-corner-free}
Within the admissible class of Definition~\ref{defn admissible class}, the discrepancy region $K\triangle K'$ need not contain any corner of the polygonal obstacle.
If such a corner is present, then the corner geometry already provides an additional singular feature, and the later flat-boundary non-degeneracy issue is no longer the essential difficulty.
The ATD construction introduced here is therefore designed for the genuinely corner-free case, where the geometric discrepancy is localized near a flat boundary segment.
\end{rem}

\begin{rem}\label{rem:2D-segment-hyperplane}
The ATD test point is chosen on an exterior-visible flat boundary segment.
Accordingly, the later non-degeneracy analysis is attached to the flat boundary piece containing that point, rather than to a single distinguished point.
This is consistent with the hyperplane language introduced in Definition~\ref{defn:ext-hyper}; see also Section~\ref{sec:nondeg}.
\end{rem}

As explained in Remark~\ref{rem:relation-framework}, we carry out the detailed derivation first in the bounded-impedance setting.
Accordingly, throughout this subsection we assume that $\eta\in\Xi_1$.

Fix an incident wave $u^i$ with direction $\mathbf{p}\in\mathbb{S}^1$ and wavenumber $k\in\mathbb{R}_+$.
Let $u \in H^1(\mathbb{R}^2 \setminus K)$ and $u' \in H^1(\mathbb{R}^2 \setminus K')$ denote the solutions to \eqref{eq main system} associated with $(K,\eta)$ and $(K',\eta')$, respectively.
Since $\mathbf{p}_0=0$ after the rigid motion and $h=c_2\mathfrak d$ with $0<c_2<1$, we have $B_h\subset B_{\mathfrak d}(\mathbf{p}_0)$.
In particular, $u'$ is analytic in the test domain $Q_h$ and admits a local expansion at the origin as stated in Proposition~\ref{prop decom}.
By Lemma~\ref{lem vanish}, the analytic function $u'$ has finite vanishing order at the test point $0$, which we denote by $N$.
We use the decomposition $u' = u'_N + \delta u'_{N+1}$ given by \eqref{eq 2D_uN} and \eqref{eq 2D_delta}, which satisfies the bounds in \eqref{eq 2D_bounds}.
Since $I_1\subset \partial K$, the total field $u$ satisfies the impedance boundary condition on $I_1$.
In general, $u'$ does not satisfy this boundary condition on $I_1$.
In what follows, we repeatedly use the boundary condition for $u$ on $I_1$ when simplifying the boundary terms arising from Green's identity.

Applying Green's second identity on $Q_h$, we obtain the following integral identity.

\begin{prop}
Under Assumption~\eqref{eq contr 2} and the geometric configuration \eqref{eq I123}--\eqref{eq II'}, the following integral identity holds:
\begin{equation}\label{eq II2d}
\begin{aligned}
&\int_{\tilde{I_1} \cup \tilde{I_2}} u'_N \partial_{\nu} u_{0} \, \mathrm{d}\sigma 
  - \int_{\tilde{I_2}} \partial_\nu u'_N u_0 \, \mathrm{d}\sigma \\
&= \int_{I_1} u_{0} \partial_{\nu}(u^{\prime}-u) \, \mathrm{d}\sigma 
  + \int_{I_1} \eta(u^{\prime}-u)u_{0} \, \mathrm{d}\sigma 
  - \int_{I_1} \eta u^{\prime}_N u_{0} \, \mathrm{d}\sigma \\
&\quad  - \int_{I_1} \eta \delta{u'}_{N+1} u_{0} \, \mathrm{d}\sigma 
  - \int_{I_1\cup I_2} \delta u'_{N+1} \partial_{\nu} u_{0} \, \mathrm{d}\sigma \\
&\quad  + \int_{I_1'\cup I_2'} u'_{N} \partial_\nu u_0 \, \mathrm{d}\sigma
  - \int_{I'_2} \partial_\nu u'_N u_0 \, \mathrm{d}\sigma
  + \int_{I_2} \partial_\nu \delta u'_{N+1} u_0 \, \mathrm{d}\sigma \\
&\quad  + \int_{I_3} \left( u_{0} \partial_{\nu} u^{\prime} - u^{\prime} \partial_{\nu} u_{0} \right) \, \mathrm{d}\sigma ,
\end{aligned}
\end{equation}
where $\nu$ denotes the exterior unit normal to $Q_h$ on $\partial Q_h$.
\end{prop}

\begin{proof}
By Assumption~\eqref{eq contr 2}, there exists a test domain $Q_h$ as described above.
In the Lipschitz domain $Q_h$ we apply Green's second identity
\begin{equation}\label{eq green}
\int_{Q_h} \left(g\Delta f-f\Delta g\right)\, \mathrm{d}\mathbf{x}
=
\int_{\partial Q_h}\left(g\partial_{\nu}f-f\partial_{\nu}g \right)\, \mathrm{d}\sigma .
\end{equation}
Let $u_0$ be the CGO solution from Definition~\ref{defn CGO} with the form \eqref{eq CGO}.
Set $f=u'$ and $g=u_0$ in \eqref{eq green}.
Since $Q_h\cap K'=\emptyset$, we have $(\Delta+k^2)u'=0$ in $Q_h$, and by definition $(\Delta+k^2)u_0=0$.
Hence, the volume integral in \eqref{eq green} vanishes, and \eqref{eq green} reduces to a boundary identity on $\partial Q_h$.

In $Q_h$ we decompose 
\begin{equation}\label{eq deltauN}
	u'=u'_N+\delta u'_{N+1}
\end{equation} 
using \eqref{eq 2D_uN} and \eqref{eq 2D_delta}, and we split the boundary integral over $\partial Q_h=I_1\cup I_2\cup I_3$.
On $I_1\subset\partial K$ the total field $u$ satisfies the impedance boundary condition, and we introduce $w:=u-u'$ to rewrite the $I_1$ contributions in terms of $w$, $u'_N$, and $\delta u'_{N+1}$.
The remaining contributions on $I_2$, $I_3$, and on the rays $I_1', I_2'$ are grouped according to the decomposition of $u'$.
Collecting all terms gives \eqref{eq II2d}.

The proof is complete.
\end{proof}

We do not repeat here the full smallness propagation argument from the far field to the boundary, since the proof developed in Section~\ref{sec 4} for the three-dimensional case carries over to the present two-dimensional configuration with only minor changes in the constants.
Under Assumption~\eqref{eq contr 2} and the definition of the modified Hausdorff distance \eqref{eq d_M}, the test domain $Q_h$ lies in the exterior of $K'$ and can be connected to the far-field region by the same exterior-visibility geometry as in Section~\ref{sec 4}, through the two-dimensional analogue of Proposition~\ref{prop G1}.
Along this connecting curve, we first transfer the far-field error to a near-field bound by Lemma~\ref{lem Isakov}, thereby obtaining quantitative smallness of $u-u'$ on a ball intersecting the observation region.
We then propagate this smallness step by step towards $Q_h$ by an iteration scheme based on the two-dimensional counterparts of Lemmas~\ref{lem three spheres} and~\ref{lem three}, exactly as in the three-dimensional case but with dimension-dependent constants.
The only genuinely two-dimensional input is the local H\"older regularity of the total field near the boundary, which is ensured by \cite[Lemma~3.2]{diao2024stable}.
With this regularity at hand, the estimate \eqref{eq T_1} in Proposition~\ref{prop w} remains valid in two dimensions up to a change of constants.
Therefore,
\begin{equation}\label{eq error I}
	\max_{\mathbf{x} \in I_1} \left\{ |\nabla w(\mathbf{x})|, |w(\mathbf{x})| \right\} \leq T(\varepsilon),
\end{equation}
where $T(\varepsilon)$ has the same functional form as in \eqref{eq T_1}, but possibly with different a~priori constants.


\subsection{Two-sided estimates for the leading ATD term}

We first derive an upper bound for the left-hand side of \eqref{eq II2d}.

\begin{prop}\label{prop up2d}
Assume that the parameter $\tau$ in the CGO solution satisfies
\begin{equation}\label{eq 639}
\tau > \max(1,k).
\end{equation}
Then there exist constants $C'>0$ and $C''>0$, depending only on the a~priori parameters, such that
\begin{equation}\label{eq up7}
\begin{aligned}
&\left| \int_{\tilde{I_1} \cup \tilde{I_2}} u'_N \partial_{\nu} u_{0} \, \mathrm{d}\sigma
 - \int_{\tilde{I_2}} \partial_\nu u'_N u_0 \, \mathrm{d}\sigma \right|\\
&\leq C' T(\varepsilon)\,h 
 + C'' \left( \frac{1}{\tau^{N+1}} + \frac{1}{\tau^{N+2}}
 + e^{-\alpha^* \tau h} + \frac{e^{-\alpha^* \tau h}}{\tau}
 + \tau h^{1/2} e^{-\alpha^* \tau h} \right).
\end{aligned}
\end{equation}
\end{prop}

\begin{proof}
We estimate the right-hand side of \eqref{eq II2d} term by term.
We first treat the error terms involving $w:=u-u'$ on $I_1$, and then bound the remaining terms involving $u'_N$, $\delta u'_{N+1}$, and the CGO solution $u_0$.

Since $|u_0(\mathbf{x})|\leq 1$ on $I_1$, the error terms on $I_1$ can be controlled by \eqref{eq error I}.
By the Cauchy--Schwarz inequality and the fact that $\sigma(I_1)=h$, we obtain
\begin{equation}\label{eq 632}
\begin{aligned}
\left| \int_{I_1} u_{0} \partial_{\nu} (u' - u) \, \mathrm{d} \sigma \right|
&\leq \sqrt{\sigma(I_1)} \left( \int_{I_1} |u_0 \partial_\nu (u-u')|^2 \, \mathrm{d} \sigma \right)^{1/2} \\
&\leq \sqrt{h}\,\|u_0\|_{L^\infty(I_1)} \left( \int_{I_1} |\partial_\nu w|^2 \, \mathrm{d}\sigma \right)^{1/2} \\
&\leq \sqrt{h}\,\|u_0\|_{L^\infty(I_1)} \sqrt{h}\,\|\nabla w\|_{L^\infty(I_1)} \\
&\leq h \|\nabla w\|_{L^\infty(I_1)} \\
&\leq C_1 T(\varepsilon)\,h .
\end{aligned}
\end{equation}
Similarly,
\begin{equation}\label{eq 633}
\begin{aligned}
\left| \int_{I_1} \eta (u' - u) u_{0} \, \mathrm{d} \sigma \right|
&\leq \|\eta\|_{L^\infty(I_1)} \sqrt{\sigma(I_1)} \left( \int_{I_1} |u_0 (u-u')|^2 \, \mathrm{d} \sigma \right)^{1/2} \\
&\leq \|\eta\|_{L^\infty(I_1)} \sqrt{h}\,\|u_0\|_{L^\infty(I_1)} \left( \int_{I_1} |w|^2 \, \mathrm{d}\sigma \right)^{1/2} \\
&\leq \|\eta\|_{L^\infty(I_1)} \sqrt{h}\,\|u_0\|_{L^\infty(I_1)} \sqrt{h}\,\|w\|_{L^\infty(I_1)} \\
&\leq M_0 h \|w\|_{L^\infty(I_1)} \\
&\leq C_2 T(\varepsilon)\,h .
\end{aligned}
\end{equation}

We now turn to the remaining terms in \eqref{eq II2d}, which involve $u'_N$, $\delta u'_{N+1}$, and $u_0$.
We work in the planar test domain $Q_h$ and fix $\theta_0=\frac{\pi}{2}$.
We choose directions $\mathbf d\in\mathcal K_{\alpha^*}$ so that \eqref{eq ec condition} and \eqref{eq alpha} hold.
Then
\[
|u_0(\mathbf{x})|\leq e^{-\alpha^*\tau|\mathbf{x}|}
\qquad \text{for } \mathbf{x}\in Q_h,
\]
and this decay will be used throughout the following estimates.

Using the explicit form of $u'_N$ in \eqref{eq 2D_uN} together with the decay of $u_0$, we obtain
\begin{equation}\label{eq 610}
\left| \int_{I_1}\eta u'_N u_{0} \, \mathrm{d} \sigma \right|
\leq C_3 \left( \frac{1}{\tau^{N+1}} + \frac{1}{\tau} e^{-\alpha^*\tau h} \right).
\end{equation}
Similarly, using \eqref{eq 2D_delta}, we have
\begin{equation}\label{eq 611}
\left| \int_{I_1} \eta \delta u'_{N+1} u_{0} \, \mathrm{d} \sigma \right|
\leq C_4 \left( \frac{1}{\tau^{N+2}} + \frac{1}{\tau} e^{-\alpha^*\tau h} \right).
\end{equation}

For the term without the impedance parameter, we use \eqref{eq 2D_delta} together with \eqref{eq 639} to get
\begin{equation}\label{eq 612}
\left| \int_{I_1\cup I_2} \delta u'_{N+1} \partial_{\nu} u_{0} \, \mathrm{d} \sigma \right|
\leq C_5 \left( \frac{1}{\tau^{N+1}} + \frac{1}{\tau} e^{-\alpha^*\tau h} \right).
\end{equation}

For the integrals over the rays $I_1'$ and $I_2'$, a direct Laplace transform computation yields
\begin{equation}\label{eq 613}
\left| \int_{I_1'\cup I_2'} u'_N \partial_\nu u_0 \, \mathrm{d}\sigma \right|
\leq C_6 e^{-\alpha^* \tau h}.
\end{equation}
We also have
\begin{equation}\label{eq 614}
\left| \int_{I'_2} \partial_\nu u'_N u_0 \, \mathrm{d}\sigma \right|
\leq C_7 \frac{1}{\tau} e^{-\alpha^*\tau h}.
\end{equation}

Using the two-dimensional representation
\[
\nabla(\delta u'_{N+1})
=
\frac{\partial(\delta u'_{N+1})}{\partial r}\,\boldsymbol e_r
+
\frac{1}{r}\frac{\partial(\delta u'_{N+1})}{\partial \theta}\,\boldsymbol e_\theta
\]
together with \eqref{eq 2Dexpansion}, we obtain
\begin{equation}\label{eq 615}
\left| \int_{I_2} \partial_\nu(\delta u'_{N+1}) u_0 \, \mathrm{d}\sigma \right|
\leq C_8 \left( \frac{1}{\tau^{N+1}} + \frac{1}{\tau} e^{-\alpha^*\tau h} \right).
\end{equation}

Finally, for the boundary term over $I_3$, we obtain
\begin{equation}\label{eq 617}
\left| \int_{I_3} \left( u_{0} \partial_{\nu} u^{\prime} - u^{\prime} \partial_{\nu} u_{0} \right) \, \mathrm{d} \sigma \right|
\leq C_9 \tau h^{1/2} e^{-\alpha^* \tau h}.
\end{equation}

All constants $C_i>0$ above depend only on the a~priori parameters.
Substituting the above estimates into \eqref{eq II2d} and absorbing the finitely many constants into new constants $C'>0$ and $C''>0$, we obtain \eqref{eq up7}.
This completes the proof.
\end{proof}

We next estimate the left-hand side of \eqref{eq II2d} in order to identify the leading contribution for sufficiently large $\tau$.

\begin{prop}\label{prop lower2D}
Suppose that $u'$ is the solution to \eqref{eq main system} associated with $(K',\eta')$.
Let $\tilde{I_1}$ and $\tilde{I_2}$ denote the corresponding rays extending $I_1$ and $I_2$, respectively.
Assume that $u'$ is analytic in $Q_h$ and has vanishing order $N$ at the origin.
If the parameter $\tau$ of $u_0$ satisfies
\begin{equation}\label{eq tau3}
\tau \geq \max(1,k,\tau_1),
\end{equation}
then the leading contribution is of order $\tau^{-N}$.
More precisely,
\begin{equation}\label{eq lower bound2}
\left|
\int_{\tilde{I_1}\cup\tilde{I_2}} u'_N \partial_{\nu}u_{0}\, \mathrm{d}\sigma
-
\int_{\tilde{I_2}} \partial_\nu u'_N u_0 \,\mathrm{d}\sigma
\right|
\geq C_* \frac{1}{\tau^{N}},
\end{equation}
where one may take
\[
C_*
=
\frac{k^N}{2^{N+1} N!}\,|b_N-a_N|.
\]
In particular, $C_*$ depends on $k$ and on the leading coefficients $a_N,b_N$ in the local expansion of $u'$ at the test point.
\end{prop}

\begin{proof}
To establish \eqref{eq lower bound2}, we evaluate the ray integrals in \eqref{eq II'} for the two-dimensional test domain, with $\theta_0\in(0,\frac{\pi}{2}]$ fixed.
Denote the exterior unit normals by $\nu_4=(0,-1)$ on $\tilde{I_1}$ and $\nu_5=(-\sin\theta_0,\cos\theta_0)$ on $\tilde{I_2}$.

We use the CGO solution \eqref{eq CGO} with $\mathbf d,\mathbf d^\perp\in\mathbb S^1$ given by
\begin{equation}\label{eq d}
\mathbf d=(\cos\psi,\sin\psi)^\intercal,\quad
\mathbf d^\perp=(-\sin\psi,\cos\psi)^\intercal,\quad
\psi\in\bigl(\theta_0+\tfrac{\pi}{2},\tfrac{3\pi}{2}\bigr),
\end{equation}
which satisfies the geometric condition \eqref{eq ec condition} in two dimensions.
A direct computation yields
\begin{equation}\label{eq nu45}
\partial_{\nu_4}u_0=\tau G(\psi)u_0,\qquad
\partial_{\nu_5}u_0=\tau H(\psi,\theta_0)u_0,
\end{equation}
where
\begin{equation}\label{eq GH}
G(\psi)=-\sin\psi-\mathrm{i}\vartheta\cos\psi,\qquad
H(\psi,\theta_0)=\sin(\psi-\theta_0)+\mathrm{i}\vartheta\cos(\psi-\theta_0),
\end{equation}
and
\[
\vartheta=\sqrt{1+\frac{k^2}{\tau^2}}=1+\mathcal O(\tau^{-2})
\qquad\text{as }\tau\to\infty.
\]

Since $u'$ has vanishing order $N$ at the origin, its first nonzero homogeneous term is given by \eqref{eq 2D_uN}, namely
\[
u'_N(r,\theta)
=
\left(a_N e^{\mathrm{i}N\theta}+b_N e^{-\mathrm{i}N\theta}\right)
\frac{\mathrm{i}^N k^N}{2^N N!}\,r^N,
\]
with constants $a_N,b_N$ satisfying $|a_N|+|b_N|\neq 0$.
Using
\[
\nabla u'_N
=
\frac{\partial u'_N}{\partial r}\,\mathbf e_r
+
\frac{1}{r}\frac{\partial u'_N}{\partial\theta}\,\mathbf e_\theta,
\]
one finds the exterior normal derivative on $\tilde{I_2}$:
\begin{equation}\label{eq nu5}
\partial_{\nu_5}u'_N(r,\theta_0)
=
p_3\,r^{N-1},\qquad
p_3
=
\frac{\mathrm{i}^N k^N}{2^N (N-1)!}
\left(\mathrm{i}a_N e^{\mathrm{i}N\theta_0}-\mathrm{i}b_N e^{-\mathrm{i}N\theta_0}\right).
\end{equation}
Similarly, we set
\[
p_1
=
(a_N+b_N)\frac{\mathrm{i}^N k^N}{2^N N!},\qquad
p_2
=
\left(a_N e^{\mathrm{i}N\theta_0}+b_N e^{-\mathrm{i}N\theta_0}\right)\frac{\mathrm{i}^N k^N}{2^N N!}.
\]

By \eqref{eq nu45}--\eqref{eq nu5} and a direct Laplace transform computation, the leading contribution to the left-hand side of \eqref{eq lower bound2} is of the form
\[
\frac{1}{\tau^{N}}
\left(
\frac{G q_1}{L^{N+1}}
+
\frac{H q_2}{J^{N+1}}
-
\frac{q_3}{J^{N}}
\right),
\]
where we use that $\Re(\rho\cdot\hat{\mathbf{x}}_j)<0$ along the rays by \eqref{eq ec condition}.
We set $q_i=N!p_i$ for $i=1,2$ and $q_3=(N-1)!\,p_3$, that is,
\[
(q_1,q_2,q_3)
=
\frac{\mathrm{i}^N k^N}{2^N}
\left(
a_N+b_N,\ 
a_N e^{\mathrm{i}N\theta_0}+b_N e^{-\mathrm{i}N\theta_0},\ 
\mathrm{i}a_N e^{\mathrm{i}N\theta_0}-\mathrm{i}b_N e^{-\mathrm{i}N\theta_0}
\right),
\]
and
\[
J(\psi,\theta_0)
:=
-\cos(\theta_0-\psi)-\mathrm{i}\vartheta\sin(\theta_0-\psi),\qquad
L(\psi)
:=
-\cos\psi+\mathrm{i}\vartheta\sin\psi.
\]
For $\psi\in(\theta_0+\frac{\pi}{2},\frac{3\pi}{2})$ we define
\begin{equation}\label{eq G}
\mathcal G(\psi)
:=
\frac{G q_1}{L^{N+1}}
+
\frac{H q_2}{J^{N+1}}
-
\frac{q_3}{J^{N}}.
\end{equation}
Since $G(\psi)\neq 0$ and $H(\psi,\theta_0)\neq 0$ on this interval, the denominators in \eqref{eq G} are well defined.

Let $\mathcal G_\infty(\psi)$ denote the expression in \eqref{eq G} with $\vartheta=1$.
As $\tau \to \infty$ and $\vartheta \to 1$, the parameters in \eqref{eq G} satisfy
\[
G=\mathrm{i}L,\qquad
J=\mathrm{i}H.
\]
Hence
\[
\mathcal G_\infty(\psi)
=
\frac{\mathrm{i}q_1}{L^N}
+
\frac{q_2}{\mathrm{i}^{N+1}H^N}
-
\frac{q_3}{\mathrm{i}^N H^N}.
\]
Using Euler's formula, one has $L=-e^{-\mathrm{i}\psi}$ and $H=\mathrm{i}e^{-\mathrm{i}(\psi-\theta_0)}$.
A direct calculation then gives
\begin{equation}\label{eq finG}
\mathcal G_\infty(\psi)
=
\frac{(-1)^N k^N \mathrm{i} e^{\mathrm{i}N\psi}}{2^N N!}\,(b_N-a_N).
\end{equation}
Since $\vartheta=\sqrt{1+k^2/\tau^2}\to1$ as $\tau\to\infty$, the quantity $\mathcal G(\psi)$ depends continuously on $\vartheta$ and $\psi$.
Hence $\mathcal G(\psi)\to\mathcal G_\infty(\psi)$ uniformly for $\psi\in(\theta_0+\frac{\pi}{2},\frac{3\pi}{2})$.
Therefore, there exists $\tau_1$ such that for all $\tau\ge\tau_1$ one has
\[
|\mathcal G(\psi)|\ge \frac12|\mathcal G_\infty(\psi)|
\]
uniformly in $\psi$.
Since $|\mathrm{e}^{\mathrm{i}N\psi}|=1$, \eqref{eq finG} shows that $|\mathcal G_\infty(\psi)|$ is independent of $\psi$.
We may therefore take
\begin{equation}\label{eq nonde2d}
C_*
:=
\frac{k^N}{2^{N+1} N!}\,|b_N-a_N|,
\end{equation}
which yields \eqref{eq lower bound2}.
This completes the proof.
\end{proof}

The following remark explains the role of Proposition~\ref{prop lower2D} in the later non-degeneracy analysis.

\begin{rem}\label{rem:2D-first-obstruction}
Proposition~\ref{prop lower2D} shows that, in the planar ATD construction, the first leading quantity is governed by the coefficient $C_*$ in \eqref{eq nonde2d}, namely by the combination $b_N-a_N$.
If $b_N-a_N\neq0$, then the left-hand side of \eqref{eq lower bound2} is genuinely of order $\tau^{-N}$, and the ATD extraction is already non-degenerate at the first step.
If $b_N-a_N=0$, then this first leading quantity vanishes, but this does not imply full degeneracy of the ATD extraction.
In that case, the extraction must be continued to higher orders.

The constant $C_*$ in \eqref{eq lower bound2} is the concrete realization, in the present two-dimensional bounded-impedance setting, of the theorem-level leading ATD coefficient $C_A$ in Theorem~\ref{th relation}.
Its explicit form depends on the regime.
Under different boundary conditions, the corresponding leading ATD coefficient is replaced by a different regime-dependent combination, as explained in Remark~\ref{rem:relation-framework}.
The passage from the vanishing of the first leading quantity to successive higher-order vanishings is formulated abstractly later in Proposition~\ref{prop:ATD-degeneracy}.
\end{rem}

\begin{proof}[Proof of Theorem~\ref{th relation} in the two-dimensional case]
We work in the planar setting $n=2$ under Assumption~\eqref{eq contr 2}, so that the modified Hausdorff distance $\mathfrak d=\tilde{\mathrm d}(K,K')$ is strictly positive and the test domain $Q_h$ is constructed as in Section~\ref{subsec ATD2D}, with $h=c_2\mathfrak d$ for some fixed $c_2\in(0,1)$.

By Proposition~\ref{prop decom} and Lemma~\ref{lem vanish}, the total field $u'$ associated with $(K',\eta')$ is analytic in $Q_h$ and has finite vanishing order $N$ at the test point $0$.
We decompose
\[
u'=u'_N+\delta u'_{N+1}
\]
in $Q_h$ as in \eqref{eq 2D_uN}--\eqref{eq 2D_delta}.
The corresponding integral identity \eqref{eq II2d} holds for this decomposition.

The propagation-of-smallness argument from Section~\ref{sec 4} applies in the present two-dimensional configuration, with the only genuinely two-dimensional input given by the local H\"older regularity in \cite[Lemma~3.2]{diao2024stable}.
As in Proposition~\ref{prop w}, this yields
\begin{equation}\label{eq error II}
\max_{\mathbf{x}\in I_1}\bigl\{|\nabla w(\mathbf{x})|,\ |w(\mathbf{x})|\bigr\}\le T(\varepsilon),
\end{equation}
where $w:=u-u'$ and $T(\varepsilon)$ has the same functional form as in \eqref{eq T_1}, up to a different multiplicative constant depending only on the a~priori parameters.

Using \eqref{eq error II} in the boundary identity \eqref{eq II2d}, Proposition~\ref{prop up2d} gives, for all
\[
\tau>\max(1,k),
\]
the upper bound
\begin{equation}\label{eq up7-2D}
\begin{aligned}
&\left| \int_{\tilde{I_1} \cup \tilde{I_2}} u'_N \partial_{\nu} u_{0} \, \mathrm{d} \sigma
       - \int_{\tilde{I_2}} \partial_\nu u'_N u_0 \, \mathrm{d} \sigma \right|\\
&\leq C_1 T(\varepsilon) h
 + C_2 \left( \frac{1}{\tau^{N+1}} + \frac{1}{\tau^{N+2}}
           + e^{-\alpha^*\tau h} + \frac{e^{-\alpha^*\tau h}}{\tau}
           + \tau h^{1/2} e^{-\alpha^* \tau h} \right),
\end{aligned}
\end{equation}
where $C_1,C_2>0$ depend only on the a~priori parameters, once the auxiliary ATD geometry has been fixed.

On the other hand, Proposition~\ref{prop lower2D} identifies the first extracted leading contribution of the ray integrals.
In the present two-dimensional bounded-impedance setting, the theorem-level leading ATD coefficient $C_A$ in Theorem~\ref{th relation} is realized by the planar coefficient $C_*$ in \eqref{eq nonde2d}.
Accordingly, Proposition~\ref{prop lower2D} yields
\begin{equation}\label{eq lower2D-relation}
\left|\int_{\tilde{I_1}\cup\tilde{I_2}} u'_N \partial_{\nu}u_{0}\, \mathrm{d}\sigma
      - \int_{\tilde{I_2}} \partial_\nu u'_N u_0 \,\mathrm{d}\sigma\right|
\geq \frac{|C_A|}{\tau^{N}},
\end{equation}
for all
\[
\tau\geq \max(1,k,\tau_1).
\]
Here $C_A$ is independent of $\tau$, depends only on the local jet of $u'$ at the test point, and may vanish in degenerate configurations.

Comparing \eqref{eq up7-2D} with \eqref{eq lower2D-relation}, we obtain
\begin{equation}\label{eq 634}
\frac{|C_A|}{\tau^{N}}
\leq C_1' T(\varepsilon) h
  + C_2' \left( \frac{1}{\tau^{N+1}} + \frac{1}{\tau^{N+2}}
           + e^{-\alpha^*\tau h} + \frac{e^{-\alpha^*\tau h}}{\tau}
           + \tau h^{1/2} e^{-\alpha^* \tau h} \right),
\end{equation}
for suitable constants $C_1',C_2'>0$ depending only on the a~priori parameters.

We now estimate the exponentially decaying terms using
\begin{equation}\label{eq exp decay}
e^{-t}\leq \frac{1}{t},
\qquad
e^{-t}\leq \frac{m!}{t^{m}}
\quad\text{for all } t>0,\ m\in\mathbb{N}.
\end{equation}
Applying \eqref{eq exp decay} with $t=\alpha^*\tau h$, and using in addition $\tau>1$ and $h<1$, the exponential terms in \eqref{eq 634} can be absorbed into the algebraic remainder terms.
Thus,
\begin{equation}\label{eq 635}
|C_A|
\leq C \Bigl( \tau^{N} T(\varepsilon) h
         + \tau^{-1} h^{-(N+2)} \Bigr),
\end{equation}
for some constant $C>0$ depending only on the a~priori parameters.

We now choose $\tau$ by balancing the two terms on the right-hand side of \eqref{eq 635}.
Solving
\[
\tau^{N} T(\varepsilon) h \sim \tau^{-1} h^{-(N+2)}
\]
gives
\begin{equation}\label{eq tau-balance}
\tau=\tau_e:=T(\varepsilon)^{-\frac{1}{N+1}}\,h^{-\frac{N+3}{N+1}}.
\end{equation}
Substituting $\tau=\tau_e$ into \eqref{eq 635}, and using that the two contributions are then of the same order, we obtain
\[
|C_A|
\leq C\,T(\varepsilon)^{\frac{1}{N+1}}\,
        h^{-\frac{N^2+2N-1}{N+1}}.
\]
Equivalently,
\begin{equation}\label{eq relation-2D-local}
|C_A|\,h^{p}
\leq C\,T(\varepsilon)^{\frac{1}{N+1}},
\qquad
p:=\frac{N^2+2N-1}{N+1}.
\end{equation}

By the three-sphere iteration in Section~\ref{sec 4}, the boundary error modulus $T(\varepsilon)$ admits a double-logarithmic upper bound of the form
\[
T(\varepsilon)\leq C_T\,(\log|\log(1/\varepsilon)|)^{-\alpha},
\]
for some $\alpha>0$ depending only on the a~priori parameters.
Inserting this into \eqref{eq relation-2D-local} yields
\[
|C_A|\,h^{p}
\leq C\,(\log|\log(1/\varepsilon)|)^{-\kappa_0},
\qquad
\kappa_0:=\frac{\alpha}{N+1}>0.
\]

Finally, recalling that $h=c_2\mathfrak d$ with $c_2\in(0,1)$, and using Lemmas~\ref{lem dm} and~\ref{lem bdry-vs-set} to identify $\mathfrak d$ with the boundary Hausdorff distance $\mathrm d_H(K,K')$ up to multiplicative constants, we obtain
\[
|C_A|\,(\mathrm d_H(K,K'))^{p}
\leq C\,(\log|\log(1/\varepsilon)|)^{-\kappa_0}.
\]
Since $D=K$ and $D'=K'$ in two dimensions, this is precisely the far-field--geometry relation \eqref{eq relation main} in the planar case.

This completes the proof of Theorem~\ref{th relation} for $n=2$.
\end{proof}

\section{The Three-Dimensional Proof of Theorem~\ref{th relation}}\label{sec proof}

In this section, we introduce the ATD in the three-dimensional case in order to construct a parameterized integral domain suited to microlocal analysis.

\subsection{The ATD construction in the three-dimensional case}\label{subsec ATD3D}

Let $\Sigma$ and $\Sigma'$ be admissible polyhedral obstacles as in Definition~\ref{defn admissible class}, with corresponding admissible impedance parameters $\eta$ and $\eta'$.
For a fixed incident direction $\mathbf{p}\in\mathbb{S}^2$ and wavenumber $k>0$, let $u\in H^1(\Omega\setminus\Sigma)$ and $u'\in H^1(\Omega\setminus\Sigma')$ denote the solutions to \eqref{eq main system} associated with $(\Sigma,\eta)$ and $(\Sigma',\eta')$, respectively.
Let
\[
\mathfrak{d}=\tilde{\mathrm{d}}(\Sigma,\Sigma')
\]
be the modified Hausdorff distance defined in \eqref{eq d_M}, which vanishes if and only if $\Sigma=\Sigma'$ as subsets of $\mathbb{R}^3$.
Under the far-field error $\varepsilon>0$, we assume
\begin{equation}\label{eq contradiction}
\mathfrak{d}>0.
\end{equation}

By the definition of $\tilde{\mathrm d}$ and the discussion in Subsection~\ref{subsec DRO}, after possibly exchanging $\Sigma$ and $\Sigma'$, there exists a point $\mathbf{y}_0\in \partial\Sigma\cap\partial\mathbf{G}_2$ contained in the relative interior of a planar face of $\Sigma$ such that
\[
B_{\mathfrak{d}}(\mathbf{y}_0)\cap\Sigma'=\emptyset.
\]
Let $\Pi$ denote that face.
Choose $h\in(0,\mathfrak{d})$ so that
\[
(B_h(\mathbf{y}_0)\cap\partial\Sigma)\subset \Pi,
\]
equivalently $h=c_1\,\mathfrak{d}$ with some $c_1\in(0,1)$.
By Lemma~\ref{lem vanish}, the field $u'$ is real-analytic in $B_h(\mathbf{y}_0)$ and has finite vanishing order at $\mathbf{y}_0$.
We construct in $B_h(\mathbf{y}_0)\setminus\Sigma$ a parameterized test domain that serves as the localization device for the microlocal estimates below.

Apply a rigid motion that sends $\mathbf{y}_0$ to the origin and aligns $\Pi$ with the plane $\{x_3=0\}$.
Set
\[
B_h:=\{\mathbf{x}\in\mathbb{R}^3:\ |\mathbf{x}|\leq h\},
\qquad
\mathbb{B}_h^{+}:=B_h\cap\{x_1\ge0,\ x_2\ge0,\ x_3\ge0\}.
\]
Define
\[
S_h:=B_h\cap\partial\Sigma\subset\Pi,
\]
where $u$ satisfies the impedance boundary condition.
We define a domain $P_h\subset\mathbb{B}_h^{+}$ with $S_h\subset\partial P_h$ and boundary decomposition
\[
\partial P_h=S_1\cup S_2\cup S_3\cup S_4.
\]
We take $S_1\subset\{x_3=0\}$, $S_2\subset\{x_2=0\}$, and $S_4\subset\partial B_h$.
The remaining planar face $S_3$ forms an angle
\begin{equation}\label{eq phi_0}
\phi_0\in(0,\tfrac{\pi}{2}]
\end{equation}
with the $x_1x_3$-plane.
The patches $S_i$ for $i=1,2,3$ are fan-shaped, with opening angle $\phi_0$ for $S_1$ and opening angle $\frac{\pi}{2}$ for $S_2$ and $S_3$.
Let $\tilde S_i$ denote the full sectors swept by the boundary rays of $S_i$ for $i=1,2,3$, and define
\[
S_i':=\tilde S_i\setminus S_i,
\qquad i=1,2,3.
\]
This parameterized geometry fixes the notation and the angle $\phi_0$ for the local estimates on $P_h$ and its boundary components.
This geometry is illustrated in Figure~\ref{fig:quarter_circle3D}, where \eqref{eq S} is shown in the representative case $\phi_0=\frac{\pi}{2}$.

\begin{figure}[htbp]
    \centering
    \includegraphics[width=0.6\textwidth, keepaspectratio]{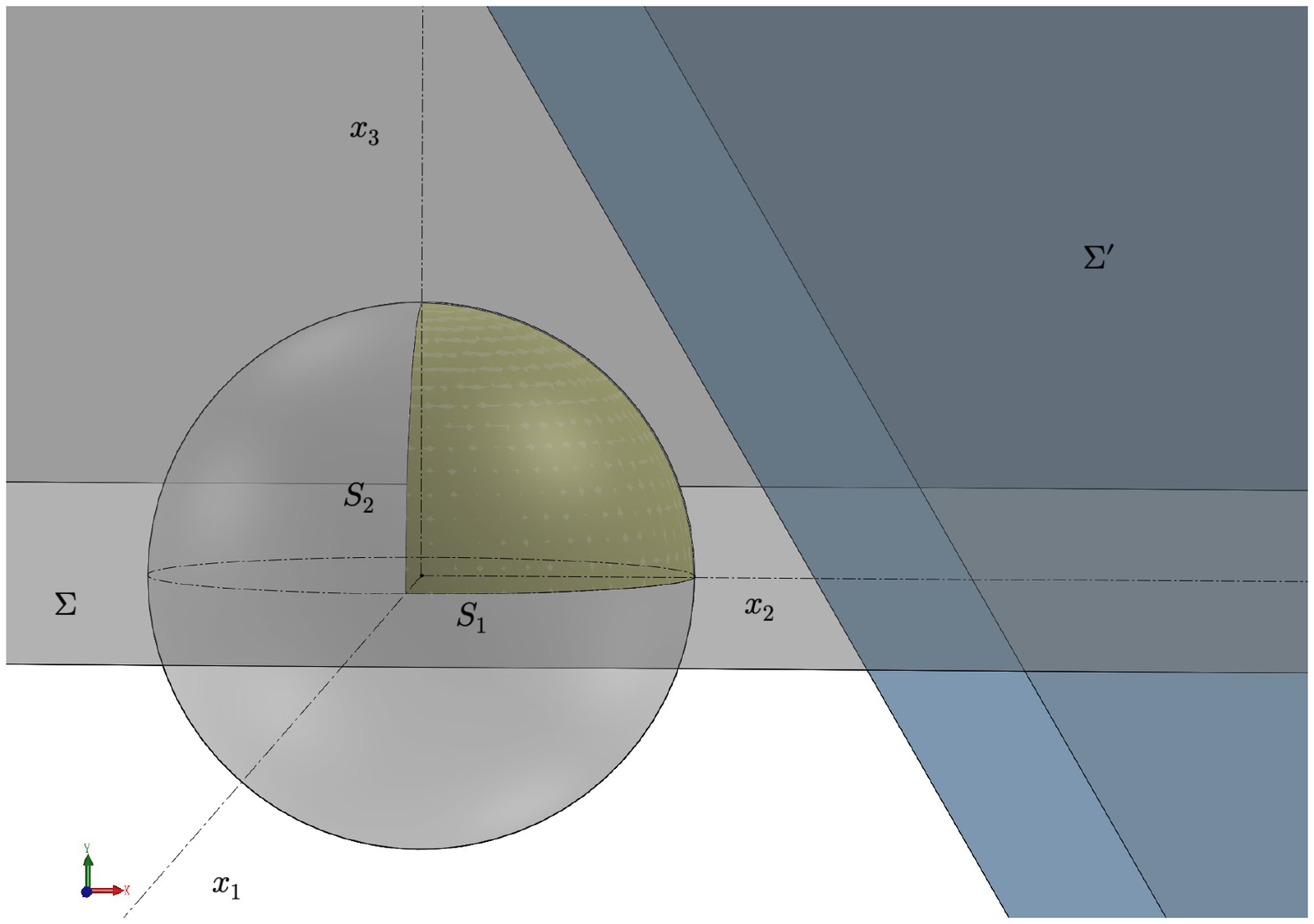}
    \caption{The three-dimensional ATD geometry in the representative case $\phi_0=\frac{\pi}{2}$.}
    \label{fig:quarter_circle3D}
\end{figure}

The test domain $P_h$, characterized by the parameter $\phi_0$ and the test point $\mathbf{y}_0$, is flexible enough to adapt to the microlocal estimates required in our analysis.
This flexibility is used later, in particular in the proof of Proposition~\ref{prop low2}, to avoid degenerate configurations by adjusting the local test geometry around $\mathbf{y}_0$.

We next present the integral identity in the test domain $P_h$.
This identity is the starting point of the three-dimensional ATD analysis.
It links the far-field discrepancy to a localized boundary integral and is built on the CGO solution defined in \eqref{eq CGO}.

We now explain a dimension-dependent point in the local analysis.
In the two-dimensional case treated in Section~\ref{sec 2D}, the ATD localization can be carried out at a test point where the total field $u'$ has arbitrary finite vanishing order.
In the three-dimensional case, by contrast, we present the detailed argument under the condition
\[
u'(\mathbf{y}_0)\neq 0.
\]
Equivalently, the vanishing order of $u'$ at $\mathbf{y}_0$ is zero.
Under this condition, the local expansion of $u'$ starts at order zero, and the leading term can be extracted explicitly from the corresponding integral identity and CGO solutions.

The three-dimensional ATD argument is not restricted to vanishing order zero.
After suitable modifications of the local expansion and the coefficient extraction, analogous results can also be obtained for general vanishing order.
However, the corresponding analysis is substantially more technical.
For this reason, we restrict the detailed presentation in the present paper to the case \(u'(\mathbf{y}_0)\neq 0\).

This condition is also natural in certain physically relevant situations.
For instance, when \(k\,\operatorname{diam}(D')\ll 1\), the scattered field is expected to remain small compared with the incident field.
In that case, the total field is dominated by the incident wave and is therefore non-vanishing at points outside \(D'\).
We do not use such sufficient conditions in the analysis below; rather, the condition \(u'(\mathbf{y}_0)\neq 0\) is taken as part of the present three-dimensional setup.

We first present the bounded-impedance case $\eta\in\Xi_1$, which serves as the main model derivation for the ATD argument.

\begin{prop}\label{prop II2}
Let $u \in H^1(\Omega \setminus \Sigma)$ and $u' \in H^1(\Omega \setminus \Sigma')$ be the solutions to \eqref{eq main system} associated with $(\Sigma,\eta)$ and $(\Sigma',\eta')$, respectively.
Assume that $u'$ is real-analytic in $P_h$ and satisfies
\[
u'(0)\neq 0,
\]
so that, in the case $N=0$, we may write
\begin{equation}\label{eq 3D-decomp-N0}
u'(\mathbf{x})=u'(0)+\delta u'_1(\mathbf{x})
\qquad\text{in }P_h,
\end{equation}
where $\delta u'_1$ is given by Proposition~\ref{prop decom}.
Then, under \eqref{eq contradiction}, the following integral identity holds in $P_h$:
\begin{equation}\label{eq integral identity2}
\begin{aligned}
0
&=-\int_{S_1\cup S_2\cup S_3} u'(0)\,\partial_{\nu} u_0 \,\mathrm{d}\sigma \\
&\quad + \int_{S_1} u_0 \,\partial_{\nu}(u' - u) \,\mathrm{d}\sigma 
  + \int_{S_1} \eta (u' - u)\,u_0 \,\mathrm{d}\sigma 
  - \int_{S_1} \eta\,u'(0)\,u_0 \,\mathrm{d}\sigma \\
&\quad - \int_{S_1} \eta\,\delta u'_1\,u_0 \,\mathrm{d}\sigma 
  - \int_{S_1 \cup S_2 \cup S_3} \delta u'_1\,\partial_{\nu} u_0 \,\mathrm{d}\sigma \\
&\quad + \int_{S_2 \cup S_3} \partial_\nu\delta u'_1\,u_0 \,\mathrm{d}\sigma 
  + \int_{S_4} \bigl( u_0 \,\partial_{\nu} u' - u' \,\partial_{\nu} u_0 \bigr) \,\mathrm{d}\sigma.
\end{aligned}
\end{equation}
Here and in what follows, $\nu$ denotes the exterior unit normal to $P_h$ on $\partial P_h$.
\end{prop}
\subsection{Two-sided estimates for the leading ATD term}

Let $\eta,\eta'\in\Xi_1$ be admissible impedance parameters in the sense of Definition~\ref{defn eta}, and let $u,u'$ be the solutions to \eqref{eq main system} associated with $(\Sigma,\eta)$ and $(\Sigma',\eta')$, respectively.  
We assume that $\Sigma,\Sigma'$ satisfy \eqref{eq contradiction}.  
Fix an incident plane wave $u^i$ with direction $\mathbf p\in\mathbb S^2$ and wavenumber $k>0$.  
Suppose that for some sufficiently small $0<\varepsilon<\varepsilon_m$ the far-field discrepancy satisfies \eqref{eq far-field error}.

The smallness propagation argument from the far-field to the boundary of the obstacle, developed in Section~\ref{sec 4} under a uniform framework, then yields a near-field control on $S_1$ as stated in Proposition~\ref{prop w}.  
More precisely, if $w$ is defined by \eqref{eq w}, we have
\begin{equation}\label{eq S_1}
	\max_{\mathbf x\in S_1}\{\,|\nabla w(\mathbf x)|,\ |w(\mathbf x)|\,\}\leq T(\varepsilon).
\end{equation}

To derive the stability estimate, we use \eqref{eq S_1} as input in the integral identity \eqref{eq integral identity2}.  
This produces an inequality that couples the error term $T(\varepsilon)$ with the geometric parameter $h$ and the CGO parameter $\tau$, and forms the starting point for the quantitative ATD relation in three dimensions.

We first derive the upper bound for the three-dimensional ATD boundary integral in \eqref{eq integral identity2}.
The matching lower bound will be established afterwards.

\begin{prop}\label{prop up1}
Assume that \eqref{eq contradiction} and the far-field smallness \eqref{eq far-field error} hold with $0<\varepsilon<\varepsilon_m$, and that $\tau>\max(1,k)$.
Then the ATD boundary integral in \eqref{eq integral identity2} satisfies the estimate
\begin{equation}\label{eq upper bound}
\begin{aligned}
&\left|\int_{\tilde{S_1}}\eta u'(0)u_0\,\mathrm{d}\sigma
+\int_{\tilde{S_1}\cup \tilde{S_2}\cup \tilde{S_3}} u'_1 \partial_{\nu} u_{0} \, \mathrm{d} \sigma
-\int_{\tilde{S_2}\cup \tilde{S_3}}\partial_\nu u'_1 u_0\, \mathrm{d}\sigma\right| \\
&\leq C^{*} T(\varepsilon)\, h^2
 + C^{**} \left( \frac{1}{\tau^{3}}
 + e^{-\alpha'\tau h} + \frac{1}{\tau}e^{-\alpha'\tau h}
 + \tau h \, e^{-\alpha'\tau h} \right),
\end{aligned}
\end{equation}
where $C^{*}, C^{**}>0$ depend only on the a~priori parameters.
\end{prop}

\begin{proof}
We estimate the right-hand side of \eqref{eq integral identity2} term by term.
We first control the error terms on $S_1$ by the boundary estimate \eqref{eq S_1}, and then estimate the remaining terms involving $u'(0)$, $\delta u'_1$, and the CGO solution $u_0$.

By the geometry of the test domain $P_h$, there exist $\alpha'>0$ and a conic set $\mathcal{K}_{\alpha'}\subset\mathbb{S}^2$ such that every direction $\mathbf{d}\in\mathcal{K}_{\alpha'}$ satisfies \eqref{eq ec condition}.
We parametrize $\mathbf{d}$ by spherical angles $(\psi_1,\psi_2)$:
\begin{equation*}
\mathbf{d}=\bigl(\sin\psi_1\cos\psi_2,\ \sin\psi_1\sin\psi_2,\ \cos\psi_1\bigr).
\end{equation*}
We define
\begin{equation}\label{eq Ka}
\mathcal{K}_{\alpha'}:=\{\mathbf{d}(\psi_1,\psi_2):\ \psi_1\in(\tfrac{\pi}{2},\pi),\ \psi_2\in(\phi_0+\tfrac{\pi}{2},\tfrac{3\pi}{2})\}.
\end{equation}
Here $\phi_0$ is the parameter of $P_h$ specified in \eqref{eq phi_0}, and the range of $\psi_2$ is chosen so that \eqref{eq ec condition} holds for every $\mathbf{d}\in\mathcal{K}_{\alpha'}$.
In particular, for $\mathbf{x}$ on the relevant boundary pieces, the CGO solution enjoys the exponential decay induced by \eqref{eq ec condition}; in particular, $|u_0(\mathbf{x})|\leq 1$ there.

By \eqref{eq far-field error} and the propagation-of-smallness argument in Section~\ref{sec 4}, we obtain the boundary estimate \eqref{eq S_1}.
We first estimate the error terms on $S_1$.
By the Cauchy--Schwarz inequality, one has
\begin{equation}\label{eq 532}
\begin{aligned}
\left|\int_{S_1} u_{0}\,\partial_{\nu}(u^{\prime}-u) \, \mathrm{d}\sigma \right|
&\leq \sqrt{\sigma(S_1)} \left( \int_{S_1} |u_0\,\partial_\nu (u-u')|^2 \, \mathrm{d}\sigma \right)^{1/2} \\
&\leq \sigma(S_1)\|\nabla w\|_{L^\infty(S_1)} \\
&\leq \frac{\pi h^2}{4}\,\|\nabla w\|_{L^\infty(S_1)} \\
&\leq C_1\,T(\varepsilon)\,h^2 .
\end{aligned}
\end{equation}
Similarly,
\begin{equation}\label{eq 533}
\begin{aligned}
\left|\int_{S_1} \eta\,(u^{\prime}-u)\,u_{0} \, \mathrm{d}\sigma \right|
&\leq \|\eta\|_{L^\infty(S_1)} \sqrt{\sigma(S_1)} \left( \int_{S_1} |u_0\,(u-u')|^2 \, \mathrm{d}\sigma \right)^{1/2} \\
&\leq M_0\,\sigma(S_1)\| w\|_{L^\infty(S_1)} \\
&\leq M_0\,\frac{\pi h^2}{4}\,\|w\|_{L^\infty(S_1)} \\
&\leq C_2\,T(\varepsilon)\,h^2 .
\end{aligned}
\end{equation}

We next estimate the remaining terms in \eqref{eq integral identity2}.
Since $u'$ is analytic in $P_h$ and satisfies $u'(0)\neq0$, it admits the decomposition
\[
u'(\mathbf{x})=u'(0)+\delta u'_1(\mathbf{x}),
\]
where, by Proposition~\ref{prop decom}, the remainder $\delta u'_1$ can be further decomposed as
\[
\delta u'_1=u'_1+\delta u'_2.
\]
In particular, using \eqref{eq decomposition estimation}, we have
\[
|\delta u'_1(\mathbf{x})|\le C_{\mathrm{dec}}\,r,
\qquad
|\delta u'_2(\mathbf{x})|\le C_{\mathrm{dec}}\,r^2,
\]
for $r\le h$.

A direct computation using \eqref{eq CGO} gives
\begin{equation}\label{eq 3D-u0-const}
\left|\int_{S_1 \cup S_2 \cup S_3} u'(0)\,\partial_{\nu} u_0 \, \mathrm{d}\sigma\right|
\leq C_3\left(\frac{1}{\tau^3}+e^{-\alpha'\tau h}\right).
\end{equation}

For the contribution on $S_1$, using \eqref{eq decomposition estimation}, \eqref{eq gamma}, \eqref{eq gamma2}, and $0<\phi_0\leq\frac{\pi}{2}$, we obtain
\begin{equation}\label{eq 68}
\begin{aligned}
\left| \int_{S_1} \eta \,\delta u'_1\,u_0 \,\mathrm d\sigma \right|
&\leq M_0 C_{\mathrm{dec}} \int_0^{\phi_0} \int_0^h r^{2} e^{-\alpha'\tau r}\,\mathrm dr\,\mathrm d\phi \\
&\le C_4\left(\frac{1}{\tau^{3}}+\frac{1}{\tau}e^{-\alpha'\tau h}\right).
\end{aligned}
\end{equation}

We now bound the terms involving the normal derivative of the CGO solution.
Assume $\tau>k$.
By \eqref{eq decomposition estimation} and the explicit form of $\partial_\nu u_0$, we have
\begin{equation}\label{eq 70}
\begin{aligned}
\left|\int_{S_1\cup S_2\cup S_3} \delta u'_2 \,\partial_{\nu} u_{0} \, \mathrm{d} \sigma\right|
&\leq 3C_{\mathrm{dec}}\frac{\pi}{2}\sqrt{k^2+2\tau^2}\int_0^h r^{3}e^{-\alpha'\tau r}\,\mathrm d r \\
&\leq 3C_{\mathrm{dec}}\frac{\pi}{2}\sqrt{k^2+2\tau^2}
\left( \frac{\Gamma(4)}{(\alpha'\tau)^{4}}+\frac{2}{\tau}e^{-\alpha' \tau h}\right) \\
&\leq C_5\left(\frac{1}{\tau^3}+e^{-\alpha' \tau h}\right).
\end{aligned}
\end{equation}
Likewise, using polar coordinates for $\partial_\nu \delta u'_2$ and \eqref{eq u expansion}, one has
\begin{equation}\label{eq 71}
\begin{aligned}
\left|\int_{S_2\cup S_3}\partial_\nu \delta u'_2\, u_0\, \mathrm{d}\sigma\right|
&\leq 2C_{\mathrm{dec}}\frac{\pi}{2}\int_0^h r^{2}e^{-\alpha'\tau r}\, \mathrm{d}r \\
&\leq 2C_{\mathrm{dec}}\frac{\pi}{2}
\left(\frac{\Gamma(3)}{(\alpha'\tau)^{3}}+\frac{2}{\tau}e^{-\alpha' \tau h} \right) \\
&\leq C_6\left(\frac{1}{\tau^3}+\frac{1}{\tau}e^{-\alpha' \tau h} \right).
\end{aligned}
\end{equation}

For the complementary regions $S'_1$, $S'_2$, and $S'_3$, the exponential damping from \eqref{eq gamma2} yields
\begin{equation}\label{eq 72}
\begin{aligned}
\left|\int_{S_1'\cup S_2'\cup S_3'}u'_1\,\partial_\nu u_0\,\mathrm d\sigma \right|
&\leq 3C_{\mathrm{dec}}\sqrt{k^2+2\tau^2}\frac{2}{\tau}e^{-\alpha' \tau h} \\
&\leq C_7 e^{-\alpha' \tau h}.
\end{aligned}
\end{equation}
Moreover, by direct computation there exists a constant $C_8>0$ such that
\begin{equation}\label{eq 73}
\left|\int_{S'_2\cup S'_3}\partial_\nu u'_1\, u_0\, \mathrm{d}\sigma \right|
\leq C_8 \frac{1}{\tau}e^{-\alpha' \tau h}.
\end{equation}

Finally, we estimate the spherical part $S_4\subset \partial B_h$.
By direct computation, we have
\[
\|u_0\|_{H^1(S_4)}=\sqrt{1+2\tau^2+k^2}\,\|u_0\|_{L^2(S_4)},
\qquad
\|\partial_\nu u_0\|_{L^2(S_4)}\leq \sqrt{2\tau^2+k^2}\,\|u_0\|_{L^2(S_4)}.
\]
Hence
\begin{equation}\label{eq 719}
\begin{aligned}
\left|\int_{S_4} \left( u_{0}\partial_\nu u^{\prime} - u^{\prime}\partial_{\nu} u_{0} \right) \, \mathrm{d} \sigma\right|
&\leq \|u_0\|_{H^{1/2}(S_4)}\|\partial_\nu u'\|_{H^{-1/2}(S_4)}
 + \|u'\|_{L^2(S_4)}\|\partial_\nu u_0\|_{L^2(S_4)} \\
&\leq \|u_0\|_{H^1(S_4)}\|u'\|_{H^1(P_h)}
 + \|u'\|_{H^1(P_h)}\|\partial_\nu u_0\|_{L^2(S_4)} \\
&\leq \|u'\|_{H^1(P_h)}
\left(\|u_0\|_{H^1(S_4)}+\|\partial_\nu u_0\|_{L^2(S_4)}\right) \\
&\leq C\mathcal{E}\sqrt{1+2\tau^2+k^2}\sqrt{\frac{\pi}{2}}\, h\, e^{-\alpha' \tau h} \\
&\leq C_9\,\tau h\, e^{-\alpha' \tau h},
\end{aligned}
\end{equation}
where $\mathcal{E}$ is the uniform bound from \eqref{eq uniform bounded} in Lemma~\ref{lem uniform bound}.

Substituting the above estimates into \eqref{eq integral identity2} and absorbing the finitely many constants into new constants $C^{*}>0$ and $C^{**}>0$, we obtain \eqref{eq upper bound}.
\end{proof}

\begin{prop}\label{prop low2}
Let $u'$ be the solution to \eqref{eq main system} associated with $(\Sigma', \eta')\in(\mathcal{A}, \Xi_1)$.
Assume that the test domain satisfies the geometric conditions
\[
S_1 \subset \partial \Sigma,
\qquad
S_2,S_3 \subset \partial P_h,
\]
that $u'$ is analytic in $P_h$, has vanishing order $0$ at the origin, and that the parameter $\tau$ of $u_0$ satisfies
\begin{equation}\label{eq tau2}
\tau \geq \max(1, k, \tau_0).
\end{equation}
Then the leading contribution to the three-dimensional ATD boundary integral appears at order $\tau^{-2}$.
More precisely, there exists a complex coefficient $C_A$, depending only on the local coefficients $a_0^0,a_1^0,a_1^{\pm1}$ of $u'$ at the test point, on $\eta(0)$, and on the chosen ATD parameters, such that
\begin{equation}\label{eq lower bound1}
\left|\int_{\tilde{S_1}}\eta u'(0)u_0\,\mathrm{d}\sigma
+\int_{\tilde{S_1}\cup \tilde{S_2}\cup \tilde{S_3}} u'_1 \partial_{\nu} u_{0} \, \mathrm{d} \sigma
-\int_{\tilde{S_2}\cup \tilde{S_3}}\partial_\nu u'_1 u_0\, \mathrm{d}\sigma\right|
\geq |C_A|\frac{1}{\tau^{2}}.
\end{equation}
\end{prop}

\begin{proof}
We prove \eqref{eq lower bound1} in two steps.
In Step~I, we parameterize the boundary integrals on the swept planes $\tilde S_1,\tilde S_2,\tilde S_3$.
In Step~II, we identify the leading contribution of order $\tau^{-2}$ and the corresponding coefficient $C_A$.

\medskip
\noindent\textit{Step I: Parameterization of the integrals.}

In the test domain $P_h\subset \mathbb R^3$, we use spherical coordinates $(r,\theta,\phi)$ and fix an arbitrary meridian angle $\phi_0\in(0,\frac{\pi}{2}]$.
We parameterize the three boundary faces and their swept planes as follows:
\begin{equation}\label{eq S}
\begin{aligned}
S_1&=\{\mathbf x=r(\cos\phi,\sin\phi,0):\ 0\leq r\leq h,\ 0\le\phi\le\phi_0\},\\
S_2&=\{\mathbf x=r(\sin\theta,0,\cos\theta):\ 0\leq r\leq h,\ 0\le\theta\le\frac{\pi}{2}\},\\
S_3&=\{\mathbf x=r(\sin\theta\cos\phi_0,\sin\theta\sin\phi_0,\cos\theta):\ 0\leq r\leq h,\ 0\le\theta\le\frac{\pi}{2}\},
\end{aligned}
\end{equation}
and we write $\tilde S_j$ for the corresponding swept planes with $0\leq r<\infty$.
The corresponding outward unit normals are
\begin{equation}\label{eq nu}
\nu_1=(0,0,-1),\quad \nu_2=(0,-1,0),\quad \nu_3=(-\sin\phi_0,\ \cos\phi_0,\ 0).
\end{equation}

We choose unit vectors $\mathbf d,\mathbf d^\perp\in\mathbb S^2$ entering the CGO phase $\rho$ defined by \eqref{eq rho} for the solution $u_0$ as follows:
\begin{equation}\label{eq d dperp}
\mathbf d = (\sin\psi_1\cos\psi_2, \sin\psi_1\sin\psi_2, \cos\psi_1)^\intercal,
\qquad
\mathbf d^\perp = (-\sin\psi_2, \cos\psi_2, 0)^\intercal,
\end{equation}
where the angles $(\psi_1,\psi_2)$ are taken from the admissible set $\mathcal{K}_{\alpha'}$ defined by \eqref{eq Ka}, so that \eqref{eq ec condition} holds for the fixed value $\phi_0$.
We set
\[
\vartheta:=\sqrt{1+\frac{k^2}{\tau^2}}=1+\mathcal{O}(\tau^{-2})\geq 1.
\]
The exterior normal derivatives on $S_1,S_2,S_3$ corresponding to \eqref{eq nu} are given by
\begin{equation}\label{eq 11}
\partial_{\nu_1}u_0 = \tau A(\psi_1,\psi_2)u_0,\qquad
\partial_{\nu_2}u_0 = \tau B(\psi_1,\psi_2)u_0,\qquad
\partial_{\nu_3}u_0 = \tau Z(\phi_0,\psi_1,\psi_2)u_0,
\end{equation}
where
\begin{equation}\label{eq ABC}
\begin{aligned}
A(\psi_1,\psi_2)&=-\cos\psi_1,\\
B(\psi_1,\psi_2)&=-\sin\psi_1\sin\psi_2-\mathrm{i}\vartheta\cos\psi_2,\\
C(\psi_1,\psi_2)&=-\sin\psi_1\cos\psi_2+\mathrm{i}\vartheta\sin\psi_2,
\end{aligned}
\end{equation}
and
\begin{equation}\label{eq Z}
Z(\phi_0,\psi_1,\psi_2):= C(\psi_1,\psi_2)\,\sin\phi_0 - B(\psi_1,\psi_2)\,\cos\phi_0.
\end{equation}

Since $u'$ is analytic in $P_h$ and has vanishing order $0$ at the test point $0$, we may write the first two terms in the decomposition \eqref{eq decomposition} as
\begin{equation}\label{eq Un'}
u'(0)=2\sqrt{\pi}\,a_0^0,
\end{equation}
\begin{equation}\label{eq U1'}
u'_1
=\frac{4\pi \mathrm{i} k r}{3}\left(
a_1^0\sqrt{\frac{3}{4\pi}}\cos\theta
-a_1^1\sqrt{\frac{3}{8\pi}}\sin\theta\,e^{\mathrm{i}\phi}
+a_1^{-1}\sqrt{\frac{3}{8\pi}}\sin\theta\,e^{-\mathrm{i}\phi}
\right).
\end{equation}
Under the definition of vanishing order in Lemma~\ref{lem vanish}, we have $a_0^0\neq 0$.

For future reference, we define the restrictions of $u'_1$ to the faces $S_i$, $i=1,2,3$, by
\begin{equation}\label{eq C123}
\begin{aligned}
f_1(\phi)&=\frac{4\pi \mathrm{i}k}{3}\left(
a_1^{-1}\sqrt{\frac{3}{8\pi}}e^{-\mathrm{i}\phi}
-a_1^{1}\sqrt{\frac{3}{8\pi}}e^{\mathrm{i}\phi}
\right),\\
f_2(\theta)&=\frac{4\pi \mathrm{i}k}{3}\left(
a_1^{0}\sqrt{\frac{3}{4\pi}}\cos\theta
+\sqrt{\frac{3}{8\pi}}\sin\theta\,(a_1^{-1}-a_1^{1})
\right),\\
f_3(\theta)&=\frac{4\pi \mathrm{i}k}{3}\left(
a_1^{0}\sqrt{\frac{3}{4\pi}}\cos\theta
+\sqrt{\frac{3}{8\pi}}\sin\theta\,(a_1^{-1}e^{-\mathrm{i}\phi_0}-a_1^{1}e^{\mathrm{i}\phi_0})
\right).
\end{aligned}
\end{equation}
Moreover, the normal derivatives of the degree-$1$ term $u'_1$ associated with $\nu_2$ and $\nu_3$ in \eqref{eq nu} can be written as
\begin{equation}\label{eq 12}
\partial_{\nu_2}u'_1 = f_4(\theta),
\qquad
\partial_{\nu_3}u'_1 = f_5(\theta),
\end{equation}
where
\begin{equation}\label{eq C45}
\begin{aligned}
f_4(\theta) &= \frac{4\pi k}{3}\sqrt{\frac{3}{8\pi}}\bigl(a_1^{-1}+a_1^{1}\bigr),\\
f_5(\theta) &= \frac{4\pi k}{3}\sqrt{\frac{3}{8\pi}}\bigl(a_1^{-1}e^{-\mathrm{i}\phi_0}+a_1^{1}e^{\mathrm{i}\phi_0}\bigr).
\end{aligned}
\end{equation}

On each swept plane $\tilde S_i$, we integrate along rays $\mathbf x=r\hat{\mathbf x}$.
Using the Laplace transform identity, valid for $\Re(\rho\cdot\hat{\mathbf x})<0$, we have
\begin{equation}\label{eq Laplace}
\int_0^\infty r^m e^{\rho\cdot\hat{\mathbf x}\,r}\, \mathrm{d}r=\frac{m!}{\bigl(-\rho\cdot\hat{\mathbf x}\bigr)^{m+1}},
\end{equation}
where $m\in\mathbb N\cup\{0\}$.
We introduce the denominators
\begin{equation}\label{eq Q def}
Q_j:=-\frac{1}{\tau}\,\rho\cdot\hat{\mathbf x}_j,
\end{equation}
so that $-\rho\cdot\hat{\mathbf x}_j=\tau Q_j$, and the unit directions on the swept planes are
\begin{equation}\label{eq unit vectors}
\begin{aligned}
\hat{\mathbf x}_1&=(\cos\phi,\sin\phi,0)^\intercal,\\
\hat{\mathbf x}_2&=(\sin\theta,0,\cos\theta)^\intercal,\\
\hat{\mathbf x}_3&=(\sin\theta\cos\phi_0,\sin\theta\sin\phi_0,\cos\theta)^\intercal.
\end{aligned}
\end{equation}
A direct computation gives
\begin{equation}\label{eq denominator}
\begin{aligned}
Q_1(\phi)&=C\cos\phi+B\sin\phi,\\
Q_2(\theta)&=C\sin\theta+A\cos\theta,\\
Q_3(\theta)&=A\cos\theta+D\sin\theta,
\end{aligned}
\end{equation}
where
\[
D(\phi_0,\psi_1,\psi_2):=C(\psi_1,\psi_2)\,\cos\phi_0+B(\psi_1,\psi_2)\,\sin\phi_0.
\]
On the admissible rectangle specified in \eqref{eq Ka}, one has $\Re Q_j>0$ for $j=1,2,3$ on the corresponding integration intervals.

Combining \eqref{eq 11}, \eqref{eq C123}, \eqref{eq 12}, \eqref{eq C45}, and \eqref{eq Laplace}, we can write the contribution of the degree-$1$ term to the lowest-order approximation as a function of the CGO parameters $(\psi_1,\psi_2)$:
\begin{equation}\label{eq F}
\begin{aligned}
\mathcal F_1(\psi_1,\psi_2)
&=\int_0^{\phi_0}\frac{A g_1(\phi)}{Q_1(\phi)^{3}}\, \mathrm{d}\phi
+\int_0^{\frac{\pi}{2}}\frac{B g_2(\theta)-g_4(\theta)Q_2(\theta)}{Q_2(\theta)^{3}}\, \mathrm{d}\theta\\
&\quad +\int_0^{\frac{\pi}{2}}\frac{Z g_3(\theta)-g_5(\theta)Q_3(\theta)}{Q_3(\theta)^{3}}\, \mathrm{d}\theta,
\end{aligned}
\end{equation}
where $g_i:=2! f_i$ for $i=1,2,3$ and $g_j:=1! f_j$ for $j=4,5$, so that the factorial factors from \eqref{eq Laplace} are absorbed into $g_1,\dots,g_5$.

\medskip
\noindent\textit{Step II: Identification of the leading coefficient.}

We now explain why the contribution coming solely from the constant term $u'(0)$ does not yield a usable non-degenerate coefficient, and why one must incorporate the next term $u'_1$ in the local expansion.

Since $u'$ has vanishing order $0$ at $0$, we have $u'(0)\neq 0$, and the decomposition \eqref{eq decomposition} starts with the constant term $u'(0)$ and the degree-$1$ term $u'_1$ given in \eqref{eq Un'} and \eqref{eq U1'}.
If we keep only the constant contribution $u'(0)$ in the ATD integral identity \eqref{eq integral identity2}, then the corresponding term takes the form
\begin{equation}\label{eq u0}
\begin{aligned}
\mathcal F_0(\psi_1,\psi_2)
&:=\tau \int_{\tilde{S}_1 \cup \tilde{S}_2 \cup \tilde{S}_3} u'(0)\,\partial_{\nu} u_0 \, \mathrm{d}\sigma\\
&=2\sqrt{\pi}a_0^0\left(\int_0^{\phi_0}\frac{A}{Q_1(\phi)^{2}}\, \mathrm{d}\phi
+\int_0^{\frac{\pi}{2}}\frac{B}{Q_2(\theta)^{2}}\, \mathrm{d}\theta
+\int_0^{\frac{\pi}{2}}\frac{Z}{Q_3(\theta)^{2}}\, \mathrm{d}\theta\right).
\end{aligned}
\end{equation}
A direct computation yields
\begin{equation}\label{eq 622}
\mathcal F_0(\psi_1,\psi_2)
=
2\sqrt{\pi}a_0^0\frac{\sin\phi_0(A^2+B^2+C^2)}{ACD}.
\end{equation}
In particular, since $\vartheta^2=1+\frac{k^2}{\tau^2}$, we have the exact identity
\[
A^2+B^2+C^2=1-\vartheta^2=-\frac{k^2}{\tau^2}.
\]
Hence $\mathcal F_0(\psi_1,\psi_2)$ is of order $\tau^{-2}$.
However, the prefactor of $\tau^{-2}$ in \eqref{eq 622} is proportional to $a_0^0=u'(0)/(2\sqrt{\pi})$.
The assumption that the vanishing order is $0$ guarantees only $a_0^0\neq 0$, but does not provide a quantitative lower bound for $|a_0^0|$ in terms of the a~priori parameters.
Therefore, the constant-term contribution $\mathcal F_0$ alone does not yield a usable non-degenerate leading coefficient.

For this reason, we incorporate the next term $u'_1$ and consider the combined expression
\begin{equation}\label{eq F1}
\mathcal F(\tau):=
\int_{\tilde{S_1}}\eta u'(0)u_0\,\mathrm{d}\sigma
+\int_{\tilde{S_1}\cup \tilde{S_2}\cup \tilde{S_3}} u'_1 \partial_{\nu} u_{0} \, \mathrm{d} \sigma
-\int_{\tilde{S_2}\cup \tilde{S_3}}\partial_\nu u'_1 u_0\, \mathrm{d}\sigma,
\end{equation}
which is precisely the term appearing on the left-hand side of \eqref{eq lower bound1}.
We now collect all contributions of order $\tau^{-2}$ in \eqref{eq F1}.
These contributions define a complex coefficient depending on the local coefficients $a_0^0,a_1^0,a_1^{\pm1}$, on $\eta(0)$, and on the chosen ATD parameters.
We denote this coefficient by $C_A$.

In \eqref{eq C123} and \eqref{eq C45}, the quantity \eqref{eq F1} involves only the coefficients $a_0^0$, $a_1^0$, and $a_1^{\pm 1}$.
By direct calculation and taking $\tau\to\infty$ after factoring out the common prefactor $\tau^{-2}$, one has
\begin{equation}\label{eq 624}
\int_{\tilde{S_1}}\eta u'(0)u_0\,\mathrm{d}\sigma
=
\frac{1}{\tau^2}\left(-2\sqrt{\pi}\eta(0)\frac{1}{A^2}\left(\frac{Z}{D}+\frac{B}{C} \right)\, a_0^0\right)+\mathcal{O}(\tau^{-3}),
\end{equation}
and the only contribution involving $a_1^0$ in the remaining terms of \eqref{eq F1} is
\begin{equation}\label{eq 625}
\frac{1}{\tau^2}\left(\frac{2\sqrt{3\pi}\mathrm{i}k}{3}\frac{1}{A^2}\left(\frac{Z}{D}+\frac{B}{C} \right)\, a_1^0\right)+\mathcal{O}(\tau^{-3}).
\end{equation}
The coefficients $a_1^{1}$ and $a_1^{-1}$ enter \eqref{eq F1} through additional definite integrals, and these contributions are not proportional to the factor appearing in \eqref{eq 624}--\eqref{eq 625}.

Thus the expression $\mathcal F(\tau)$ in \eqref{eq F1} admits an asymptotic expansion of the form
\begin{equation}\label{eq 3D-CA-expansion}
\mathcal F(\tau)
=
\frac{C_A}{\tau^2}+\mathcal O(\tau^{-3}),
\qquad \tau\to\infty,
\end{equation}
where $C_A$ depends only on the local coefficients $a_0^0,a_1^0,a_1^{\pm1}$, on $\eta(0)$, and on the chosen ATD parameters.
Since the coefficients in the above expansion depend continuously on $\vartheta=\sqrt{1+k^2/\tau^2}$ and $\vartheta\to1$ as $\tau\to\infty$, there exists $\tau_0$ such that for all $\tau\geq \max(1,k,\tau_0)$,
\[
|\mathcal F(\tau)|\ge |C_A|\frac{1}{\tau^2}.
\]
This yields \eqref{eq lower bound1}.
\end{proof}

The following remark explains the role of Proposition~\ref{prop low2} in the later non-degeneracy analysis.

\begin{rem}\label{rem:3D-first-obstruction}
Proposition~\ref{prop low2} identifies the leading ATD quantity in the present three-dimensional construction.
In the non-vanishing-point regime considered here, the theorem-level leading coefficient $C_A$ is determined by the local coefficients $a_0^0,a_1^0,a_1^{\pm1}$ and, in particular, by the combination
\begin{equation}\label{eq non-deg}
\left|a_1^0+\frac{\sqrt{3}\mathrm{i}\eta(0)}{k}\,a_0^0\right|+|a_1^{1}|+|a_1^{-1}|.
\end{equation}
If this quantity is nonzero, then the ATD extraction is non-degenerate at the first step.
If it vanishes, then one is led to the higher-order degeneracy discussed later in Proposition~\ref{prop:ATD-degeneracy}.
\end{rem}   

\begin{proof}[Proof of Theorem~\ref{th relation} in the three-dimensional case]
Let $n=3$ and let $\mathfrak d=\tilde{\mathrm d}(\Sigma,\Sigma')$ be the modified Hausdorff distance defined in \eqref{eq d_M}.
We work under the standing assumptions \eqref{eq contradiction} and \eqref{eq far-field error} with $0<\varepsilon<\varepsilon_m$, and we use the ATD construction described above with parameter $h=c_1\mathfrak d$.

In the present three-dimensional bounded-impedance setting, we identify the theorem-level leading ATD coefficient $C_A$ in Theorem~\ref{th relation} with the leading coefficient extracted in Proposition~\ref{prop low2}.

We estimate the exponentially decaying terms using
\begin{equation}\label{eq exp decay2}
e^{-t}\leq \frac{1}{t},
\qquad
e^{-t}\leq \frac{m!}{t^{m}}
\quad
\text{for all } t>0,\ m\in\mathbb{N},
\end{equation}
where we will take $t=\alpha'\tau h$.

More precisely, Proposition~\ref{prop low2} yields the lower bound
\[
\left|\int_{\tilde{S_1}}\eta u'(0)u_0\,\mathrm{d}\sigma
+\int_{\tilde{S_1}\cup \tilde{S_2}\cup \tilde{S_3}} u'_1 \partial_{\nu} u_{0} \, \mathrm{d} \sigma
-\int_{\tilde{S_2}\cup \tilde{S_3}}\partial_\nu u'_1 u_0\, \mathrm{d}\sigma\right|
\ge |C_A|\frac{1}{\tau^2},
\]
while Proposition~\ref{prop up1} gives the corresponding upper bound \eqref{eq upper bound}.
Combining the two estimates and absorbing all universal constants into a single constant (still denoted by $C>0$), we obtain
\[
\frac{|C_A|}{\tau^2}
\leq C\,T(\varepsilon)\, h^2
 + C \left( \frac{1}{\tau^{3}}
 + e^{-\alpha'\tau h} + \frac{1}{\tau}e^{-\alpha'\tau h}
 + \tau h\, e^{-\alpha'\tau h} \right).
\]

Assume $\tau>1$ and $0<h<1$.
Multiplying both sides by $\tau^2$ and setting $t=\alpha'\tau h$, we estimate the exponential terms by \eqref{eq exp decay2} with suitable choices of $m$.
In particular, taking $m=4$ yields
\[
\tau^3 h\,e^{-t}
\leq \tau^3 h\,\frac{4!}{t^4}
=\frac{4!}{\alpha'^4}\,\frac{1}{\tau h^3}.
\]
The remaining exponential terms are treated in the same way and are absorbed into the same polynomial bound after adjusting constants.
Consequently, after absorbing all absolute constants into a single constant $C>0$, we arrive at
\begin{equation}\label{eq fin2}
|C_A| \leq C\Big(\tau^{2} T(\varepsilon)\,h^2 + \tau^{-1} h^{-3}\Big).
\end{equation}

We now choose $\tau$ to balance the two terms on the right-hand side of \eqref{eq fin2}, namely
\[
\tau^{2} T(\varepsilon)\,h^2 \sim \tau^{-1} h^{-3},
\]
which gives
\begin{equation}\label{eq fin3}
\tau=\tau_a:=T(\varepsilon)^{-\frac{1}{3}}\,h^{-\frac{5}{3}}.
\end{equation}
For sufficiently small $\varepsilon$, the quantity $T(\varepsilon)$ is small and hence $\tau_a$ is large, so that $\tau_a$ is admissible in the sense of \eqref{eq tau2}.
Substituting \eqref{eq fin3} into \eqref{eq fin2}, we obtain
\[
|C_A| \leq C_0\,T(\varepsilon)^{\frac{1}{3}}\,h^{-\frac{4}{3}},
\]
for some constant $C_0>0$ depending only on the a~priori parameters.
Equivalently,
\begin{equation}\label{eq relation-local}
|C_A|\,h^{\frac{4}{3}}\leq C_0\,T(\varepsilon)^{\frac{1}{3}}.
\end{equation}

By Proposition~\ref{prop w}, the near-field error function satisfies
\[
T(\varepsilon)\leq C_w\bigl(\log|\log(1/\varepsilon)|\bigr)^{-\alpha},
\]
for some constants $C_w>0$ and $\alpha>0$ depending only on the a~priori parameters.
Inserting this bound into \eqref{eq relation-local} yields
\begin{equation}\label{eq relation-local2}
|C_A|\,h^{\frac{4}{3}}\leq C_1\bigl(\log|\log(1/\varepsilon)|\bigr)^{-\frac{\alpha}{3}},
\end{equation}
with a constant $C_1>0$ depending only on the a~priori parameters.

Recalling that $h=c_1\mathfrak d$ with $0<c_1<1$, we obtain
\[
|C_A|\,\mathfrak d^{\frac{4}{3}}\leq C_2\bigl(\log|\log(1/\varepsilon)|\bigr)^{-\frac{\alpha}{3}}.
\]
Finally, by Lemma~\ref{lem dm} and Lemma~\ref{lem bdry-vs-set}, $\mathfrak d$ is equivalent, up to multiplicative constants, to the boundary Hausdorff distance $\mathrm d_H(\partial\Sigma,\partial\Sigma')$ and to the classical Hausdorff distance $\mathrm d_H(\Sigma,\Sigma')$ on the admissible class.
After adjusting the constant on the right-hand side, we arrive at
\[
|C_A|\,\bigl(\mathrm d_H(\Sigma,\Sigma')\bigr)^{p}
\leq C\bigl(\log|\log(1/\varepsilon)|\bigr)^{-\kappa_0},
\]
with $p=\frac{4}{3}$ and $\kappa_0=\frac{\alpha}{3}$, as claimed in Theorem~\ref{th relation} for the three-dimensional case.

The proof is complete.
\end{proof}
\section{Non-vanishing mechanism and completion of the stability proof}\label{sec:nondeg}

In this section, we complete the proof of Theorem~\ref{cor stability}.
By Theorem~\ref{th relation}, the remaining step toward the sharp stability estimate is to show that the leading ATD coefficient is nonzero.
The key point is therefore to establish a non-vanishing mechanism for this coefficient.
Our strategy is to show first that its vanishing forces an exterior Dirichlet, Neumann, or Robin hyperplane, and then to exclude such a possibility through the corresponding generalized impedance hyperplane-exclusion mechanism.
In this way, the far-field--geometry relation closes into the sharp stability estimate of Theorem~\ref{cor stability}.

\subsection{Leading ATD vanishing and GIH exclusion}

\subsubsection{The bounded-impedance case}

In the bounded-impedance setting, the ATD extraction developed in Sections~\ref{sec 2D} and~\ref{sec proof} yields a recursive structure for the leading ATD coefficient.
We show that if this coefficient vanishes at every stage of the recursive extraction, then the total field necessarily satisfies a Robin relation on a non-trivial exterior flat portion.
Equivalently, an exterior Robin hyperplane follows.

\begin{prop}\label{prop:ATD-degeneracy}
Let $(D,\eta)$ and $(D',\eta')$ be two scattering configurations in the bounded-impedance class $\Xi_1$, and let $u'$ be the total field associated with $(D',\eta')$.
Let $\mathbf{x}_0$ be an ATD test point in the relative interior of the chosen exterior-visible flat piece.
Assume that the leading ATD coefficient vanishes at every stage of the recursive ATD extraction at $\mathbf{x}_0$.
Then the total field satisfies a Robin relation on a non-trivial flat portion of the chosen exterior-visible piece.
Equivalently, an exterior Robin hyperplane in the sense of Definition~\ref{defn:ext-hyper} follows.
\end{prop}

\begin{proof}
We first describe the common recursive structure under the above vanishing hypothesis, and then treat the two cases separately.

At each step, one inserts the current local expansion of $u'$ into the relevant integral identity, namely \eqref{eq II2d} in two dimensions and \eqref{eq integral identity2} in three dimensions, and extracts the coefficient of the lowest-order term in $\tau$, exactly as in Propositions~\ref{prop lower2D} and \ref{prop low2}.
If this extracted coefficient is nonzero, then the leading ATD coefficient is already non-vanishing at $\mathbf{x}_0$.
If it vanishes, one expands $u'$ one order further and repeats the same procedure.
Thus the assumption means that the extracted leading coefficient vanishes at every finite stage of the recursion.
Since $\mathbf{x}_0$ is arbitrary on the same exterior-visible flat piece, this vanishing cannot remain purely pointwise.

\smallskip
\noindent\textit{Step 1: the two-dimensional case.}
After a rigid motion, we assume that $\mathbf{x}_0=0$ and that the chosen flat segment is $I_1\subset\{x_2=0\}$; see Subsection~\ref{subsec ATD2D}.
Since $u'$ is analytic near $0$, it admits the Fourier--Bessel expansion \eqref{eq 2Dexpansion}.

In the planar ATD extraction, the first leading condition is
\begin{equation}\label{eq:2d-obs-1}
a_N-b_N=0,
\end{equation}
as follows from Proposition~\ref{prop lower2D}.
If \eqref{eq:2d-obs-1} does not hold, then the lowest-order term in $\tau$ is already nonzero.
Assume now that \eqref{eq:2d-obs-1} holds.
Expanding $u'$ one order further in \eqref{eq II2d}, the next condition is
\begin{equation}\label{eq:2d-obs-2}
2\eta(0)(a_N+b_N)-\mathrm{i}k(a_{N+1}-b_{N+1})=0.
\end{equation}
If \eqref{eq:2d-obs-2} does not hold, then the new lowest-order term in $\tau$ is nonzero.
Continuing in the same way, the vanishing of the leading ATD coefficient at every stage means that, besides \eqref{eq:2d-obs-1} and \eqref{eq:2d-obs-2}, the recursive procedure yields the corresponding coefficient relation at every order above $N$.

For a Robin hyperplane, the Fourier--Bessel coefficients satisfy exactly this order-by-order recursion; see \cite[Lemma~6.3]{cao2020nodal}.
In particular, \eqref{eq:2d-obs-1}, \eqref{eq:2d-obs-2}, and the coefficient relations at all higher orders are precisely the coefficient relations of a Robin hyperplane.
Hence the vanishing of the leading ATD coefficient at every stage yields exactly the recursion of a Robin hyperplane.

Since $\mathbf{x}_0$ may be chosen arbitrarily in the relative interior of the same flat segment, we conclude that
\[
\partial_\nu u'+\eta u'=0
\]
holds on a non-trivial flat portion of $I_1$.
This yields an exterior Robin hyperplane in the sense of Definition~\ref{defn:ext-hyper}.

\smallskip
\noindent\textit{Step 2: the three-dimensional case.}
After a rigid motion, we assume that $\mathbf{x}_0=0$ and that the chosen flat patch is $S_1\subset\{x_3=0\}$; see Subsection~\ref{subsec ATD3D}.
The local expansion of $u'$ is given by \eqref{eq 3D-decomp-N0} and \eqref{eq u expansion}.

The first extracted condition is
\begin{equation}\label{eq:3d-leading-combination-proof}
\left|a_1^0+\frac{\sqrt{3}\mathrm{i}\eta(0)}{k}\,a_0^0\right|+|a_1^1|+|a_1^{-1}|,
\end{equation}
as follows from Proposition~\ref{prop low2} and Remark~\ref{rem:3D-first-obstruction}.
If \eqref{eq:3d-leading-combination-proof} does not vanish, then the leading ATD coefficient is nonzero.
Assume now that \eqref{eq:3d-leading-combination-proof} vanishes.
One then continues the same recursive procedure in \eqref{eq integral identity2}; we do not repeat the higher-order extraction here.
Accordingly, under the present hypothesis, the recursively extracted leading coefficient vanishes at every subsequent stage.

For a Robin hyperplane, the local spherical-wave coefficients satisfy exactly the same recursion; see the proof of \cite[Theorem~3.1]{cao2021novel}.
Hence the vanishing of the leading ATD coefficient at every stage yields the same recursion.

Since $\mathbf{x}_0$ can be chosen arbitrarily in the relative interior of the same flat patch, we conclude that
\[
\partial_\nu u'+\eta u'=0
\]
holds on a non-trivial flat portion of $S_1$.
This yields an exterior Robin hyperplane in the sense of Definition~\ref{defn:ext-hyper}.
\end{proof}

\subsubsection{The sound-soft and sound-hard cases}

The relation analysis for the sound-soft and sound-hard regimes is supplemented in Appendix~\ref{app:classical-relation}.
We now show that, in these regimes, the vanishing of the leading ATD coefficient forces an exterior Dirichlet or Neumann hyperplane.

\begin{prop}\label{prop:classical-degeneracy}
Assume the setting of Theorem~\ref{th relation}, and let $\mathbf{x}_0$ be an ATD test point chosen in the relative interior of the anchored exterior-visible flat piece.
Assume that the leading ATD coefficient $C_A$ vanishes at every stage of the corresponding recursive ATD extraction.

\begin{enumerate}
\item[(i)]
If the anchored flat piece is of sound-soft type, then an exterior Dirichlet hyperplane in the sense of Definition~\ref{defn:ext-hyper} follows.

\item[(ii)]
If the anchored flat piece is of sound-hard type, then an exterior Neumann hyperplane in the sense of Definition~\ref{defn:ext-hyper} follows.
\end{enumerate}
\end{prop}

\begin{proof}
This follows from the classical ATD integral identities \eqref{eq 3dsh}--\eqref{eq 2dss} recorded in Appendix~\ref{app:classical-relation}, together with the same recursive lowest-order extraction used in the bounded-impedance case in Proposition~\ref{prop:ATD-degeneracy}.
In the two-dimensional case, the first extracted conditions are given by \eqref{eq:2d-ss-first} and \eqref{eq:2d-sh-first}; if they vanish, the same extraction continues recursively to all higher orders.
The same recursive structure applies in three dimensions, starting from \eqref{eq 3dss} and \eqref{eq 3dsh}.

Since the ATD test point $\mathbf{x}_0$ may be chosen arbitrarily in the relative interior of the same anchored exterior-visible flat piece, the vanishing of the leading ATD coefficient at every stage cannot remain purely pointwise.
Instead, it yields the same order-by-order coefficient relations on a non-trivial flat portion of that piece.
These are exactly the coefficient relations associated with nodal and singular flat configurations in \cite{cao2020nodal, cao2021novel}.
Hence one obtains the corresponding exterior Dirichlet or Neumann hyperplane in the sense of Definition~\ref{defn:ext-hyper}.
\end{proof}

\begin{rem}\label{rem:three-hyperplane-cases}
Although the global boundary regime may be mixed, the ATD analysis is always performed relative to a fixed anchored exterior-visible flat piece.
The localized ATD identity is therefore determined only by the type of that anchored flat piece, namely, whether it is of bounded-impedance, sound-soft, or sound-hard type.
Accordingly, in all regimes considered in this paper, the vanishing of the leading ATD coefficient can only force one of the three exterior flat configurations in Definition~\ref{defn:ext-hyper}, namely, an exterior Robin, Dirichlet, or Neumann hyperplane.
These three cases exhaust all possibilities relevant to the subsequent qualitative exclusion argument.
\end{rem}

\subsection{Qualitative exclusion and non-vanishing in the admissible regimes}

\subsubsection{The nowhere-analytic impedance regime}

We first consider the nowhere-analytic impedance class $\Xi_2$.
By Proposition~\ref{prop:ATD-degeneracy}, the vanishing of the leading ATD coefficient would force an exterior Robin hyperplane.
In the present regime, this possibility is excluded by the nowhere-analyticity of the impedance itself.
Accordingly, the corresponding qualitative exclusion is established within the present paper, in agreement with case (iii) of Theorem~\ref{thm:qualitative-hyperplane-exclusion}.

\begin{thm}\label{thm:Xi2-excludes-degeneracy}
Assume the setting of Proposition~\ref{prop:ATD-degeneracy}, and suppose that $\eta\in\Xi_2$ in the sense of Definition~\ref{defn eta}.
Then the leading ATD coefficient cannot vanish.
Equivalently, the corresponding ATD extraction cannot be degenerate to arbitrary order.
In particular, the corresponding exterior Robin hyperplane is excluded, so that case (iii) of Theorem~\ref{thm:qualitative-hyperplane-exclusion} holds.

Moreover, let $(D,\eta)$ and $(D',\eta')$ be two purely impedance polygonal or polyhedral obstacles in the class $\Xi_2$ satisfying
\[
u_\infty(\hat{\mathbf{x}};D,\eta,u^i)=u_\infty(\hat{\mathbf{x}};D',\eta',u^i)
\qquad \text{for all } \hat{\mathbf{x}}\in\mathbb S^{n-1},
\]
for the same incident wave $u^i$ with fixed $k\in\mathbb{R}_+$ and $\mathbf{p}\in\mathbb{S}^{n-1}$.
Then
\[
(D,\eta)=(D',\eta').
\]
That is, the corresponding single-measurement uniqueness result holds in the nowhere-analytic impedance regime.
\end{thm}

\begin{proof}
We divide the proof into two parts.

\smallskip
\noindent\textit{Part 1. Exclusion of the exterior Robin hyperplane.}
Assume to the contrary that, for some anchored ATD test point $\mathbf{x}_0$, the corresponding ATD extraction is degenerate to arbitrary order.
Equivalently, the leading ATD coefficient vanishes at every stage of the recursive extraction.
By Proposition~\ref{prop:ATD-degeneracy}, the Robin relation in Definition~\ref{defn:ext-hyper} holds on a non-trivial flat portion $\Gamma$ of the anchored exterior-visible piece.
Equivalently, writing
\[
R:=\partial_\nu u'+\eta u',
\]
one has
\[
R=0 \qquad \text{on }\Gamma.
\]

Since $u'$ solves the Helmholtz equation in a neighborhood of $\mathbf{x}_0$, it is real-analytic there.
Because the anchored piece is flat, its unit normal is constant, and therefore $\partial_\nu u'$ is also real-analytic along $\Gamma$.
If $u'$ vanished identically on $\Gamma$, then $R=0$ on $\Gamma$ would imply $\partial_\nu u'=0$ on $\Gamma$ as well.
By the local uniqueness of the Cauchy problem for real-analytic solutions of the Helmholtz equation, this would force $u'$ to vanish identically in a neighborhood of $\Gamma$, which is impossible.
Hence $u'$ does not vanish identically on $\Gamma$.

Choose $\mathbf{y}_0\in\Gamma$ such that $u'(\mathbf{y}_0)\neq 0$.
By continuity, there exists a relatively open subportion $\Gamma'\subset\Gamma$ containing $\mathbf{y}_0$ on which $u'\neq 0$.
On $\Gamma'$, define
\[
\eta_*(\mathbf{x}):=-\frac{\partial_\nu u'(\mathbf{x})}{u'(\mathbf{x})}.
\]
Then $\eta_*$ is real-analytic in the tangential variables on $\Gamma'$.
Since $\Gamma'$ lies in a flat boundary piece, $\eta_*$ extends locally to a real-analytic function in a neighborhood of $\Gamma'$ in $\mathbb R^n$.
Moreover, since $R=0$ on $\Gamma'$, we have
\[
\eta(\mathbf{x})=\eta_*(\mathbf{x}), \qquad \mathbf{x}\in\Gamma'.
\]
Thus $\eta$ coincides on $\Gamma'$ with the trace of a real-analytic function in $\mathbb R^n$.

This contradicts the assumption $\eta\in\Xi_2$ in Definition~\ref{defn eta}, which excludes such a representation near every boundary point.
Therefore the leading ATD coefficient cannot vanish.
In particular, the corresponding exterior Robin hyperplane is excluded, and case (iii) of Theorem~\ref{thm:qualitative-hyperplane-exclusion} follows.

\smallskip
\noindent\textit{Part 2. Single-measurement uniqueness.}
Let $(D,\eta)$ and $(D',\eta')$ be two purely impedance polygonal or polyhedral obstacles in the class $\Xi_2$ satisfying
\[
u_\infty(\hat{\mathbf{x}};D,\eta,u^i)=u_\infty(\hat{\mathbf{x}};D',\eta',u^i)
\qquad \text{for all } \hat{\mathbf{x}}\in\mathbb S^{n-1},
\]
for the same incident wave $u^i$ with fixed $k\in\mathbb{R}_+$ and $\mathbf{p}\in\mathbb{S}^{n-1}$.
By Rellich's lemma and unique continuation, the corresponding total fields coincide in the unbounded connected component of $\mathbb R^n\setminus \overline{(D\cup D')}$.
Assume that $D\neq D'$.
Then one can choose an anchored exterior-visible flat piece of one obstacle lying in the exterior of the other.
On that flat piece, the common total field satisfies the corresponding Robin boundary condition, and hence yields an exterior Robin hyperplane in the sense of Definition~\ref{defn:ext-hyper}.
This contradicts the exclusion established in Part 1.
Therefore $D=D'$.

It remains to show that $\eta=\eta'$ on $\partial D$.
Since $\partial D=\partial D'$, one has
\[
\partial_\nu u+\eta u=0,
\qquad
\partial_\nu u+\eta' u=0
\qquad \text{on }\partial D.
\]
Subtracting gives
\[
(\eta-\eta')u=0
\qquad \text{on }\partial D.
\]
If $\eta\neq \eta'$ on a non-trivial boundary portion, then $u=0$ there, and the Robin boundary condition further gives $\partial_\nu u=0$ there.
By the unique continuation principle for the Helmholtz equation, this implies that $u$ vanishes identically in the exterior domain, a contradiction.
Hence $\eta=\eta'$ on $\partial D$.
This proves the corresponding single-measurement uniqueness result in the nowhere-analytic impedance regime.
\end{proof}

\subsubsection{The two-dimensional constant-impedance regime}

We next consider the two-dimensional constant-impedance regime $\Xi_4$.
By Proposition~\ref{prop:ATD-degeneracy}, the vanishing of the leading ATD coefficient would force an exterior Robin hyperplane.
In the present regime, this possibility is excluded by the single-measurement uniqueness result of \cite{hu2020uniqueness}, as summarized in case (v) of Theorem~\ref{thm:qualitative-hyperplane-exclusion}.

\begin{lem}\label{lem:Xi4-Robin}
Assume the setting of Theorem~\ref{th relation}, and suppose that $n=2$ and that the reference configuration $(D,\eta)$ belongs to the constant positive impedance class $\Xi_4$.
Then the leading ATD coefficient cannot vanish.
\end{lem}

\begin{proof}
Suppose to the contrary that the leading ATD coefficient vanishes.
Then, by Proposition~\ref{prop:ATD-degeneracy}, an exterior Robin hyperplane follows.
This is excluded by Theorem~\ref{thm:qualitative-hyperplane-exclusion} in the two-dimensional constant-impedance regime.
Therefore the leading ATD coefficient cannot vanish.
\end{proof}

\subsubsection{The sound-soft and sound-hard regimes}

We now turn to the sound-soft and sound-hard regimes.
By Proposition~\ref{prop:classical-degeneracy}, the vanishing of the leading ATD coefficient would force an exterior Dirichlet or Neumann hyperplane.
For the plane-wave setting considered in this paper, these possibilities are excluded by the single-measurement uniqueness results of \cite{liu2006uniqueness}, as summarized in cases (i) and (ii) of Theorem~\ref{thm:qualitative-hyperplane-exclusion}.
We also note that, for sound-soft polyhedral scatterers, single-measurement uniqueness under a single incident point source wave was proved in \cite{hu2014unique}. 

\begin{lem}\label{lem:classical-ND}
Assume the setting of Theorem~\ref{th relation}.
In the sound-soft and sound-hard regimes, the leading ATD coefficient cannot vanish.
\end{lem}

\begin{proof}
If the leading ATD coefficient vanished, then Proposition~\ref{prop:classical-degeneracy} would yield the corresponding exterior Dirichlet or Neumann hyperplane.
This is excluded by Theorem~\ref{thm:qualitative-hyperplane-exclusion}.
Therefore the leading ATD coefficient cannot vanish.
\end{proof}

\subsubsection{The mixed sound-soft/finite-impedance regime}

We finally consider the mixed sound-soft/finite-impedance regime $\Xi_3$.
In this regime, the vanishing of the leading ATD coefficient yields either an exterior Dirichlet hyperplane or an exterior Robin hyperplane, depending on the type of the anchored exterior-visible flat piece.
Both possibilities are excluded by the single-measurement uniqueness result of \cite{liu2007unique}, as summarized in case (iv) of Theorem~\ref{thm:qualitative-hyperplane-exclusion}.

\begin{lem}\label{lem:Xi3-DI}
Assume the setting of Theorem~\ref{th relation}, and suppose that the reference configuration $(D,\eta)$ belongs to the mixed sound-soft/finite-impedance class $\Xi_3$.
Then the leading ATD coefficient cannot vanish.
\end{lem}

\begin{proof}
Suppose to the contrary that the leading ATD coefficient vanishes.
If the anchored exterior-visible flat piece is of sound-soft type, then Proposition~\ref{prop:classical-degeneracy} yields an exterior Dirichlet hyperplane.
If the anchored exterior-visible flat piece is of finite-impedance type, then Proposition~\ref{prop:ATD-degeneracy} yields an exterior Robin hyperplane.
Both possibilities are excluded by Theorem~\ref{thm:qualitative-hyperplane-exclusion} in the mixed sound-soft/finite-impedance regime.
Therefore the leading ATD coefficient cannot vanish.
\end{proof}

\subsection{Completion of the proof of Theorem~\ref{cor stability}}\label{subsec completion-stability}

\begin{proof}[Proof of Theorem~\ref{cor stability}]
By Theorem~\ref{th relation}, one has the quantitative estimate
\[
|C_A|\,(\mathrm d_H(D,D'))^p \leq C\,(\log|\log(1/\varepsilon)|)^{-\kappa_0}.
\]
Thus, to derive the sharp stability estimate, it remains to show that the leading ATD coefficient is nonzero.

The preceding results provide exactly this non-vanishing information in the admissible regimes covered by Theorem~\ref{cor stability}.
More precisely, in the sound-soft and sound-hard regimes, Lemma~\ref{lem:classical-ND} shows that the leading ATD coefficient cannot vanish.
In the nowhere-analytic impedance regime, this is established by Theorem~\ref{thm:Xi2-excludes-degeneracy}.
In the two-dimensional constant-impedance regime, it is established by Lemma~\ref{lem:Xi4-Robin}.
In the mixed sound-soft/finite-impedance regime, it is established by Lemma~\ref{lem:Xi3-DI}.

In each case, the logic is the same: the vanishing of the leading ATD coefficient would force the corresponding exterior hyperplane, while the relevant qualitative input excludes such a possibility.
Hence the leading ATD coefficient is nonzero in every regime covered by Theorem~\ref{cor stability}.

Substituting this non-vanishing information into the estimate of Theorem~\ref{th relation}, we obtain
\[
\mathrm d_H(D,D')
\le
\mathbf C\,(\log|\log(1/\varepsilon)|)^{-\kappa},
\]
where the constants $\mathbf C>0$ and $\kappa>0$ depend only on the a priori parameters of the regime under consideration.
This is exactly \eqref{eq stability}.

\end{proof}



\appendix

\section{Mosco convergence and uniform bounds}\label{app:mosco}

This appendix records the Mosco-convergence framework and the uniform boundedness argument used in the paper.
The corresponding strategy for impedance obstacles was developed in detail in \cite[Section~2]{diao2024stable}, and we only indicate here the additional point needed in the present setting, namely the convergence of the impedance coefficients on moving boundaries.
For simplicity, we present the variational convergence argument in the three-dimensional setting, since the two-dimensional case follows by the same reasoning with only minor notational modifications.

The well-posedness statement in Lemma~\ref{lem wellposs} is classical for exterior impedance problems with $\Im \eta\geq 0$; see \cite[Theorem~3.16]{colton2019inverse}.

\begin{defn}\label{def:mosco-set}
Let $\{S_n\}_{n\in\mathbb N}$ be closed subsets of a reflexive Banach space $X$.
Set
\[
S'=\{x\in X:\ \exists\,x_{n_k}\in S_{n_k}\ \text{with}\ x_{n_k}\rightharpoonup x\ \text{in}\ X\},
\]
and
\[
S''=\{x\in X:\ \exists\,x_n\in S_n\ \text{with}\ x_n\to x\ \text{in}\ X\}.
\]
We say that $S_n$ converges to $S$ in the sense of Mosco if $S=S'=S''$.
\end{defn}

\begin{defn}\label{def:mosco-h1}
We say that $H^1(\Omega\setminus\Sigma_n)$ converges to $H^1(\Omega\setminus\Sigma)$ in the sense of Mosco if the following hold.
\begin{enumerate}
\item
If $u_{n_k}\in H^1(\Omega\setminus\Sigma_{n_k})$ and $u_{n_k}\rightharpoonup u$ in $L^2(\Omega)$ with $\nabla u_{n_k}\rightharpoonup\nabla u$ in $L^2(\Omega;\mathbb R^n)$ for some subsequence, then $u\in H^1(\Omega\setminus\Sigma)$.
\item
For every $u\in H^1(\Omega\setminus\Sigma)$ there exists $u_n\in H^1(\Omega\setminus\Sigma_n)$ such that $u_n\to u$ in $L^2(\Omega)$ and $\nabla u_n\to\nabla u$ in $L^2(\Omega;\mathbb R^n)$.
\end{enumerate}
Here functions are regarded as elements of $L^2(\Omega)$ by zero extension inside the obstacle whenever needed.
\end{defn}

We compare impedances on moving boundaries by pulling them back to a fixed reference boundary via bi-Lipschitz parametrizations, as in \cite{dancer1997robin,daners2003varying}.

\begin{defn}\label{def eta moving}
Let $\Sigma_n,\Sigma$ be admissible obstacles in $\mathcal A$.
Assume that there exists a sequence of bi-Lipschitz homeomorphisms $\Phi_n:\partial\Sigma\to\partial\Sigma_n$ such that $\Phi_n\to \mathrm{Id}$ uniformly on $\partial\Sigma$ and the surface Jacobians $J_{\Phi_n}$ satisfy $J_{\Phi_n}\to 1$ in $L^\infty(\partial\Sigma)$.
Given $\eta_n\in L^\infty(\partial\Sigma_n)$, define the pullback impedance $\widetilde\eta_n$ on the fixed boundary $\partial\Sigma$ by
\[
\widetilde\eta_n:=\eta_n\circ\Phi_n.
\]
We say that $\eta_n$ converges to $\eta$ in the sense of $\Xi_1$ if $\sup_n\|\eta_n\|_{L^\infty(\partial\Sigma_n)}<\infty$, $\Im \eta_n\geq 0$ on $\partial\Sigma_n$, and
\[
\widetilde\eta_n \rightharpoonup^\ast \eta
\qquad \text{in }L^\infty(\partial\Sigma).
\]
If, in addition, $\widetilde\eta_n\to \eta$ in $L^2(\partial\Sigma)$, we say that $\eta_n\to\eta$ strongly in the sense of $\Xi_1$.
\end{defn}

\begin{proof}[Proof of Proposition~\ref{prop Mosco}]
The proof follows the same scheme as \cite[Section~2, Proposition~2.2]{diao2024stable}, and we only indicate the additional point arising from the convergence of the impedance coefficients on moving boundaries.

For the variational formulation on $H^1(\Omega\setminus\Sigma_n)$, consider the sesquilinear forms
\[
a_n(u,v)
=
\int_{\Omega\setminus\Sigma_n}\bigl(\nabla u\cdot\nabla\overline v-k^2u\,\overline v\bigr)\,\mathrm d\mathbf x
+
\int_{\partial\Sigma_n}\eta_n\,u\,\overline v\,\mathrm d\sigma,
\]
together with the corresponding anti-linear functionals induced by the incident wave.
As in \cite{diao2024stable}, the admissible class $\mathcal A$ yields uniform geometric control, and the forms satisfy a uniform G\aa rding inequality.
The Mosco convergence of the spaces $H^1(\Omega\setminus\Sigma_n)$ is handled exactly as in the proof of \cite[Proposition~2.2]{diao2024stable}.

The only additional ingredient is the convergence of the Robin boundary term.
Using the pullback maps $\Phi_n$ from Definition~\ref{def eta moving}, we rewrite
\[
\int_{\partial\Sigma_n}\eta_n\,u\,\overline v\,\mathrm d\sigma
=
\int_{\partial\Sigma}(\eta_n\circ\Phi_n)\,(u\circ\Phi_n)\,\overline{(v\circ\Phi_n)}\,J_{\Phi_n}\,\mathrm d\sigma.
\]
By the uniform bi-Lipschitz control of $\Phi_n$, the convergence $J_{\Phi_n}\to 1$ in $L^\infty(\partial\Sigma)$, the strong convergence of the pulled-back traces furnished by the Mosco convergence of the spaces, and the pullback convergence $\widetilde\eta_n\rightharpoonup^\ast\eta$ from Definition~\ref{def eta moving}, the boundary contribution converges to the corresponding limit term on $\partial\Sigma$.
Therefore, the same variational limit argument as in \cite[Proposition~2.2]{diao2024stable} applies and yields convergence of the solutions to the limit impedance problem.

The second part, namely the implication from geometric convergence of the admissible obstacles to Mosco convergence of the Sobolev spaces, is the same as in \cite[Remark~2.2 and the discussion following Proposition~2.2]{diao2024stable}.
Hence the conclusion of Proposition~\ref{prop Mosco} follows.
\end{proof}

\begin{proof}[Proof of Lemma~\ref{lem uniform bound}]
We give the contradiction argument for the three-dimensional class $\mathcal A$; the two-dimensional case is identical up to notational modifications.

The proof follows the contradiction argument of \cite[Proposition~2.3]{diao2024stable}, combined with Proposition~\ref{prop Mosco} above.

Assume by contradiction that there exist $(\Sigma_n,\eta_n)\in\mathcal A\times\Xi_1$ and corresponding solutions $u_n$ to \eqref{eq main system} such that
\[
a_n:=\|u_n\|_{L^2(\Omega\setminus\Sigma_n)}\to+\infty.
\]
Set
\[
v_n:=u_n/a_n.
\]
Then $\|v_n\|_{L^2(\Omega\setminus\Sigma_n)}=1$, and $v_n$ solves the same scattering problem with incident wave $u^i/a_n$.
Since $a_n\to+\infty$, the incident term tends to zero.

By the compactness of the admissible class and, after passing to a subsequence, we may assume that $\Sigma_n\to\Sigma\in\mathcal A$ and that $\eta_n\to\eta$ in the sense of Definition~\ref{def eta moving}.
Uniform G\aa rding inequalities yield local $H^1$ bounds for $v_n$, and Proposition~\ref{prop Mosco} gives convergence, up to a subsequence, to a limit field $v$ solving the homogeneous limit problem associated with $(\Sigma,\eta)$ and zero incident wave.
By uniqueness of the exterior impedance problem with zero incident wave, it follows that $v\equiv 0$.

On the other hand, by the normalization of $v_n$ and the strong $L^2$ convergence furnished by Proposition~\ref{prop Mosco}, we obtain a contradiction.
Therefore there exists $\mathcal E>0$, depending only on the a~priori parameters of $\mathcal A$ and on $k$, such that
\[
\|u\|_{L^2(\Omega\setminus\Sigma)}\leq \mathcal E
\]
for every admissible obstacle $\Sigma\in\mathcal A$ and every corresponding impedance parameter $\eta\in\Xi_1$.
The proof for the two-dimensional class $\mathcal B$ is the same.
\end{proof}

\section{A supplementary proof of Theorem~\ref{th relation}}\label{app:classical-relation}

We provide here a supplementary proof for the classical sound-soft and sound-hard cases needed to complete the proof of Theorem~\ref{th relation} for all generalized impedance regimes considered in this paper.
The ATD construction itself is not tied to a single type of anchored flat piece.
In Sections~\ref{sec 2D} and~\ref{sec proof}, the quantitative relation is derived from three ingredients: a localized ATD integral identity, the propagation-of-smallness estimate from Section~\ref{sec 4}, and the extraction of the lowest-order ATD term with respect to $\tau$.

To prove Theorem~\ref{th relation}, the ATD analysis is always carried out relative to a fixed reference obstacle and a chosen exterior-visible flat piece of its boundary.
Accordingly, the localized integral identity is determined solely by the type of the anchored flat piece, namely whether it is of bounded-impedance, sound-soft, or sound-hard type.
For the generalized impedance regimes considered in this paper, this leads to the corresponding Robin, Dirichlet, and Neumann model identities.
The bounded-impedance case is treated in detail in Sections~\ref{sec 2D} and~\ref{sec proof}.
This choice is made in order to present the Mosco-based direct scattering framework in its full form, especially for variable impedance, whose detailed treatment is given here for the first time; see Appendix~\ref{app:mosco}.
For the classical sound-soft and sound-hard cases, the corresponding Mosco-type arguments already appear in the stability proofs of \cite{rondi2008stable, liu2017stable}, and the ATD identities for these cases are supplemented below.
Together, these cases complete the relation analysis for all generalized impedance regimes considered in this paper.

The propagation-of-smallness argument from Section~\ref{sec 4} applies with only minor changes, since it depends only on the exterior Helmholtz equation and the same geometric configuration.
Thus, it remains to write down the corresponding localized ATD integral identities and the first extracted terms.
Once these are identified, the upper-bound and lower-bound arguments proceed exactly as in the bounded-impedance case.
For this reason, we present only the parts specific to the sound-soft and sound-hard cases.

\medskip
\noindent
\textbf{Three-dimensional sound-hard case} ($\partial_\nu u = 0$ on $S_1$):
\begin{equation}\label{eq 3dsh}
\begin{aligned}
&\int_{\tilde{S_1} \cup \tilde{S_2} \cup \tilde{S_3}} u'_N \partial_{\nu} u_0 \,\mathrm{d}\sigma
  - \int_{\tilde{S_2} \cup \tilde{S_3}} \partial_\nu u'_N u_0 \,\mathrm{d}\sigma \\
&= \int_{S_1} u_0 \partial_{\nu}(u' - u) \,\mathrm{d}\sigma
  - \int_{S_1 \cup S_2 \cup S_3} \delta u'_{N+1} \partial_{\nu} u_0 \,\mathrm{d}\sigma \\
&\quad + \int_{S_1' \cup S_2' \cup S_3'} u'_N \partial_\nu u_0 \,\mathrm{d}\sigma
  - \int_{S_2' \cup S_3'} \partial_\nu u'_N u_0 \,\mathrm{d}\sigma
  + \int_{S_2 \cup S_3} \partial_\nu(\delta u'_{N+1}) u_0 \,\mathrm{d}\sigma \\
&\quad + \int_{S_4} \left( u_0 \partial_{\nu} u' - u' \partial_{\nu} u_0 \right) \,\mathrm{d}\sigma .
\end{aligned}
\end{equation}

\noindent
\textbf{Three-dimensional sound-soft case} ($u = 0$ on $S_1$):
\begin{equation}\label{eq 3dss}
\begin{aligned}
&\int_{\tilde{S_1} \cup \tilde{S_2} \cup \tilde{S_3}} u'_N \partial_{\nu} u_0 \,\mathrm{d}\sigma
  - \int_{\tilde{S_2} \cup \tilde{S_3}} \partial_\nu u'_N u_0 \,\mathrm{d}\sigma \\
&= \int_{S_1} (u' - u)\partial_\nu u_0 \,\mathrm{d}\sigma
  - \int_{S_1} u_0 \partial_\nu u' \,\mathrm{d}\sigma
  - \int_{S_1 \cup S_2 \cup S_3} \delta u'_{N+1} \partial_{\nu} u_0 \,\mathrm{d}\sigma \\
&\quad + \int_{S_1' \cup S_2' \cup S_3'} u'_N \partial_\nu u_0 \,\mathrm{d}\sigma
  - \int_{S_2' \cup S_3'} \partial_\nu u'_N u_0 \,\mathrm{d}\sigma
  + \int_{S_2 \cup S_3} \partial_\nu(\delta u'_{N+1}) u_0 \,\mathrm{d}\sigma \\
&\quad + \int_{S_4} \left( u_0 \partial_{\nu} u' - u' \partial_{\nu} u_0 \right) \,\mathrm{d}\sigma .
\end{aligned}
\end{equation}

\noindent
\textbf{Two-dimensional sound-hard case} ($\partial_\nu u = 0$ on $I_1$):
\begin{equation}\label{eq 2dsh}
\begin{aligned}
&\int_{\tilde{I_1} \cup \tilde{I_2}} u'_N \partial_{\nu} u_{0} \,\mathrm{d}\sigma
  - \int_{\tilde{I_2}} \partial_\nu u'_N u_0 \,\mathrm{d}\sigma \\
&= \int_{I_1} u_{0} \partial_{\nu}(u^{\prime}-u) \,\mathrm{d}\sigma
  - \int_{I_1 \cup I_2} \delta u'_{N+1} \partial_{\nu} u_{0} \,\mathrm{d}\sigma \\
&\quad + \int_{I_1' \cup I_2'} u'_{N} \partial_\nu u_0 \,\mathrm{d}\sigma
  - \int_{I_2'} \partial_\nu u'_N u_0 \,\mathrm{d}\sigma
  + \int_{I_2} \partial_\nu(\delta u'_{N+1}) u_0 \,\mathrm{d}\sigma \\
&\quad + \int_{I_3} \left( u_{0} \partial_{\nu} u^{\prime} - u^{\prime} \partial_{\nu} u_{0} \right) \,\mathrm{d}\sigma .
\end{aligned}
\end{equation}

\noindent
\textbf{Two-dimensional sound-soft case} ($u = 0$ on $I_1$):
\begin{equation}\label{eq 2dss}
\begin{aligned}
&\int_{\tilde{I_1} \cup \tilde{I_2}} u'_N \partial_{\nu} u_{0} \,\mathrm{d}\sigma
  - \int_{\tilde{I_2}} \partial_\nu u'_N u_0 \,\mathrm{d}\sigma \\
&= \int_{I_1} (u' - u)\partial_\nu u_0 \,\mathrm{d}\sigma
  - \int_{I_1} u_0 \partial_\nu u' \,\mathrm{d}\sigma
  - \int_{I_1 \cup I_2} \delta u'_{N+1} \partial_{\nu} u_{0} \,\mathrm{d}\sigma \\
&\quad + \int_{I_1' \cup I_2'} u'_{N} \partial_\nu u_0 \,\mathrm{d}\sigma
  - \int_{I_2'} \partial_\nu u'_N u_0 \,\mathrm{d}\sigma
  + \int_{I_2} \partial_\nu(\delta u'_{N+1}) u_0 \,\mathrm{d}\sigma \\
&\quad + \int_{I_3} \left( u_{0} \partial_{\nu} u^{\prime} - u^{\prime} \partial_{\nu} u_{0} \right) \,\mathrm{d}\sigma .
\end{aligned}
\end{equation}

The identities \eqref{eq 3dsh}--\eqref{eq 2dss} are the localized ATD integral identities for the sound-soft and sound-hard cases.
The propagation-of-smallness argument is the same as in Section~\ref{sec 4}, and the subsequent ATD extraction follows the same scheme as in Sections~\ref{sec 2D} and~\ref{sec proof}.
The difference lies in the form of the first extracted term.

In the two-dimensional case, a direct calculation using the same approach as in Proposition~\ref{prop lower2D} shows that the first extracted term is determined by
\begin{equation}\label{eq:2d-ss-first}
a_N+b_N=0
\end{equation}
in the sound-soft case, and by
\begin{equation}\label{eq:2d-sh-first}
a_N-b_N=0
\end{equation}
in the sound-hard case.
If the corresponding lowest-order term does not vanish, then the lower-bound argument closes at the first step.
If it vanishes, one continues the same ATD extraction to the next lowest-order term in $\tau$, and then recursively to all higher orders.
The three-dimensional sound-soft and sound-hard cases are treated in the same way, starting from \eqref{eq 3dss} and \eqref{eq 3dsh}, respectively.

Therefore, the same quantitative relation is obtained in the classical sound-soft and sound-hard regimes.
Combined with the bounded-impedance analysis in Sections~\ref{sec 2D} and~\ref{sec proof}, this completes the proof of Theorem~\ref{th relation} for all generalized impedance regimes considered in this paper.

\section*{Acknowledgments}

The work of H. Diao is supported by National Natural Science Foundation of China  (No. 12371422),  2025 National Foreign Experts Program (Grant No. D20250157), and the Fundamental Research Funds for the Central Universities, JLU. 
The work of H. Liu was supported by NSFC/RGC Joint Research Scheme, N\_CityU101/21, ANR/RGC Joint Research Scheme, A-CityU203/19, and the Hong Kong RGC General Research Funds (projects 11311122,  11304224, and 11303125).

\bibliographystyle{siam}
\bibliography{ref}
	
\end{document}